\documentclass[11pt]{article} 
\usepackage[latin1]{inputenc}
\usepackage{amsmath}
\usepackage{mathtools}
\usepackage{amsfonts}
\usepackage{amssymb}
\usepackage{amsthm}
\usepackage{amsxtra}
\usepackage{enumerate}
\usepackage{authblk}

\usepackage{titlesec}
\titleformat{\subsection}[hang]{\normalfont\bfseries}{\thesubsection}{1em}{}

\titlespacing\section{0pt}{3.5ex plus 0.5ex minus .2ex}{0.3ex plus .2ex}
\titlespacing\subsection{0pt}{2.5ex plus 0.5ex minus .2ex}{0.3ex plus .2ex}
\titlespacing\subsubsection{0pt}{2.5ex plus 0.5ex minus .2ex}{0.3ex plus .2ex}

\usepackage{multind} 
\usepackage{xr-hyper}

\usepackage{fullpage}
\usepackage{url} 
\usepackage[all]{xy}
\usepackage{comment}

\usepackage{color} 

\makeindex{notation-ky}

\usepackage{hyperref} 

\numberwithin{equation}{subsection}

\newtheorem{theorem}[equation]{Theorem}
\newtheorem*{theorem*}{Theorem}
\newtheorem{lemma}[equation]{Lemma}
\newtheorem*{lemma*}{Lemma}
\newtheorem{property}[equation]{Property}
\newtheorem{proposition}[equation]{Proposition}
\newtheorem*{proposition*}{Proposition}
\newtheorem{corollary}[equation]{Corollary}
\newtheorem*{corollary*}{Corollary}

\theoremstyle{definition}
\newtheorem{definition}[equation]{Definition}
\newtheorem*{definition*}{Definition}
\newtheorem{notation}[equation]{Notation}
\newtheorem*{notation*}{Notation}

\newtheorem*{choice*}{Choice}
\newtheorem{remark}[equation]{Remark}
\newtheorem*{remark*}{Remark}

\newtheorem*{axiom*}{Axiom}

\allowdisplaybreaks[1]

\newcommand{\Roots}[2]{{\Phi(#1, #2)}}

\renewcommand{\restriction}{|}

\newcommand{\abs}[1]{\left\lvert#1\right\rvert}

\newcommand{\Pyx}{{P^y_x}}
\newcommand{\Vyx}{{V^y_x}}
\newcommand{\Vyxp}{{V^y_{x,+}}}
\newcommand{\Vxy}{{V^x_y}}
\newcommand{\Vxyp}{{V^x_{y,+}}}

\newcommand{\Nsym}{N\textup{-sym}}
\newcommand{\Nasym}{N\textup{-asym}}

\DeclareMathOperator{\End}{End}

\DeclareMathOperator{\ad}{ad}
\DeclareMathOperator{\ind}{ind}

\DeclareMathOperator{\red}{red}
\DeclareMathOperator{\Mod}{Mod-}

\DeclareMathOperator{\id}{id}
\DeclareMathOperator{\ord}{ord}

\DeclareMathOperator{\aff}{aff}
\DeclareMathOperator{\Sp}{Sp}
\DeclareMathOperator{\sgn}{sgn}
\DeclareMathOperator{\Gal}{Gal}
\DeclareMathOperator{\scn}{sc}

\DeclareMathOperator{\nt}{nt}
\DeclareMathOperator{\ort}{O}
\DeclareMathOperator{\Lie}{Lie}
\DeclareMathOperator{\ab}{ab}
\DeclareMathOperator{\der}{der}
\DeclareMathOperator{\sn}{sn}
\DeclareMathOperator{\GL}{GL}
\DeclareMathOperator{\SL}{SL}

\DeclareMathOperator{\Symp}{Sp}

\DeclareMathOperator{\HW}{HW}

\newcommand{\unr}{{\textup{unr}}}
\newcommand{\ram}{{\textup{ram}}}
\newcommand{\gen}{{\textup{gen}}}
\newcommand{\dashgen}{{\textup{-gen}}}

\newcommand{\sym}{{\textup{sym}}}
\newcommand{\asym}{{\textup{asym}}}

\newcommand{\Krel}{{\mathcal{K}\textup{-rel}}}
\newcommand{\Kzrel}{{\mathcal{K}^0\textup{-rel}}}

\newcommand{\isoarrow}{\stackrel{\sim}{\longrightarrow}}
\newcommand{\isoarrowleft}{\stackrel{\sim}{\longleftarrow}}
\newcommand{\Nheart}{N(\rho_{M})^{\heartsuit}_{[x_{0}]_{M}}}
\newcommand{\Nzeroheart}{N(\rho_{M^{0}})^{\heartsuit}_{[x_{0}]_{M^{0}}}}

\newcommand{\Wzeroheart}{W(\rho_{M^{0}})^{\heartsuit}_{[x_{0}]_{M^{0}}}}

\newcommand{\muTzero}{\mu^{\cT^0}}

\newcommand{\Waffz}{W(\rho_{M^0})_{\mathrm{aff}}}

\newcommand{\Wzeroz}{\Omega(\rho_{M^0})}
\newcommand{\Coeff}{\cC}   
\newcommand{\Isom}{\cI} 
 
\DeclareMathOperator{\Rep}{Rep}
\newcommand{\IEC}{\mathfrak{I}} 

\newcommand{\bC}{\mathbb{C}}

\newcommand{\bF}{\mathbb{F}}

\newcommand{\bQ}{\mathbb{Q}}
\newcommand{\bR}{\mathbb{R}}

\newcommand{\bZ}{\mathbb{Z}}

\newcommand{\cA}{\mathcal{A}}
\newcommand{\cB}{\mathcal{B}}
\newcommand{\cC}{\mathcal{C}}

\newcommand{\cH}{\mathcal{H}}
\newcommand{\cI}{\mathcal{I}}
\newcommand{\cK}{\mathcal{K}}

\newcommand{\cO}{\mathcal{O}}

\newcommand{\cT}{\mathcal{T}}
\newcommand{\cU}{\mathcal{U}}

\newcommand{\ff}{\mathfrak{f}}

\newcommand{\fp}{\mathfrak{p}}

\newcommand{\fr}{\mathfrak{r}}
\newcommand{\fs}{{\mathfrak{s}}}
\newcommand{\fsz}{{\mathfrak{s}_0}}

\newcommand{\fS}{{\mathfrak{S}}}
\newcommand{\fSz}{{\mathfrak{S}_0}}

\newcommand{\spacingatend}[1]{}

\begin{document}
\externaldocument[HAI-][nocite]{Adler--Fintzen--Mishra--Ohara_Structure_of_Hecke_algebras_arising_from_types}

\author{
	Jeffrey D. Adler,
	Jessica Fintzen,
	Manish Mishra,
	and
	Kazuma Ohara
}
\AtEndDocument{\bigskip{\footnotesize%
	\par
       \textsc{%
	Department of Mathematics and Statistics,
	American University,
	4400 Massachusetts Ave NW,
	Washington, DC 20016-8050, USA} \par
       \textit{E-mail address}: \texttt{jadler@american.edu}
}}
\AtEndDocument{\bigskip{\footnotesize%
	\par
       \textsc{
	Universit\"at Bonn,
           Mathematisches Institut,
           Endenicher Allee 60,
           53115 Bonn,
           Germany } \par
       \textit{E-mail address}: \texttt{fintzen@math.uni-bonn.de}
}}
\AtEndDocument{\bigskip{\footnotesize%
	\par
	\textsc{%
	Department of Mathematics,
        IISER Pune,
        Dr.\ Homi Bhabha Road,
        Pune, Maharashtra 411008, India}
\par
       \textit{E-mail address}: \texttt{manish@iiserpune.ac.in}  
}}
\AtEndDocument{\bigskip{\footnotesize%
	\par
	\textsc{%
	Graduate School of Mathematical Sciences, The University of Tokyo, 
	3-8-1 Komaba, Meguroku, Tokyo 153-8914, Japan}
	\par
       \textit{E-mail address}: \texttt{kohara@ms.u-tokyo.ac.jp}
}}

\title{Reduction to depth zero for tame $p$-adic groups via Hecke algebra isomorphisms}
\date{}
\maketitle
\begin{abstract}
Let $F$ be a nonarchimedean local field of residual characteristic $p$. Let $G$ denote a connected reductive group over $F$ that splits over a tamely ramified extension of $F$. Let $(K ,\rho)$ be a type as constructed by Kim and Yu. We show that there exists a twisted Levi subgroup $G^0 \subset G$ and a type $(K^0, \rho^0)$ for $G^0$ such that the corresponding Hecke algebras $\mathcal{H}(G(F), (K, \rho))$ and $\mathcal{H}(G^0(F), (K^0, \rho^0))$ are isomorphic. 
If $p$ does not divide the order of the absolute Weyl group of $G$, then every Bernstein block is equivalent to modules over such a Hecke algebra. Hence, under this assumption on $p$, our result implies that every Bernstein block is equivalent to a depth-zero Bernstein block. This allows one to reduce many problems about (the category of) smooth, complex representations of $p$-adic groups to analogous problems about (the category of) depth-zero representations.

Our isomorphism of Hecke algebras is very explicit and also includes an explicit description of the Hecke algebras as semi-direct products of an affine Hecke with a twisted group algebra.
Moreover, we work with arbitrary algebraically closed fields of characteristic different from $p$ as our coefficient field.

This paper relies on a prior axiomatic result about the structure of Hecke algebras by the same authors and a key ingredient consists of extending the quadratic character of Fintzen--Kaletha--Spice to the support of the Hecke algebra, which might be of independent interest.
\end{abstract}

{
	\renewcommand{\thefootnote}{}  
	\footnotetext{MSC2020: Primary 22E50, 22E35, 20C08, 20C20. Secondary 22E35}
	\footnotetext{Keywords: $p$-adic groups, smooth representations, Hecke algebras, types, Bernstein blocks, mod-$\ell$ coefficients}

\footnotetext{The first-named author
was partially supported by
the American University College of Arts and Sciences
Faculty Research Fund.}
	\footnotetext{The second-named author was partially supported by NSF Grants DMS-2055230 and DMS-2044643, a Royal Society University Research Fellowship, a Sloan Research Fellowship and the European Research Council (ERC) under the European Union's Horizon 2020 research and innovation programme (grant agreement no. 950326).}
\footnotetext{
The third-named author was partially supported by
a Fulbright-Nehru Academic and Professional Excellence Fellowship and
a SERB Core Research Grant (CRG/2022/000415).} 
\footnotetext{
The fourth-named author is supported by the FMSP program at Graduate School of Mathematical Sciences, the University of Tokyo and JSPS KAKENHI Grant number JP22J22712.
}
}

\setcounter{tocdepth}{2}

\newpage
\tableofcontents

\numberwithin{equation}{section}
\section{Introduction} 
The category of all smooth, complex representations of a $p$-adic group decomposes as a product of
indecomposable
full subcategories, called Bernstein blocks, each of which is equivalent to modules over a Hecke algebra under minor tameness assumptions. Therefore knowing the explicit structure of these Hecke algebras and their modules yields an understanding of the category of smooth representations. While these Hecke algebras are known for $\GL_n$ and in that case played an important role in the representation theory, comparatively little has been known for general reductive groups. An exception are the Bernstein blocks that consist only of depth-zero representations. For those blocks, a description of the attached Hecke algebras has essentially (for details see \cite{HAI}) been known for over 30 years, thanks to work of Morris (\cite{Morris}). In this paper, we show
that the above Hecke algebra attached to arbitrary Bernstein blocks is isomorphic to a depth-zero Hecke algebra under the above minor tameness assumptions, which we assume to hold for the next few sentences. Therefore we now have an explicit description of all those Hecke algebras. 
Moreover, as a direct consequence of the Hecke algebra isomorphism, we obtain 
an equivalence between arbitrary Bernstein blocks and depth-zero Bernstein blocks, which allows one to reduce a plethora of problems in the representation theory of $p$-adic groups and beyond, including in the Langlands program, to the depth-zero setting, where solutions are often easier to obtain or are already known.

We obtain the Hecke algebra isomorphism between arbitrary and depth-zero Bernstein blocks by verifying the relevant axioms of \cite{HAI} to allow us to apply the general Hecke algebras isomorphism of \cite{HAI}. This involves as a key step the extension of the quadratic character of \cite{FKS} to a group of double-coset representatives for the whole support of the Hecke algebra, a result that might be of independent interest to mathematicians in this area. It also involves the study and extension of Heisenberg--Weil representations that is expected to have applications beyond this paper.

The Hecke algebra isomorphisms as well as the extensions of the quadratic character and the Heisenberg--Weil representations are all described in an explicit way, making them suitable for a larger range of future applications.

To explain our results in more detail, we fix a non-archimedean local field $F$ with residual characteristic $p$ and denote by $G$ a connected reductive group over $F$ that splits over a tamely ramified extension of $F$. In this introduction we consider smooth representations with complex coefficients for simplicity, but all results about Hecke algebra isomorphisms are also proven with $\Coeff$-coefficients for an arbitrary algebraically closed field $\Coeff$ of characteristic different from $p$. In the complex case,
by Bernstein (\cite{MR771671}), the category $\Rep(G(F))$ of smooth representations of $G(F)$ decomposes into a product of indecomposable full subcategories, $\Rep^\fs(G(F))$, that are called \emph{Bernstein blocks} and that are indexed by the set of inertial equivalence classes $\IEC(G)$, i.e., equivalence classes $[L, \sigma]$ of pairs $(L, \sigma)$ consisting of a Levi subgroup $L$ of (a parabolic subgroup of) $G$ and an irreducible supercuspidal representation $\sigma$ of $L(F)$:
\[
\Rep(G(F))=\prod_{\fs\in\IEC(G)}\Rep^\fs(G(F)).
\]
A pair $(K, \rho)$ consisting of a compact, open subgroup $K \subset G(F)$ and an irreducible smooth representation $(\rho, V_\rho)$ of $K$ is called an $\fs$-type for $\fs \in \IEC(G)$ if for every smooth representation $\pi$ of $G(F)$, the representation $\pi$ belongs to $\Rep^\fs(G(F))$ if and only if every irreducible subquotient of $\pi$
contains $\rho$ upon restriction to $K$. If $(K, \rho)$ is an $\fs$-type, then we have an equivalence of categories between $\Rep^{\fs}(G(F))$ and the category of right unital $\cH(G(F), (K, \rho))$-modules:
\[
\Rep^{\fs}(G(F)) \simeq \Mod \cH(G(F), (K, \rho)),
\]
where  $\cH(G(F),(K, \rho))$ denotes the Hecke algebra attached to $(K, \rho)$, i.e., as a vector space the collection
of all compactly supported,
$\End_{\bC}(V_\rho)$-valued
functions on $G(F)$ that transform
on the left and right under $K$  
by $\rho$.
The algebra structure on $\cH(G(F),(K, \rho))$ is given by convolution,
see \cite[Section \ref{HAI-Hecke algebras and endomorphism algebras}]{HAI}
for details.

Building upon the construction of supercuspidal representations by Yu (\cite{Yu}) and using the theory of covers introduced by Bushnell and Kutzko (\cite{BK-types}), Kim and Yu (\cite{Kim-Yu, MR4357723}) provide a construction of types. This construction yields types for every Bernstein block if $p$ does not divide the order $\abs{W}$ of the absolute Weyl group $W$ of $G$ by \cite{Fi-exhaustion}. Thus, understanding the structure of the corresponding Hecke algebras $\cH(G(F), (K, \rho))$ and its category of modules leads to an understanding of the whole category of smooth representations of $G(F)$ if $p \nmid \abs{W}$.
 The input for the construction of Kim and Yu includes a twisted Levi subgroup $G^0 \subset G$, a compact, open subgroup $K^0 \subset G^0$ that contains a parahoric subgroup of $G^0(F)$, a depth-zero representation $\rho^0_{KY}$ of $K^0$, and a positive-depth (or trivial) character of $G^0(F)$, satisfying various conditions, see Definition \ref{definitionofGdatum} for details. The pair $(K^0, \rho^0_{KY})$ is a type for a Bernstein block of $G^0$ that consists of depth-zero representations. To a given input, Kim and Yu then attach a compact, open subgroup $K \subset G(F)$, which satisfies $K \cap G^0(F)=K^0$, and a representation $\rho = \rho^0_{KY} \otimes \kappa^{\nt}$, where $\rho^0_{KY}$ is viewed as a representation of $K$ via an appropriate inflation and $\kappa^{\nt}$ is an irreducible smooth representation constructed from the positive-depth character via the theory of Heisenberg--Weil representations. (We caution the reader that what we denote by $\rho$ and $\rho^0_{KY}$ here for simplicity is denoted by $\rho_{x_0}^{\nt}$ and $\rho_{x_0}^0$ in the main part of the paper.) 
 
 We set $\rho^0=\rho^0_{KY} \otimes \epsilon$, where $\epsilon$ (denoted by $\epsilon^{\overrightarrow{G}}_x$ in Section \ref{subsec:construction}) is a quadratic character introduced by \cite{FKS}. Then our main result, Theorem \ref{heckealgebraisomforKim-Yutype}, is the construction of an explicit, support-preserving algebra isomorphism
 \begin{equation}
 \label{eq:intromain}
 \cH(G(F), (K, \rho)) \stackrel{\sim}{\longrightarrow}  \cH(G^0(F),(K^0, \rho^0)).
 \end{equation}
We provide an example in Appendix \ref{subsec:quadratictwistisnecessary} that shows that twisting $\rho^0_{KY}$ by $\epsilon$ is necessary, i.e., the above isomorphism would not always hold if we replaced $\rho^0$ by $\rho^0_{KY}$.
As a direct corollary, we obtain that if $p \nmid \abs{W}$, then an arbitrary Bernstein block is equivalent to a depth-zero Bernstein block.
 We also remind the reader that we have an explicit description of the structure of the depth-zero Hecke algebra
  \begin{equation}
   \label{eq:introdepthzero}
  \cH(G^0(F),(K^0, \rho^0)) \simeq  \bC[\Wzeroz, \muTzero] \ltimes \cH(\Waffz, q) 
 \end{equation}
by \cite[Theorem \ref{HAI-theoremstructureofheckefordepthzero}]{HAI}, which is a generalization of prior work of Morris (\cite{Morris}), and hence we obtain an explicit description of the structure of $\cH(G(F), (K, \rho)) $ by combining \eqref{eq:intromain} and \eqref{eq:introdepthzero}. We refer the reader to the introduction of \cite{HAI} or to Theorem \ref{theoremstructureofheckeforKimYu} for the details of the notation on the right-hand side of \eqref{eq:introdepthzero}.

The isomorphism \eqref{eq:intromain} is obtained by applying the general result \cite[Theorem \ref{HAI-thm:isomorphismtodepthzero}]{HAI} about Hecke algebra isomorphisms to the setting of this paper, which requires checking all the relevant axioms of \cite{HAI}. Most of the axioms follow easily from the construction of Kim and Yu, except those that deal with the (extension of the) twisted Heisenberg--Weil part $\kappa \coloneqq \kappa^{\nt} \otimes \epsilon$ in the construction. In order to apply  \cite[Theorem \ref{HAI-thm:isomorphismtodepthzero}]{HAI} we need to construct a whole compatible family of twisted Heisenberg--Weil representations, see Corollary \ref{corollaryofpropositioninductiontwistedsequencever} for the compatibility with respect to compact induction that we require. The existence of such a compatible family relies  crucially on the quadratic twist $\epsilon$ of \cite{FKS}. Moreover, we need to show that the twisted Heisenberg--Weil representation extends to a group of double-coset representatives of the support of the Hecke algebra $\cH(G(F), (K, \rho))$. These results are expected to be of independent interest.

To provide a few more details about the last result, we fix an $[M, \sigma]$-type $(K, \rho)$ as constructed by Kim and Yu for some $[M, \sigma] \in \IEC(G)$. Then $(K^0, \rho^0)$ is an $[M^0, \sigma^0]$-type for a Levi subgroup $M^0 \subseteq G^0$  and a depth-zero supercuspidal representation $\sigma^0$ of $M^0(F)$. If the representative $M$ is chosen with care, then $M^0=M \cap G^0$. Moreover, $K^0$ is the stabilizer $G^0(F)_{x}$ in $G^0(F)$ of a point $x$ in the enlarged Bruhat--Tits building of $M^0$ that we view as a subset of the enlarged Bruhat--Tits building of $G^0$. We denote by $N_{G^{0}}(M^{0})(F)_{[x]_{M^{0}}}$ the subgroup of the $F$-points of the normalizer of $M^0$ in $G^0$ that preserves the image $[x]_{M^{0}}$ of $x$ in the reduced Bruhat--Tits building of $M^0$. This group normalizes $K_M\coloneqq K \cap M(F)$ and contains a set of double-coset representatives for the support of the Hecke algebras $\cH(G^0(F), (K^0, \rho^0))$ and $\cH(G(F), (K, \rho))$. 
A key result in this paper, see Definition \ref{definitionofkappatilde} and Proposition \ref{propproofofaxiomaboutKM0vsKM}, is the construction of a representation $\widetilde \kappa_M$ of $N_{G^{0}}(M^{0})(F)_{[x]_{M^{0}}} \cdot K_M$ that restricted to $K_M$ agrees with the twisted Heisenberg--Weil representation $\kappa\restriction_{K_M}$ (which is defined in detail on page \pageref{kappaM} in Section \ref{subsec:construction}). This construction involves, in particular, the extension of the quadratic character $\epsilon$ to a character $\widetilde \epsilon$ of $N_{G^{0}}(M^{0})(F)_{[x]_{M^{0}}} \cdot K_M$, which is achieved via Definition \ref{definitiontildeepsilon} and Theorem \ref{twistextension}. The character $\widetilde \epsilon$ has its image contained in the fourth roots of unity but cannot always be chosen to be quadratic. 

We briefly sketch some of the steps in the construction of $\widetilde \epsilon$. The character $\epsilon$ was defined as a product of three characters in \cite{FKS} and we take as our starting point their reinterpretation as the composition of a map to the orthogonal group of a quadratic space over the residue field $\ff$ of $F$ with the spinor norm. For one of the three characters we can reinterpret the quadratic space in a way that we can extend the action of  $K_M$ on it to an action of $N_{G^{0}}(M^{0})(F)_{[x]_{M^{0}}} \cdot K_M$ that preserves the quadratic form, which is done in Section \ref{subsection:character2}. For the other two factors we consider and extend only their product, combining the two underlying quadratic spaces into a new larger quadratic space $V_x$. Since the group $N_{G^{0}}(M^{0})(F)_{[x]_{M^{0}}}$ does not fix the point $x$, the construction of an action of $N_{G^{0}}(M^{0})(F)_{[x]_{M^{0}}}$ on $V_x$ involves a careful construction of compatible isomorphisms between $V_x$ and $V_y$ for $y$ in the $N_{G^{0}}(M^{0})(F)_{[x]_{M^{0}}}$-orbit of $x$, which is achieved in Section \ref{subsection:constructionofon}. Unfortunately, the resulting action of $N_{G^{0}}(M^{0})(F)_{[x]_{M^{0}}}$ on $V_x$ does not always preserve the quadratic form on $V_x$ used for the construction of the spinor norm and also only yields a morphism of $V_x$ viewed as an $\ff[\sqrt{-1}]$-vector space, not as an $\ff$-vector space in general. At the same time, the image of $K_M$ acting on $V_x$ is not the full orthogonal group. We construct a new group that contains the image of $K_M$ in the orthogonal group of $V_x$ and also the image of  $N_{G^{0}}(M^{0})(F)_{[x]_{M^{0}}}$ in $\GL(V_x)(\ff[\sqrt{-1}])$ and succeed in extending the restriction of the spinor norm to this larger group, see Section \ref{subsection:character13}. This then yields the desired extension $\widetilde \epsilon$ with values in the group of fourth roots of unity. If $\ff$ contains a square-root of $-1$, i.e. $\ff[\sqrt{-1}]=\ff$, then the extension $\widetilde \epsilon$ is a quadratic character.

We point out that the Hecke algebras associated with types have been already previously heavily studied in special cases, see \cite[\S\ref{HAI-subsec:history}]{HAI}
for more details.
The reduction to depth-zero isomorphism \eqref{eq:intromain} was already achieved for $\GL_n$ by Bushnell and Kutzko (\cite{MR1204652}) and is an important tool for a variety of applications, including its recent use in the construction of a part of a categorical local Langlands correspondence (\cite{BZCHN}). Moreover, for more general reductive groups, our isomorphism \ref{eq:intromain} was previously obtained by Roche (\cite{MR1621409}) in the case where $F$ has characteristic zero, $G$ is a split reductive group, and the above introduced Levi subgroup $M$ is a maximal split torus;
by Adler--Mishra (\cite{AdMi21})
in the more general situation where the connected part of the center of $G^0$
is split modulo the center of $G$
and the underlying supercuspidal representation of a Levi subgroup is regular,
and by Ohara (\cite{2021arXiv210101873O}) in the case of supercuspidal blocks, i.e., the case where $M=G$.  
Obtaining a general result \eqref{eq:intromain} beyond special cases has been an open problem for more than 20 years,
and indeed it can be thought of as a sharper version of
Conjectures 0.2 and 17.7 in \cite{Yu}.
However, the literature contains a discussion of a particular Hecke algebra
(see \cite[\S11.8]{GoldbergRoche-Hecke})
that appears to be too complicated to agree with the seemingly simpler Hecke algebras of depth-zero Bernstein blocks. 
This led to a large part of the mathematical community working in this area believing 
that \eqref{eq:intromain} cannot be true
in full generality (which we however now prove to be the case in this paper),
but still wishing for weaker form of it to be true, yet not knowing what form this would be.
We will address these concerns in a separate paper \cite{SL8}
explaining why the observations in \cite{GoldbergRoche-Hecke}
do not lead to a contradiction to our results.

\subsection*{Structure of the paper and guidance for the reader} 
Section \ref{sec:extensionofthetwist} concerns the extension $\widetilde \epsilon$ of the quadratic character $\epsilon$. 
We have summarized the result about the extension $\widetilde \epsilon$ of $\epsilon$ in the first page of Section \ref{sec:extensionofthetwist}. Thus, a reader mostly interested in the statement and not the details of the construction and proofs (even though they mark the core of our paper) is welcome to read this one page and skip the subsections of \ref{sec:extensionofthetwist} on a first reading. 

In Section \ref{sec:twistedweil} we introduce families of Heisenberg--Weil representations and prove various compatibility and extension results for those representations. Sections \ref{sec:extensionofthetwist} and \ref{sec:twistedweil} are independent of each other and can therefore be read in any order.

In Section \ref{subsec:construction} we recall the construction of Kim and Yu, but we allow more general coefficients and include a twist by the quadratic character of \cite{FKS} in the construction. The remainder of Section \ref{Section-KimYutypes} is then concerned with verifying all the necessary properties in order to obtain the isomorphism \eqref{eq:intromain}. While the proofs crucially rely on the results of Sections \ref{sec:extensionofthetwist} and \ref{sec:twistedweil}, we have written Section \ref{Section-KimYutypes} in way that a reader who is already a bit familiar with the construction of Yu (\cite{Yu}) and who is willing to simply believe our results from Sections \ref{sec:extensionofthetwist} and \ref{sec:twistedweil} on a first reading, can start by reading Section \ref{Section-KimYutypes}.

Appendix \ref{appendixYufortwopoints} spells out some technical details that, while too similar to work of J.-K. Yu to be original, are also too elaborate to be left to the reader.

In Appendix \ref{subsec:quadratictwistisnecessary} we provide an explicit example that shows that our main result \eqref{eq:intromain} would not be true in general if we omitted the twist by the quadratic character $\epsilon$.

\subsection*{Acknowledgments}
Various subsets of the authors benefited from the hospitality
of American University, Duke University, the University of Bonn, 
the  Hausdorff Research Institute for Mathematics in Bonn,
the Max Planck Institute for Mathematics in Bonn, the Indian Institute for Science Education and Research (Pune), and the University of Michigan.
The authors also thank Alan Roche, 
Dan Ciubotaru,
Tasho Kaletha, and Ju-Lee Kim for discussions related to this paper. 
This project was started independently by three different subgroups of the authors, and the fourth-named author thanks his supervisor Noriyuki Abe for his enormous support and helpful advice.
The fourth-named author is also grateful to Tasho Kaletha for the discussions during his stay at the University of Michigan in 2023.

\subsection*{Notation}
Let $F$ be a non-archimedean local field endowed with a discrete valuation $\ord$ on $F^{\times}$ with value group $\mathbb{Z}$.
We fix a separable closure $\overline{F}$ of $F$.
When we refer to a separable extension $E/F$, we always assume that $E \subset \overline{F}$.
For any such extension $E$,
we denote the unique extension of $\ord$ to $E^\times$ also by $\ord$.
We write
$\cO_{E}$ for the ring of integers of $E$
and $\fp_{E}$ for the maximal ideal of $\cO_E$.
We also write $\ff$ for the residue field of $F$,
and let $p$ denote the characteristic of $\ff$.

If $L$ is any field, and $L'$ is any Galois field extension of $L$,
then we
let $\Gal(L'/L)$ denote the corresponding Galois group.

Let $\bR$ and $\bC$ denote the fields of real and complex numbers,
respectively.
Let $\Coeff$
\index{notation-ky}{C@$\Coeff$ (coefficient field)}%
denote an algebraically closed field of characteristic
\index{notation-ky}{l@$\ell$ (characteristic of $\Coeff$)}%
$\ell\neq p$.
Except when otherwise indicated,
all representations below are on vector spaces over $\Coeff$.
We fix an additive character $\Psi \colon F \longrightarrow \Coeff^{\times}$
\index{notation-ky}{Psi@$\Psi$ (additive character $F\rightarrow \Coeff^\times$)}%
that is trivial on $\fp_{F}$ and non-trivial on $\cO_{F}$.

We also adopt the fairly standard notation that was
defined in
\cite[\S\ref{HAI-subsec:notation}--\ref{HAI-Hecke algebras and endomorphism algebras}]{HAI} for the following symbols:
$Z(G)$, $A_{G}$, 
\index{notation-ky}{A G@$A_{G}$}%
$N_G(M)$ and $Z_G(M)$, 
$\ind_K^H(\rho)$, 
$I_{H}(\rho)$ and $N_{H}(\rho)$ for a representation $\rho$ of a subgroup of a group $H$, 
and the Hecke algebra $\cH(G(F), (K, \rho)) = \cH(G(F), \rho)$%
\index{notation-ky}{H(GFr@$\cH(G(F), \rho)$}.
As in \cite[\S\ref{HAI-subsec:notation}]{HAI}, by $\cU(M)$\index{notation-ky}{U M@$\cU(M)$} we denote the set of unipotent radicals of all parabolic
subgroups of $G$ with Levi factor $M$,

For a linear algebraic group $G$ over $F$ and an algebraic extension $E$ of $F$, let $G_{E}$ denote the base change of $G$ to $E$.
We denote by $X^{*}(G)$ the (algebraic) character group of $G_{\overline F}$ and $X_{*}(G)$ denote the cocharacter group of $G_{\overline F}$.
We also denote by $\Lie(G)$ the Lie algebra of $G$ and by $\Lie^{*}(G)$ the dual of $\Lie(G)$.
Let $\Lie^{*}(G)^{G}$ denote the subscheme of $\Lie^{*}(G)$ fixed by the coadjoint action of $G$ on $\Lie^{*}(G)$.

Suppose that $G$ is a connected reductive group defined over $F$.
For a torus $S$ of $G$, we denote by $\Phi(G, S)$ the non-zero weights of $S_{\overline F}$ acting on
the Lie algebra $\Lie(G)$ of $G$
equipped with the action of the absolute Galois group $\Gal(\overline F/F)$ of $F$.
In particular, if $S$ is a maximal split torus of $G$, the set $\Phi(G, S)$\index{notation-ky}{Phi G S@$\Phi(G, S)$} denotes the relative root system of $G$ with respect to $S$.
In this case, we also let $\Phi_{\aff}(G, S)$\index{notation-ky}{Phi aff G S@$\Phi_{\aff}(G, S)$} denote the relative affine root system associated to $(G, S)$ by the work of Bruhat and Tits \cite{MR327923}.

We denote by $\cB(G, E)$ the enlarged Bruhat--Tits building of $G_{E}$, and for a maximal torus $S$ of $G$ that splits over $E$, we let $\cA(G, S, E)$ denote the apartment of $S_{E}$ in $\cB(G, E)$.
We also write $\cB^{\red}(G, E)$ for the reduced building of $G_{E}$ and $\cA^{\red}(G, S, E)$ for the apartment of $S_{E}$ in $\cB^{\red}(G, E)$.
If $x \in \cB(G, F)$,
then we write $[x]_G$
\index{notation-ky}{000x@$[x]_G$}%
for the image of $x$ in the reduced building
$\cB^{\red}(G,F)$, and we may omit the subscript $G$ when it is clear
from the context.
For any abstract group $H$ that acts on $G(F)$, and thus on $\cB(G,F)$ and $\cB^{\red}(G,F)$,
we let $H_{x}$ and $H_{[x]_G}$ denote the stabilizers of $x$ and $[x]_G$ under the actions of $H$, respectively.
Let $\widetilde{\mathbb{R}}$ denote the set
\(
\{ r, r+ \mid r \in \mathbb{R} \} \cup \{\infty\}
\)
with the obvious ordering and the obvious addition operation (see \cite[6.4.1]{MR327923}).
Suppose that $E/F$ is a field
extension of finite ramification degree, and 
$S$ is a maximal torus of $G$ that splits over $E$.
For $x \in \cB(G, E)$, $\alpha \in \Phi(G, S)$, and $r \in \widetilde{\mathbb{R}}\smallsetminus \{ \infty\}$,
let $U_{\alpha}(E)_{x, r}$
\index{notation-ky}{U alpha E x r@$U_{\alpha}(E)_{x, r}$}%
denote the Moy--Prasad filtration subgroup of depth $r$ associated to $x$ of the root subgroup
\index{notation-ky}{U alpha E@$U_{\alpha}(E)$}%
$U_{\alpha}(E)$,
 and set  $U_{\alpha}(E)_{x, \infty} = \{1\}$.
Here, we use the valuation $\ord$ on $E^{\times}$.
For $x \in \cB(G, E)$ and $r \in \widetilde{\mathbb{R}}\smallsetminus \{ \infty\}$ with $r \ge 0$, we also let $G(E)_{x, r}$ be the Moy--Prasad filtration subgroup of $G(E)$ of depth $r$ associated to $x$ (see \cite{MR1253198, MR1371680}), and set $G(E)_{x, \infty} = \{1\}$.
We may abbreviate $G(E)_{x, r}/G(E)_{x, r+}$ to $G(E)_{x, r: r+}$.
We use the analogous notation for the Lie algebra, where $r$ is allowed to be any element of $\widetilde{\mathbb{R}}$.

For a representation $(\rho, V_{\rho})$ of a group $H$, we identify $\rho$ with its representation space $V_{\rho}$ by abuse of notation.
For any vector space $V$, let $\id_V$
\index{notation-ky}{id@$\id$}%
denote the identity map on $V$.

Throughout the paper we let $G$ be a connected reductive group defined over $F$.
We assume that $G$ splits over a tamely ramified field extension of $F$,
and that the residual characteristic $p$ of $F$ is not two.
We require these assumptions in order to apply 
the construction of Kim and Yu, which uses tameness to be able to work over the splitting field for various arguments, and requires that $p \neq 2$ 
whenever the construction involves
the use of nontrivial Heisenberg--Weil representations,
as constructed in 
Section \ref{sec:twistedweil}.


\numberwithin{equation}{subsection}
\section{Extension of the quadratic character of Fintzen--Kaletha--Spice}
\label{sec:extensionofthetwist}
In this section, we will prove that the quadratic character $\epsilon^{G/G'}_{x_0}$ defined in \cite[Section~4]{FKS} extends to a character of a larger group that contains coset representatives for the support of the Hecke algebra that we will study in Section \ref{Section-KimYutypes}. This result will be used to prove
Axiom~\ref{HAI-axiomaboutKM0vsKM}\eqref{HAI-axiomaboutKM0vsKMaboutanextension}
of \cite{HAI}
in the setting of this paper (see Proposition~\ref{propproofofaxiomaboutKM0vsKM} below), and might also be of independent interest.

To state the main result of this section more precisely,
let $G'$ be a twisted Levi subgroup of $G$ that splits over a tamely ramified extension of $F$, i.e., $G'$ is a subgroup of $G$ that becomes a Levi subgroup of a parabolic subgroup of $G$ after base change to an appropriate tamely ramified extension over which $G$ is split.
 Let $M'$ be a Levi subgroup of $G'$ and
let $M$ denote the centralizer of $A_{M'}$ in $G$.
Then $M$ is a Levi subgroup of $G$, and $M'$ is a twisted Levi subgroup of $M$. Moreover, we note that $A_M=A_{M'}$.

Note that this notation is different from that of \cite{FKS}, where the authors denote our $G'$ by $M$, and our $M$ and $M'$ do not appear.

We fix a positive real number $r > 0$.  
Let $x_{0}$ be a point of $\mathcal{B}(M', F)$ and $\{\iota\}$ \label{pageiota}be a commutative diagram
\[
\xymatrix{
\cB(G', F) \ar@{^{(}->}[r] \ar@{}[dr]|\circlearrowleft & \cB(G, F)
\\
\cB(M', F) \ar@{^{(}->}[u] \ar@{^{(}->}[r] & \cB(M, F) \ar@{^{(}->}[u]
}
\]
of admissible embeddings of buildings in the sense of \cite[\S~14.2]{KalethaPrasad} such that the embedding $\iota \colon \cB(M, F) \longrightarrow \cB(G, F)$ is $\frac{r}{2}$-generic relative to $x_{0}$ in the sense of \cite[Definition~3.2]{Kim-Yu}. Here, and from now on, we use these embeddings to identify points in buildings of subgroups with points in the building of the larger groups.

If there exists a character $\phi$ of $G'(F)$ that is $G$-generic of depth $r$ relative to $x_{0}$ in the sense of \cite[Definition~3.5.2]{Fintzen-IHES}, 
then we obtain from \cite[Lemma~4.1.2]{FKS} (using \cite[Remark~4.1.3]{FKS}) the quadratic character 
\[
\epsilon^{G/G'}_{x_0} \colon G'(F)_{[x_0]_{G}} \rightarrow \{ \pm 1 \} .
\]
\index{notation-ky}{epsilonGG'x@$\epsilon^{G/G'}_{x}$}%
In this section (see Definition \ref{definitiontildeepsilon} and Theorem \ref{twistextension}) we will prove that
the restriction of the character $\epsilon^{G/G'}_{x_{0}}$ to $N_{G'}(M')(F)_{[x_{0}]_{G}}$ extends to a character 
\[
\widetilde{\epsilon}^{G/G'}_{x_{0}} \colon N_{G'}(M')(F)_{[x_{0}]_{M'}} \longrightarrow \mu_{4},
\]
where $\mu_4 =\{\zeta \in \Coeff^\times \mid \zeta^4 = 1\}$. Instead of requiring the existence of a $G$-generic character of depth $r$, which will be the setting in which we use the result in Section~\ref{Section-KimYutypes},
we will prove this result in the slightly more general set up of \cite{FKS} that we recall below.

\subsection{Notation}
\label{subsec:notation-twistextension}
We introduce some additional notation, closely following \cite{FKS}, that will be used throughout the remainder of this section.
Let $G_{\ad}$ denote the adjoint group of $G$ and let $G'_{\ad}$, $M'_{\ad}$, and $M_{\ad}$ denote the images of $G'$, $M'$, and $M$ in $G_{\ad}$, respectively.
Let $G_{\scn}$ denote the simply connected cover of the derived subgroup of $G$
and $G'_{\scn}$ denote the preimage of $G'$ in $G_{\scn}$.
We also write $G'_{\scn, \ab} = G'_{\scn} / G'_{\scn, \der}$, where $G'_{\scn, \der}$ denotes the derived subgroup of $G'_{\scn}$.

We write $E^{\unr}$ for the maximal unramified extension of a separable field extension $E$  of $F$ (in $\overline{F}$)
and $I_{E}$ denotes the inertia subgroup of the absolute Galois group $\Gal(\overline{F}/E)$.
Let $\overline{\ff}$ denote the residue field of $F^{\unr}$,
which is a separable closure of $\ff$.

Given a non-degenerate quadratic form $\varphi$ on a vector space $V$
(over $\ff$, say),
we have the corresponding orthogonal group $\ort (V, \varphi)$.
Let 
$\sn \colon \ort(V, \varphi)(\ff) \longrightarrow \ff^\times / (\ff^\times)^2$
denote the spinor norm
as defined in defined in \cite[Chapter~9, Definition~3.4]{MR770063}
(see also \cite[Section~5.2]{FKS}).

Let $L = M'_{\ad}$ or $G'_{\ad}$. 
Then $Z(L)$ is a torus, and we denote by
$\Roots{G_{\ad}}{Z(L)}_{\asym}$ and $\Roots{G_{\ad}}{Z(L)}_{\sym, \ram}$ the sets of nonzero weights in $\Roots{G_{\ad}}{Z(L)}$ that are asymmetric and ramified symmetric in the sense of \cite[Section~2]{FKS}, respectively.
We also write
$
\Roots{G_{\ad}}{Z(L)}^{\sym, \ram} = \Roots{G_{\ad}}{Z(L)} \smallsetminus \Roots{G_{\ad}}{Z(L)}_{\sym, \ram}
$.
We define 
$\Roots{M_{\ad}}{Z(L)}_{\asym}$,
$\Roots{M_{\ad}}{Z(L)}_{\sym, \ram}$,
and
$\Roots{M_{\ad}}{Z(L)}^{\sym, \ram}$ analogously by replacing $G_{\ad}$ with $M_{\ad}$.

For $\alpha_{L} \in \Roots{G_{\ad}}{Z(L)}$, we denote by $\Lie(G_{\ad})_{\alpha_{L}}$ the $\alpha_{L}$-weight space of $Z(L)$ acting on $\Lie(G_{\ad})$.
For $\alpha_{L} \in \Roots{G_{\ad}}{Z(L)}$, $x \in \cB(L, F)$ and $t \in \mathbb{R}$, we let $E$ be a tamely ramified field extension of $F$ such that $\alpha_{L}$ is defined over $E$ and $T$ a maximal torus of $L$ that splits over a tamely ramified field extension $E_{T}$ of $F$ containing $E$ and such that $x \in \cA(L, T, E)$.
Then we write
\[
\Lie(G_{\ad})_{\alpha_{L}}(E)_{x, t} = \left(
\bigoplus_{\alpha \in \Phi(L, T), \alpha \restriction_{Z(L)} = \alpha_{L}} \Lie(G_{\ad})_{\alpha}(E_{T})_{x, t}
\right) \cap \Lie(G_{\ad})_{\alpha_{L}}(E).
\]
By the discussions in \cite[Section~1 and~2]{Yu}, the lattice $\Lie(G_{\ad})_{\alpha_{L}}(E)_{x, t} $ is independent of the choice of the maximal torus $T$ and the field extension $E_{T}$.
We also write $\Lie(G_{\ad})_{\alpha_{L}}(E)_{x, t : t+} = \Lie(G_{\ad})_{\alpha_{L}}(E)_{x, t}/\Lie(G_{\ad})_{\alpha_{L}}(E)_{x, t+}$.

\subsection{The quadratic character when $G$ is adjoint}
\label{subsec:character-adjoint}

We will prove Theorem~\ref{twistextension}
starting with the special case where $G$ is adjoint, so from now on until \S\ref{subsec:extension-general} we assume that $G$ is adjoint. We write $s=\frac{r}{2}$.

We also assume that there exists and fix a $G$-good element $X \in \Lie^{*}(G'_{\scn, \ab})(F)_{-r}$ of depth $-r$ in the sense of \cite[Section~3]{FKS}. 
 Then \cite[Theorem~3.4 and Corollary~3.6]{FKS} provide us with a quadratic character\index{notation-ky}{epsilonGG'x@$\epsilon^{G/G'}_{x}$}%
\[
\epsilon^{G/G'}_{x_{0}} \colon G'(F)_{x_{0}}/G'(F)_{x_{0}, 0+} \rightarrow \{ \pm 1 \}
\]
that we identify with its inflation to $G'(F)_{x_{0}}$.
\index{notation-ky}{epsilonGG'x@$\epsilon^{G/G'}_{x}$}%

According to \cite[Definition~5.5.8]{FKS}, the character $\epsilon^{G/G'}_{x_{0}}$ can be defined as the product of three characters; the character $_{G'} \epsilon_{\sym, \ram}$ defined in \cite[Definition~5.5.1]{FKS}; the character $_{G'} \epsilon^{\sym, \ram}_{s}$ defined in \cite[Definition~5.5.5]{FKS}; the character $_{G'} \epsilon_{0}$ defined in \cite[Definition~5.5.7]{FKS}.
We will recall (equivalent) definitions of these characters below. 
For this, we prepare some notation.

Following \cite[Definition~5.3.1]{FKS} and using \cite[Corollary~2.3]{Yu} as in \cite[Remark~5.5.3]{FKS},
for $x \in \mathcal{B}(M', F)$ and $(O_{L}, t) \in \left(
\Roots{G}{Z(L)}/I_{F}
\right) \times \mathbb{R}$, we define the $\overline{\ff}$-vector space $V_{(x, O_{L}, t)}$ 
 by \spacingatend{}
\begin{align*}
V_{(x, O_{L}, t)}
= \Bigl(
\bigoplus_{\alpha_{L} \in O_{L}} \Lie(G)_{\alpha_{L}}(E^{\unr})_{x, t : t+}
\Bigr)
^{\Gal(E^{\unr}/F^{\unr})},
\end{align*}
where
$E$ is the tamely ramified splitting field of $Z(L)$.

\subsection{Extension of $_{G'} \epsilon^{\protect\sym, \protect\ram}_{s}$} \label{subsection:character2}
To define $_{G'} \epsilon^{\sym, \ram}_{s}$, let $\mathfrak{S} = \Roots{G}{Z(G')}^{\sym, \ram}/I_{F}$.
We fix a subset $\mathfrak{S}^{+}$ such that
$
\mathfrak{S} = \mathfrak{S}^{+} \sqcup - \mathfrak{S}^{+}
$
and write $\mathfrak{S}^{-} = - \mathfrak{S}^{+}$. 
We recall from \cite[Remark~5.5.10]{FKS} that the $\overline{\ff}$-vector spaces  $V_{s}^{+}$ and $V_{s}^{-}$ are defined by
\[
V_{s}^{+} = \bigoplus_{O_{G'} \in \mathfrak{S}^{+}} V_{(x_{0}, O_{G'}, s)}
\qquad
\text{and}
\qquad 
V_{s}^{-} = \bigoplus_{O_{G'} \in \mathfrak{S}^{-}} V_{(x_{0}, O_{G'}, s)},
\]
and the quadratic form $\varphi_{s}$ on $V_{s}^{+} \oplus V_{s}^{-}$ is given by 
$
\varphi_{s}(Y^{+} + Y^{-}) = X \left(
[\widetilde{Y}^{+}, \widetilde{Y}^{-}]
\right)
$ 
for $Y^{+} \in V_{s}^{+}$ and $Y^{-} \in V_{s}^{-}$, where $\widetilde{Y}^{+}$ and $\widetilde{Y}^{-}$ denote lifts of $Y^{+}$ and $Y^{-}$ in $\Lie G_{sc}(\overline F)$ respectively (using the canonical identifications of the weight subspaces of $\Lie G(\overline F)$ and  $\Lie (G_{sc})(\overline F)$), and $[\phantom{x},\phantom{x}]$ denotes the Lie bracket in $\Lie G_{sc}(\overline F)$.
Let $\ort ( V_{s}^{+} \oplus V_{s}^{-}, \varphi_{s} )$
denote the orthogonal group of
$( V_{s}^{+} \oplus V_{s}^{-}, \varphi_{s} )$,
which is defined over $\ff$.
According to \cite[Remark~5.5.10]{FKS}, the adjoint action of $G$ on $\Lie(G)$ induces a group homomorphism
\begin{equation}
\label{grouphomtoorthogonalforexceptsymram}
G'(F)_{x_{0}} \longrightarrow
\ort ( V_{s}^{+} \oplus V_{s}^{-}, \varphi_{s} )(\ff)
\end{equation}
and the character $_{G'} \epsilon^{\sym, \ram}_{s}$ agrees with the composition of \eqref{grouphomtoorthogonalforexceptsymram} with the homomorphism
\begin{equation}
\label{spinornormVspm}
\ort ( V_{s}^{+} \oplus V_{s}^{-}, \varphi_{s} )(\ff)
\stackrel{\sn}{\longrightarrow} \ff^{\times}/(\ff^{\times})^{2} 
\stackrel{\sgn_{\ff}}{\longrightarrow} \{ \pm 1\},
\end{equation}
where $\sn$ denotes the spinor norm,
as in Section \ref{subsec:notation-twistextension}.

We also define the quadratic space $(V_{s, M}^{+} \oplus V_{s, M}^{-}, \varphi_{s, M})$
by replacing $\mathfrak{S}^{\pm}$ with
\[
\mathfrak{S}^{\pm}_{M} = \mathfrak{S}^{\pm} \cap
\left(
\Roots{M}{Z(G')}^{\sym, \ram}/I_{F}
\right)
\]
in the construction of the quadratic space $\left(V_{s}^{+} \oplus V_{s}^{-}, \varphi_{s} \right)$.
Since the embedding
$
\iota \colon \mathcal{B}(M, F) \hookrightarrow \mathcal{B}(G, F)
$
is $s$-generic relative to $x_{0}$, the inclusion $\Lie(M) \subset \Lie(G)$ induces an isomorphism of quadratic spaces 
\[
\bigl(
(
V_{s, M}^{+} \oplus V_{s, M}^{-}
)(\ff), \varphi_{s, M}
\bigr)
\simeq 
\bigl(
(
V_{s}^{+} \oplus V_{s}^{-}
)(\ff), \varphi_{s} 
\bigr).
\]
Moreover, the adjoint action of $N_{G'}(M')(F)_{[x_{0}]_{M'}}=N_{G'}(M')(F)_{[x_{0}]_{M}}$ on $\Lie(M)$ induces an action on
$(
V_{s, M}^{+} \oplus V_{s, M}^{-}
)(\ff)$
that preserves $\varphi_{s, M}$, because the action of $G'(F)$ on $X$ is trivial, and hence yields a group 
 homomorphism 
\[
N_{G'}(M')(F)_{[x_{0}]_{M'}}
\longrightarrow
\ort ( V_{s, M}^{+} \oplus V_{s, M}^{-}, \varphi_{s, M} )(\ff)
\isoarrow
\ort ( V_{s}^{+} \oplus V_{s}^{-}, \varphi_{s} )(\ff).
\]
We define the quadratic character $_{G'} \widetilde{\epsilon}^{\sym, \ram}_{s}$\index{notation-ky}{epsilontildesymrams@$_{G'} \widetilde{\epsilon}^{\sym, \ram}_{s}$}
of $N_{G'}(M')(F)_{[x_{0}]_{M'}}$
as the composition of this homomorphism with \eqref{spinornormVspm}. 

\begin{lemma} 
\label{extensionepsilon2}
The character $_{G'} \widetilde{\epsilon}^{\sym, \ram}_{s}$ is an extension of the character $_{G'} \epsilon^{\sym, \ram}_{s}\restriction_{N_{G'}(M')(F)_{x_{0}}}$ to the group $N_{G'}(M')(F)_{[x_{0}]_{M'}}$.
\end{lemma}
\begin{proof}
The lemma follows from the definition of $_{G'} \widetilde{\epsilon}^{\sym, \ram}_{s}$ and the above description of $_{G'} {\epsilon}^{\sym, \ram}_{s}$.
\end{proof}


\subsection{A description of the character $\epsilon^{G/G'}_{x_0}/_{G'} \epsilon^{\protect\sym, \protect\ram}_{s}$}

Recall that $\epsilon^{G/G'}_{x_0}=_{G'} \epsilon_{\sym, \ram} \cdot {_{G'} \epsilon^{\sym, \ram}_{s}}  \cdot  {_{G'} \epsilon_{0}}$. 
In this subsection, we will recall definitions of the characters $_{G'} \epsilon_{\sym, \ram}$ and $_{G'} \epsilon_{0}$ from \cite{FKS} 
and deduce a description of the product $_{G'} \epsilon_{\sym, \ram} \cdot {_{G'} \epsilon_{0}}$ in Lemma~\ref{rewritethetwopiecesofepsion} that will allow us to extend (the restriction of) $_{G'} \epsilon_{\sym, \ram} \cdot {_{G'} \epsilon_{0}}= \epsilon^{G/G'}_{x_0}/_{G'} \epsilon^{\sym, \ram}_{s}$, and hence $\epsilon^{G/G'}_{x_0}$,  in the subsequent subsections.

We introduce additional notation, closely following \cite{FKS}.
For $\alpha_{M'} \in \Roots{G}{Z(M')}$,
let $F_{\alpha_{M'}}$\index{notation-ky}{FalphaM'@$F_{\alpha_{M'}}$} and $F_{\pm \alpha_{M'}}$\index{notation-ky}{FalphaM'@$F_{\pm \alpha_{M'}}$} denote the subfields of $\overline{F}$ such that
\[
\Gal(\overline{F}/F_{\alpha_{M'}}) = \left\{
\sigma \in \Gal(\overline{F}/F) \mid \sigma (\alpha_{M'}) = \alpha_{M'}
\right\}
\]
and
\[
\Gal(\overline{F}/F_{\pm \alpha_{M'}}) = \left\{
\sigma \in \Gal(\overline{F}/F) \mid \sigma (\alpha_{M'}) \in \left\{
\pm \alpha_{M'}
\right\}
\right\}.
\]
We write $e_{\alpha_{M'}}$ for the ramification degree of the extension $F_{\alpha_{M'}}/F$, and $\ff_{\alpha_{M'}}$ for the residue field of $F_{\alpha_{M'}}$.
Similarly, we define $F_{\alpha_{G'}}$, $F_{\pm \alpha_{G'}}$, $e_{\alpha_{G'}}$, and $\ff_{\alpha_{G'}}$ for $\alpha_{G'} \in \Roots{G}{Z(G')}$.
We set 
\[
\Phi'_{M'} = \{
\alpha_{M'} \in \Roots{G}{Z(M')} \mid
\alpha_{M'}\restriction_{Z(G')} \in \Roots{G}{Z(G')}_{\sym, \ram}
\}.
\]
We also write 
\[
\Phi'_{M', \asym} = \Roots{G}{Z(M')}_{\asym} \cap \Phi'_{M'}
\qquad
\text{and}
\qquad
\Phi'_{M', \sym} = \Phi'_{M'} \smallsetminus \Phi'_{M', \asym}.
\]
Following \cite[Definition~5.5.1]{FKS}, for $O_{G'} \in \Roots{G}{Z(G')}_{\sym, \ram}/I_{F}$, we write $e_{O_{G'}}$ for the common value $e_{\alpha_{G'}}$ for every $\alpha_{G'} \in O_{G'}$.
For $O_{M'} \in \Phi'_{M'} / I_{F}$, we set\index{notation-ky}{sOM'@$s_{O_{M'}}$} 
$
s_{O_{M'}} =e_{O_{G'}}^{-1}/2,
$
where $O_{G'} = O_{M'}\restriction_{Z(G')}$.
For $x \in \mathcal{B}(M', F)$, we define the $\overline{\ff}$-vector spaces $V_{x}$ by
\begin{align*}
V_{x} 
&= \bigoplus_{O_{G'} \in \Roots{G}{Z(G')}_{\sym, \ram}/I_{F} } \bigoplus_{t \in (-e_{O_{G'}}^{-1}/2, e_{O_{G'}}^{-1}/2)} V_{(x, O_{G'}, t)} = \bigoplus_{O_{M'} \in \Phi'_{M'}/I_{F} } \bigoplus_{t \in (-s_{O_{M'}}, s_{O_{M'}})} V_{(x, O_{M'}, t)}.
\end{align*}
Applying \cite[Definition/Lemma~5.4.8]{FKS} to the finite subset 
\[
\mathfrak{S} = \left\{
(O_{G'}, t) \in \Roots{G}{Z(G')}_{\sym, \ram}/I_{F} \times \bR \mid t \in (-e_{O_{G'}}^{-1}/2, e_{O_{G'}}^{-1}/2), \, V_{(x, O_{G'}, t)} \neq \{0\}
\right\}
\]
of $\Roots{G}{Z(G')}/I_{F} \times \bR$, we obtain the non-degenerate, $\Gal(\overline{\ff}/\ff)$-invariant quadratic form $\varphi$ on $V_{x}$.
Instead of recalling the precise definition of $\varphi$, we only remark the following property of $\varphi$, which follows from \cite[Definition~5.4.5]{FKS}.
\begin{property}
\label{propertyofvarphi}
Let $O_{M'}, O'_{M'} \in \Phi'_{M'}/I_{F}$, $t \in (-s_{O_{M'}}, s_{O_{M'}})$, and $t' \in (-s_{O'_{M'}}, s_{O'_{M'}})$.
Then the subspace $V_{(x, O_{M'}, t)}$ of $V_{x}$ is orthogonal to $V_{(x, O'_{M'}, t')}$ with respect to the bilinear form on $V_{x}$ attached to the quadratic form $\varphi$ unless $(O_{M'}, t) = (-O'_{M'}, - t')$.
\end{property}
The adjoint action of $G'$ on $\Lie(G)$ induces an action of $G'(F^{\unr})_{x}$ on $V_{x}$, and this action preserves the form $\varphi$ by \cite[Definition/Lemma~5.4.8]{FKS}.
We also define the $G'(F^{\unr})_{x}$-stable subspaces $V_{x, 0}$ and $V_{x, \neq 0}$ of $V_{x}$ by 
\[
V_{x, 0} = \bigoplus_{O_{M'} \in \Phi'_{M'}/I_{F}} V_{(x, O_{M'}, 0)}
\qquad
\text{and}
\qquad
V_{x, \neq 0} =
\bigoplus_{O_{M'} \in \Phi'_{M'}/I_{F} } \quad
\bigoplus_{\substack{t \in (-s_{O_{M'}}, s_{O_{M'}}) \\ t \neq 0}}
V_{(x, O_{M'}, t)},
\]
and note that $V_{x_0, 0}$ agrees with the space denoted by $V_{0, \bar k}$ in \cite[page~2293]{FKS} for the subset $\Phi(G, Z(G'))_{\sym, \ram} \subset \Phi(G, Z(G'))$ that is used in \cite[Definition~5.5.7]{FKS} to define $_{G'}\epsilon_{0}$.
By Property~\ref{propertyofvarphi}, we have an orthogonal decomposition
\(
V_{x} = V_{x, 0} \oplus V_{x, \neq 0}.
\)
In particular, the paring $\varphi$ yields non-degenerate quadratic forms on $V_{x, 0}$ and $V_{x, \neq 0}$.
Since the pairing $\varphi$ is $\Gal(\overline{\ff}/\ff)$-invariant, the spaces $(V_{x}(\ff), \varphi)$, $(V_{x, 0}(\ff), \varphi)$, and $(V_{x, \neq 0}(\ff), \varphi)$ are non-degenerate quadratic spaces over $\ff$, and the actions of $G'(F)_{x}$ on these spaces induce rational orthogonal representations
$
G'(F)_{x} \longrightarrow \ort( W, \varphi )(\ff)
$
for $W = V_{x}$, $V_{x, 0}$, and $V_{x, \neq 0}$.
According to \cite[Definition~5.5.7]{FKS}, \cite[Definition~5.5.1]{FKS} and \cite[Remark~5.5.10]{FKS}, the character ${_{G'} \epsilon_{0}}$, resp.,\ $_{G'} \epsilon_{\sym, \ram}$ agrees with the composition
\(
G'(F)_{x_{0}} \longrightarrow \ort( W, \varphi )(\ff) 
\xrightarrow{\sn} \ff^{\times}/(\ff^{\times})^{2} 
\xrightarrow{\sgn_{\ff}} \{ \pm 1\},
\)
where $W = V_{x_{0}, 0}$, resp.,\ $V_{x_{0}, \neq 0}$. 
\begin{lemma}
\label{rewritethetwopiecesofepsion}
The product $_{G'} \epsilon_{\sym, \ram} \cdot {_{G'} \epsilon_{0}}$ agrees with the composition
\[
G'(F)_{x_{0}} \longrightarrow \ort ( V_{x_{0}}, \varphi )(\ff) 
\xrightarrow{\sn} \ff^{\times}/(\ff^{\times})^{2} 
\xrightarrow{\sgn_{\ff}} \{ \pm 1\}.
\]
\end{lemma}
\begin{proof}
The lemma follows from the orthogonal decomposition
$
V_{x_{0}} = V_{x_{0}, 0} \oplus V_{x_{0}, \neq 0}
$ by \cite[Section~55.4]{MR1754311}.
\end{proof}

\subsection{An action of $N_{G'}(M')(F)_{[x_{0}]_{M'}}$ on $V_{x_0}$} \label{subsection:constructionofon}
In order to extend the character $(_{G'} \epsilon_{\sym, \ram} \cdot {_{G'} \epsilon_{0}})\restriction_{N_{G'}(M')(F)_{x_{0}}}$ using Lemma \ref{rewritethetwopiecesofepsion} to $N_{G'}(M')(F)_{[x_{0}]_{M'}}$, we first construct an action of $N_{G'}(M')(F)_{[x_{0}]_{M'}}$ on $V_{x_0}$.
We define the subset $\mathcal{A}_{s\dashgen}$ of $\mathcal{B}(M', F)$ by
\[
\mathcal{A}_{s\dashgen} = \left\{
x \in x_{0} + X_{*}(A_{M'}) \otimes_{\mathbb{Z}} \mathbb{R} \mid \text{
$\iota \colon \cB(M, F) \longrightarrow \cB(G, F)$ is $s$-generic relative to $x$
}
\right\}.
\]
We note that for any $n \in N_{G'}(M')(F)_{[x_{0}]_{M'}}$, we have $n x_{0} \in \mathcal{A}_{s\dashgen}$.
More generally, let $x, y \in \cA_{s\dashgen}$.
In order to define the action of $N_{G'}(M')(F)_{[x_{0}]_{M'}}$ on $V_{x_0}$ in Definition \ref{definitionofon}, we will construct an isomorphism
$
I_{y \mid x} \colon V_{x} \longrightarrow V_{y}
$ that is compatible with the action of $ N_{G'}(M')(F)_{[x_{0}]_{M'}}$ in the sense of Lemma \ref{lemmacommutativityreplacingpoint}. To construct this isomorphism, we begin by relating the jumps of the relevant root group filtrations at $x$ and $y$ that are used to define $V_x$ and $V_y$.

Let $O_{M'} \in \Phi'_{M'}/I_{F}$ and $t \in (-s_{O_{M'}}, s_{O_{M'}})$.
Recall that we have
\[
V_{(x, O_{M'}, t)} =
\Bigl(
\bigoplus_{\alpha_{M'} \in O_{M'}} \Lie(G)_{\alpha_{M'}}(E^{\unr})_{x, t : t+}
\Bigr)
^{\Gal(E^{\unr}/F^{\unr})}.
\]
In particular, we have $V_{(x, O_{M'}, t)} = \{0\}$ unless 
$
\Lie(G)_{\alpha_{M'}}(F^{\unr}_{\alpha_{M'}})_{x, t : t+} \neq \{0\}
$
for some (hence all) $\alpha_{M'} \in O_{M'}$.
The general theory of Galois descent for vector spaces implies that
\[
\Lie(G)_{\alpha_{M'}}(F^{\unr}_{\alpha_{M'}})_{x, t : t+} = \Lie(G)_{\alpha_{M'}}(F_{\alpha_{M'}})_{x, t : t+} \otimes_{\ff_{\alpha_{M'}}} \overline{\ff}.
\]
Hence, we obtain that 
 $V_{(x, O_{M'}, t)} = \{0\}$ unless 
$
\Lie(G)_{\alpha_{M'}}(F_{\alpha_{M'}})_{x, t : t+} \neq \{0\}
$.
For $O_{M'} \in \Phi'_{M'}/I_{F}$, we define
\[
J(O_{M'}; x) =
\left\{
t \in (-s_{O_{M'}}, s_{O_{M'}}) \mid \text{$\Lie(G)_{\alpha_{M'}}(F_{\alpha_{M'}})_{x, t : t+} \neq \{0\}$ for all $\alpha_{M'} \in O_{M'}$}
\right\}.
\]
Then we have
\[
V_{x} =
\bigoplus_{O_{M'} \in \Phi'_{M'}/I_{F} } \quad
\bigoplus_{t \in J(O_{M'}; x)} V_{(x, O_{M'}, t)}.
\]
Since
$
y - x \in X_{*}(A_{M'}) \otimes_{\mathbb{Z}} \mathbb{R} \subset X_{*}(Z(M')) \otimes_{\mathbb{Z}} \mathbb{R}
$,
the pairing $\langle y - x, \alpha_{M'} \rangle$ does not depend on the choice of $\alpha_{M'} \in O_{M'}$.
We write it as $\langle y - x, O_{M'} \rangle$.
\begin{lemma}
\label{lemmat'notinoddinteger}
Let $O_{M'} \in \Phi'_{M'}/I_{F}$ and $t \in J(O_{M'}; x)$.
Then we have
\[
t + \langle y - x, O_{M'} \rangle \not \in s_{O_{M'}} \mathbb{Z} \smallsetminus 2 s_{O_{M'}} \mathbb{Z}.
\]
\end{lemma}
\begin{proof}
\addtocounter{equation}{-1}
\begin{subequations}
We write $t' = t + \langle y - x, O_{M'} \rangle$.
Fix an element $\alpha_{M'} \in O_{M'}$.
The definition of the Moy--Prasad filtration implies that
\begin{align*}
\Lie(G)_{\alpha_{M'}}(F_{\alpha_{M'}})_{x, t : t+} &= \Lie(G)_{\alpha_{M'}}(F_{\alpha_{M'}})_{y + (x-y), t : t+} \\
&= \Lie(G)_{\alpha_{M'}}(F_{\alpha_{M'}})_{y, t + \langle y - x, \alpha_{M'} \rangle : (t + \langle y - x, \alpha_{M'} \rangle )+} 
= \Lie(G)_{\alpha_{M'}}(F_{\alpha_{M'}})_{y, t' : t'+}.
\end{align*}
Hence, the assumption $t \in J(O_{M'}; x)$ implies that 
\begin{align}
\label{t'isajumpaty}
\Lie(G)_{\alpha_{M'}}(F_{\alpha_{M'}})_{y, t' : t'+} \neq \{0\}.
\end{align}
Suppose that 
$
t' \in s_{O_{M'}} \mathbb{Z} \smallsetminus 2 s_{O_{M'}} \mathbb{Z}
$.
Then there exists an odd integer $r'$ such that $t' = s_{O_{M'}} \cdot r'$.
Since $\alpha_{G'} \coloneqq  \alpha_{M'}\restriction_{Z(G')}$ is ramified symmetric, according to \cite[Lemma~5.6.5]{FKS}, the number $e_{\alpha_{G'}} r$ is an odd integer.
Hence, we have
\[
s - t' =
s - s_{O_{M'}} \cdot r' 
= r/2 - e_{\alpha_{G'}}^{-1} \cdot r'/2 
= (e_{\alpha_{G'}}r - r')/2 e_{\alpha_{G'}} 
 \in e_{\alpha_{G'}}^{-1} \mathbb{Z}.
\]
Thus, there exists an element $\varpi_{s - t'} \in F_{\alpha_{G'}} \subset F_{\alpha_{M'}}$ such that $\ord(\varpi_{s - t'}) = s - t' $.
Then, \eqref{t'isajumpaty} implies that
\begin{multline*}
\Lie(G)_{\alpha_{M'}}(F_{\alpha_{M'}})_{y, s} / \Lie(G)_{\alpha_{M'}}(F_{\alpha_{M'}})_{y, s+}
\\
=
 \varpi_{s - t'} \cdot \Lie(G)_{\alpha_{M'}}(F_{\alpha_{M'}})_{y, t'} / \varpi_{s - t'} \cdot \Lie(G)_{\alpha_{M'}}(F_{\alpha_{M'}})_{y, t'+} 
\neq \{0\}.
\end{multline*}
As the embedding $\iota \colon \cB(M, F) \longrightarrow \cB(G, F)$ is $s$-generic relative to $y$, we have $\alpha_{M'} \in \Roots{M}{Z(M')}$.
Since
$
y - x \in X_{*}(A_{M'}) \otimes_{\mathbb{Z}} \mathbb{R}
$
and since $M$ is the centralizer of $A_{M'}$,
we have $\langle y - x, \alpha_{M'} \rangle = 0$.
Thus, we obtain $t' = t \in (- s_{O_{M'}}, s_{O_{M'}})$, a contradiction.
\end{subequations}
\end{proof}
Let $O_{M'} \in \Phi'_{M'}/I_{F}$ and $t \in J(O_{M'}; x)$.
that there exists a unique $r(y-x; O_{M'}; t) \in 2 s_{O_{M'}} \mathbb{Z}$\index{notation-ky}{ry-x@$r(y-x; O_{M'}; t)$} such that
\[
t + \langle y - x, O_{M'} \rangle + r(y-x; O_{M'}; t)  \in (-s_{O_{M'}}, s_{O_{M'}}).
\]
\begin{remark}
\label{remarksymmetricvanish}
Suppose that $\alpha_{M'} \in \Roots{G}{Z(M')}$ is symmetric.
Then, there exists an element $\sigma_{\alpha_{M'}} \in \Gal(\overline{F}/F)$ such that $\sigma_{\alpha_{M'}}(\alpha_{M'}) = - \alpha_{M'}$.
In this case, we have
\[
- \langle y - x, \alpha_{M'} \rangle = \langle y - x, - \alpha_{M'} \rangle
= \langle y - x, \sigma_{\alpha_{M'}}(\alpha_{M'}) \rangle
= \langle \sigma_{\alpha_{M'}}^{-1}(y - x), \alpha_{M'} \rangle
= \langle y - x, \alpha_{M'} \rangle.
\]
Hence, we obtain that $\langle y - x, \alpha_{M'} \rangle = 0$.
Thus, we have $r(y-x; I \alpha_{M'}; t) = 0$ for all $t \in~ (-s_{O_{M'}}, s_{O_{M'}})$.
\end{remark}
\begin{lemma}
\label{lemmabijectionofJ}
Let $x, y \in \mathcal{A}_{s\dashgen}$ and $O_{M'} \in \Phi'_{M'}/I_{F}$.
Then the map
\[
t \mapsto t + \langle y - x, O_{M'} \rangle + r(y-x; O_{M'}; t)
\]
defines a bijection
$
J(O_{M'}; x) \longrightarrow J(O_{M'}; y)
$.
\end{lemma}
\begin{proof}
Let $t \in J(O_{M'}; x)$.
We write
$
t' = t + \langle y - x, O_{M'} \rangle 
$
and
$
t'' = t + \langle y - x, O_{M'} \rangle + r(y-x; O_{M'}; t)
$.
We will prove that $t'' \in J(O_{M'}; y)$.
The definition of $r(y-x; O_{M'}; t)$ implies that $t'' \in (-s_{O_{M'}}, s_{O_{M'}})$.
Thus, it suffices to show that
$
\Lie(G)_{\alpha_{M'}}(F_{\alpha_{M'}})_{y, t'' : t''+} \neq \{0\} 
$
for all $\alpha_{M'} \in O_{M'}$.
As in the proof of Lemma~\ref{lemmat'notinoddinteger}, the definition of the Moy--Prasad filtration implies that
\[
\Lie(G)_{\alpha_{M'}}(F_{\alpha_{M'}})_{x, t : t+} = \Lie(G)_{\alpha_{M'}}(F_{\alpha_{M'}})_{y, t' : t'+}.
\]
Hence, the assumption $t \in J(O_{M'}; x)$ implies that
$
\Lie(G)_{\alpha_{M'}}(F_{\alpha_{M'}})_{y, t' : t'+} \neq \{0\}
$.
Since
\[
r(y-x; O_{M'}; t) \in 2 s_{O_{M'}} \mathbb{Z}
= e_{\alpha_{G'}}^{-1} \mathbb{Z},
\]
there exists an element $\varpi_{r} \in F_{\alpha_{G'}} \subset F_{\alpha_{M'}}$ such that $\ord(\varpi_{r}) = r(y-x; O_{M'}; t)$.
Then, we have
\[
\Lie(G)_{\alpha_{M'}}(F_{\alpha_{M'}})_{y, t''} / \Lie(G)_{\alpha_{M'}}(F_{\alpha_{M'}})_{y, t''+} =
 \varpi_{r} \cdot \Lie(G)_{\alpha_{M'}}(F_{\alpha_{M'}})_{y, t'} / \varpi_{r} \cdot \Lie(G)_{\alpha_{M'}}(F_{\alpha_{M'}})_{y, t'+} \neq \{0\}.
\]
Thus, $t \mapsto t + \langle y - x, O_{M'} \rangle + r(y-x; O_{M'}; t)$ defines a map
$
J(O_{M'}; x) \longrightarrow J(O_{M'}; y)
$.
Replacing $x$ with $y$, we also obtain a map
$
J(O_{M'}; y) \longrightarrow J(O_{M'}; x)
$
by
$
t'' \mapsto t'' + \langle x - y, O_{M'} \rangle + r(x-y; O_{M'}; t'')
$.
Since we have
$
\langle x - y, O_{M'} \rangle = - \langle y - x, O_{M'} \rangle
$, these maps are inverses of each other.
\end{proof}

The construction of the isomorphism $I_{y \mid x} \colon V_{x} \longrightarrow V_{y}$ involves scaling by appropriate powers of uniformizers  $\varpi_{\alpha_{M'}}$ of $F_{\alpha_{M'}\restriction_{Z(G')}}$ for $\alpha_{M'} \in \Phi'_{M'}$ based on the bijection in Lemma \ref{lemmabijectionofJ}. This requires to be able to choose the uniformizers in a compatible way, which we prove next, see Lemma~\ref{lemmafixuniformizer}. We first introduce some terminology.

\begin{definition}
Let $\alpha_{M'} \in \Phi'_{M'}$.
We say that $\alpha_{M'}$ is \emph{$N$-symmetric} if there exists $n \in~ N_{G'}(M')(F)_{[x_{0}]_{M'}}$ and $\sigma \in \Gal(\overline{F}/F)$ such that
$
\sigma(n \alpha_{M'}) = - \alpha_{M'}
$.
We also say that an element $\alpha_{M'} \in \Phi'_{M'}$ is \emph{$N$-asymmetric} if it is not $N$-symmetric.
We define the subsets $\Phi'_{M', \Nsym}$ and $\Phi'_{M', \Nasym}$ of $\Phi'_{M'}$ by \spacingatend{the centered piece could me made slightly less wide and then moved up a bit if ``by'' stays the only word in this line}
\[
\Phi'_{M', \Nsym} = \left\{
\alpha_{M'} \in \Phi'_{M'} \mid \text{$\alpha_{M'}$ is $N$-symmetric}
\right\}
\qquad
\text{and}
\qquad
\Phi'_{M', \Nasym} = \Phi'_{M'} \smallsetminus \Phi'_{M', \Nsym}.
\]
\end{definition}
\begin{lemma}
\label{lemmagooduniformizer}
Let $\alpha_{M'} \in \Phi'_{M', \Nsym}$.
We write $\alpha_{G'} = \alpha_{M'}\restriction_{Z(G')}$.
Then there exists a uniformizer $\varpi_{\alpha_{M'}}$ of $F_{\alpha_{G'}}$ such that
$
\sigma(\varpi_{\alpha_{M'}}) = - \varpi_{\alpha_{M'}}
$
for all $\sigma \in \Gal(\overline{F}/F)$ for which there exists $n \in N_{G'}(M')(F)_{[x_{0}]_{M'}} $
with $\sigma(n \alpha_{M'}) = - \alpha_{M'}$.
\end{lemma}
\begin{proof}
\addtocounter{equation}{-1}
\begin{subequations}
Since $\alpha_{M'} \in \Phi'_{M'}$, we have $\alpha_{G'} \in \Roots{Z(G')}{G}_{\sym, \ram}$.
Hence, the field extension $F_{\alpha_{G'}} / F_{\pm \alpha_{G'}}$ is a ramified quadratic extension.
Thus, we can take a uniformizer $\varpi_{\alpha_{M'}}$ of $F_{\alpha_{G'}}$ such that 
\begin{align}
\label{iotaswitchessign}
\iota(\varpi_{\alpha_{M'}}) = - \varpi_{\alpha_{M'}},
\end{align}
where $\iota$ denotes the unique non-trivial element of $\Gal(F_{\alpha_{G'}}/F_{\pm \alpha_{G'}})$.
We will prove that this $\varpi_{\alpha_{M'}}$ satisfies the condition of the lemma.
Let $n \in N_{G'}(M')(F)_{[x_{0}]_{M'}}$ and $\sigma \in \Gal(\overline{F}/F)$ such that
$
\sigma(n \alpha_{M'}) = - \alpha_{M'}
$.
Since $n \in G'(F)$, we have
\[
\sigma(\alpha_{G'}) = \sigma(n \alpha_{G'}) 
= \sigma(n \alpha_{M'})\restriction_{Z(G')} 
= - \alpha_{M'}\restriction_{Z(G')} 
= - \alpha_{G'}.
\]
Then the definitions of $F_{\alpha_{G'}}$ and $F_{\pm \alpha_{G'}}$ imply that
$
\sigma \in \Gal(\overline{F}/F_{\pm \alpha_{G'}}) \smallsetminus \Gal(\overline{F}/F_{\alpha_{G'}})
$.
Thus, we obtain that $\sigma\restriction_{F_{\alpha_{G'}}} = \iota$.
Now, the lemma follows from \eqref{iotaswitchessign}.
\end{subequations}
\end{proof}
\begin{lemma}
\label{lemmafixuniformizer}
We can choose an element $\varpi_{\alpha_{M'}}$ for every $\alpha_{M'} \in \Phi'_{M'}$ such that  
\begin{enumerate}[(1)]
\item
$\varpi_{\alpha_{M'}}$ is a uniformizer of $F_{\alpha_{G'}}$ for all $\alpha_{M'} \in \Phi'_{M'}$, where $\alpha_{G'} = \alpha_{M'}\restriction_{Z(G')}$.
\item
$\varpi_{n \alpha_{M'}} = \varpi_{\alpha_{M'}}$
for all $n \in N_{G'}(M')(F)_{[x_{0}]_{M'}}$ and $\alpha_{M'} \in \Phi'_{M'}$.
\item
$\sigma(\varpi_{\alpha_{M'}}) = \varpi_{\sigma(\alpha_{M'})}$
for all $\sigma \in \Gal(\overline{F}/F)$ and $\alpha_{M'} \in \Phi'_{M'}$.
\item
\label{conditionfororthogonality}
$\varpi_{- \alpha_{M'}} = - \varpi_{\alpha_{M'}}$
for all $\alpha_{M'} \in \Phi'_{M'}$.
\end{enumerate}
\end{lemma}
\begin{proof}
We fix a set $C$ of representatives of 
$
\Phi'_{M'}/\left(
N_{G'}(M')(F)_{[x_{0}]_{M'}} \times \Gal(\overline{F}/F) \times \{\pm 1\}
\right)
$.
For each $\alpha_{M'} \in C$, we fix a uniformizer $\varpi_{\alpha_{M'}}$ of $F_{\alpha_{G'}}$, 
with $\varpi_{\alpha_{M'}}$ as in Lemma~\ref{lemmagooduniformizer} if $\alpha_{M'} \in \Phi'_{M', \Nsym}$.

Let $\alpha_{M'} \in C$.
Suppose that $n_{1}, n_{2} \in N_{G'}(M')(F)_{[x_{0}]_{M'}}$ and $\sigma_{1}, \sigma_{2} \in \Gal(\overline{F}/F)$ satisfy
$
\sigma_{1}(n_{1} \alpha_{M'}) = \sigma_{2}(n_{2} \alpha_{M'})
$.
Then, since $n_{1}, n_{2} \in G'(F)$, we have 
\[
\sigma_{1}(\alpha_{G'}) = \sigma_{1}(n_{1} \alpha_{G'}) 
= \sigma_{1}(n_{1} \alpha_{M'})\restriction_{Z(G')} 
= \sigma_{2}(n_{2} \alpha_{M'})\restriction_{Z(G')} 
= \sigma_{2}(n_{2} \alpha_{G'}) 
= \sigma_{2}(\alpha_{G'}).
\]
Hence, we obtain that $\sigma_{1}^{-1} \sigma_{2}$ fixes $\alpha_{G'}$.
Since $\varpi_{\alpha_{M'}} \in F_{\alpha_{G'}}$, we have $\sigma_{1}^{-1} \sigma_{2}(\varpi_{\alpha_{M'}}) = \varpi_{\alpha_{M'}}$, that is, $\sigma_{1}(\varpi_{\alpha_{M'}}) = \sigma_{2}(\varpi_{\alpha_{M'}})$.
Thus, for 
$
\beta_{M'} \in \left(
N_{G'}(M')(F)_{[x_{0}]_{M'}} \times \Gal(\overline{F}/F)
\right) \alpha_{M'}
$,
the uniformizer $\sigma(\varpi_{\alpha_{M'}})$ of $F_{\beta_{G'}} = F_{\sigma(\alpha_{G'})}$ does not depend on the choice of 
$
(n, \sigma) \in N_{G'}(M')(F)_{[x_{0}]_{M'}} \times \Gal(\overline{F}/F)
$
such that $\sigma(n \alpha_{M'}) = \beta_{M'}$, and
we set $\varpi_{\beta_{M'}} = \sigma(\varpi_{\alpha_{M'}})$.
If $\alpha_{M'} \in \Phi'_{M', \Nasym}$, we define the uniformizer $\varpi_{- \beta_{M'}}$ of $F_{- \beta_{G'}} = F_{\beta_{G'}}$ by
$
\varpi_{- \beta_{M'}} = - \varpi_{\beta_{M'}}
$
for each
$
\beta_{M'} \in \left(
N_{G'}(M')(F)_{[x_{0}]_{M'}} \times \Gal(\overline{F}/F)
\right)$  $\alpha_{M'}
$.\spacingatend{check the line break here at the end}

Now, we have defined a uniformizer $\varpi_{\alpha_{M'}}$ of $F_{\alpha_{G'}}$ for each $\alpha_{M'} \in \Phi'_{M'}$ that together
 by construction satisfy the first three conditions.
Hence it remains to show that these uniformizers satisfy Condition~\eqref{conditionfororthogonality} and by the first three conditions it suffices to do so for $\alpha_{M'} \in C$.
Let $\alpha_{M'} \in C$.
If $\alpha_{M'} \in \Phi'_{M', \Nasym}$, the choice of the uniformizers above implies that $\varpi_{- \alpha_{M'}} = - \varpi_{\alpha_{M'}}$. So we
assume that $\alpha_{M'} \in \Phi'_{M', \Nsym}$, and let $n \in N_{G'}(M')(F)_{[x_{0}]_{M'}}$ and $\sigma \in \Gal(\overline{F}/F)$ such that
$
\sigma(n \alpha_{M'}) = - \alpha_{M'}
$.
Since we chose the uniformizer $\varpi_{\alpha_{M'}}$ as in Lemma~\ref{lemmagooduniformizer}, we have
$
\sigma(\varpi_{\alpha_{M'}}) = - \varpi_{\alpha_{M'}}
$.
Thus, we obtain that
\[
\varpi_{- \alpha_{M'}} = \varpi_{\sigma(n \alpha_{M'})} 
= \sigma(\varpi_{\alpha_{M'}}) 
= - \varpi_{\alpha_{M'}}.
\qedhere
\]
\end{proof}
We fix uniformizers $\varpi_{\alpha_{M'}}$ of $F_{\alpha_{G'}}$ for $\alpha_{M'} \in \Phi'_{M'}$ as in Lemma~\ref{lemmafixuniformizer}.
We also fix a square root $\zeta$
\index{notation-ky}{zeta@$\zeta$}
of $-1$ in $\overline{F}$ and also regard $\zeta$ as an element of $\overline{\ff}$.
We write $F' = F(\zeta)$ and $\ff' = \ff(\zeta)$.

Let $O_{M'} \in \Phi'_{M'}/I_{F}$ and $t \in J(O_{M'}; x)$.
We set\index{notation-ky}{zy-x@$z(y-x; O_{M'}; t)$} 
\[
z(y-x; O_{M'}; t) \coloneqq   r(y-x; O_{M'}; t) \cdot e_{\alpha_{G'}},
\]
where $\alpha_{G'}$ is any element of $O_{M'}\restriction_{Z(G')}$.
Since $r(y-x; O_{M'}; t) \in 2 s_{O_{M'}} \mathbb{Z} = e_{\alpha_{G'}}^{-1} \bZ$, we have $z(y-x; O_{M'}; t) \in \mathbb{Z}$.
For 
 $\alpha_{M'} \in O_{M'}$,
we define the isomorphism
\begin{align*}
I_{y \mid x}( \alpha_{M'}; t) \colon
\Lie(G)_{\alpha_{M'}}(E^{\unr})_{x, t : t+} &= \Lie(G)_{\alpha_{M'}}(E^{\unr})_{y, t + \langle y - x, \alpha_{M'} \rangle : (t + \langle y - x, O_{M'} \rangle )+} \\
&\xrightarrow{\times (\zeta \varpi_{\alpha_{M'}})^{z(y-x; O_{M'}; t)}} \Lie(G)_{\alpha_{M'}}(E^{\unr})_{y, t''; t''+},
\end{align*}
where
$
t'' = t + \langle y - x, O_{M'} \rangle + r(y-x; O_{M'}; t)
$.
Our choice of the uniformizers $\varpi_{\alpha_{M'}}$ implies that the direct sum 
\[
\bigoplus_{\alpha_{M'} \in O_{M'}} \Lie(G)_{\alpha_{M'}}(E^{\unr})_{x, t : t+} \xrightarrow{\bigoplus I_{y \mid x}( \alpha_{M'}; t)} \bigoplus_{\alpha_{M'} \in O_{M'}} \Lie(G)_{\alpha_{M'}}(E^{\unr})_{y, t'' : t''+}
\]
of the isomorphisms above is defined over $F^{\unr}$.
Thus, we obtain the isomorphism
\begin{align*}
V_{(x, O_{M'}, t)}
&=
\Bigl(
\bigoplus_{\alpha_{M'} \in O_{M'}} \Lie(G)_{\alpha_{M'}}(E^{\unr})_{x, t : t+}
\Bigr)
^{\Gal(E^{\unr}/F^{\unr})} \\
&\longrightarrow
\Bigl(
\bigoplus_{\alpha_{M'} \in O_{M'}} \Lie(G)_{\alpha_{M'}}(E^{\unr})_{y, t'' : t''+}
\Bigr)
^{\Gal(E^{\unr}/F^{\unr})} 
= V_{(y, O_{M'}, t'')}.
\end{align*}
Then, using Lemma~\ref{lemmabijectionofJ}, we obtain the isomorphism
\begin{align*}
I_{y \mid x} \colon V_{x} = \bigoplus_{O_{M'} \in \Phi'_{M'}/I_{F} } \bigoplus_{t \in J(O_{M'}; x)} V_{(x, O_{M'}, t)} 
\longrightarrow \bigoplus_{O_{M'} \in \Phi'_{M'}/I_{F} } \bigoplus_{t'' \in J(O_{M'}; y)} V_{(y, O_{M'}, t'')} 
= V_{y}.
\end{align*}
The construction of $I_{y \mid x}$ implies that
$
I_{z \mid y} \circ I_{y \mid x} = I_{z \mid x}
$
for $x, y, z \in \mathcal{A}_{s\dashgen}$.
In particular, we have
$
I_{x \mid y} \circ I_{y \mid x} = \id_{V_{x}}
$. 
\begin{lemma}
\label{lemmacommutativityreplacingpoint}
Let $x, y \in \mathcal{A}_{s\dashgen}$ and $n \in N_{G'}(M')(F)_{[x_{0}]_{M'}}$.
Then the following diagram commutes:
\[
\xymatrix@R+1pc@C+2pc{
V_{x} \ar[d]_-{I_{y \mid x}} \ar[r]^-{n} \ar@{}[dr]|\circlearrowleft & V_{n x} \ar[d]^-{I_{n y \mid n x}}
\\
V_{y} \ar[r]^-{n} & V_{n y},
}
\]
where
\(
V_{x} \xrightarrow{n} V_{n x} 
\)
and
\(
V_{y} \xrightarrow{n} V_{n y} 
\)
denote the maps induced from the adjoint action of $G$ on $\Lie(G)$.
\end{lemma}
\begin{proof}
Let $\alpha_{M'} \in \Phi'_{M'}$ and $t \in J(I_F \alpha_{M'}; x)$.
Since
$
\langle y - x, \alpha_{M'} \rangle = \langle n y - n x, n \alpha_{M'} \rangle
$,
the definition of $r(y-x; I_F \alpha_{M'}; t)$ implies that
$
r(y-x; I_F \alpha_{M'}; t) = r(n y-n x; I_F n \alpha_{M'}; t)
$.
Hence, our choice of uniformizers implies that the following diagram commutes:
\[
\xymatrix@R+1pc@C+2pc{
\Lie(G)_{\alpha_{M'}}(E^{\unr})_{x, t : t+} \ar[d]_-{I_{y \mid x}( \alpha_{M'}; t)} \ar[r]^-{n} \ar@{}[dr]|\circlearrowleft & \Lie(G)_{n \alpha_{M'}}(E^{\unr})_{nx, t : t+} \ar[d]^-{I_{n y \mid n x}( n \alpha_{M'}; t)}
\\
\Lie(G)_{\alpha_{M'}}(E^{\unr})_{y, t'' : t''+} \ar[r]^-{n} & \Lie(G)_{n \cdot \alpha_{M'}}(E^{\unr})_{n y, t'' : t''+},
}
\]
where
\[
t'' = t + \langle y - x, \alpha_{M'} \rangle + r(y-x; I \alpha_{M'}; t) 
= t + \langle n y - n x, n \alpha_{M'} \rangle + r(n y-n x; I n \alpha_{M'}; t).
\]
Now, the claim follows from the construction of $I_{y \mid x}$.
\end{proof}  
\begin{definition}
	\label{definitionofon}
For $n \in N_{G'}(M')(F)_{[x_{0}]_{M'}}$, we define 
$o_{n} \in \GL(V_{x_{0}})(\overline \ff)$ as the composition
\[
o_{n} \colon V_{x_{0}} \xrightarrow{n} V_{n x_{0}} \xrightarrow{I_{x_{0} \mid n x_{0}}} V_{x_{0}}
.\]
\end{definition}
\begin{corollary}
The map
$
n \mapsto o_{n}
$
defines a group homomorphism
$
N_{G'}(M')(F)_{[x_{0}]_{M'}} \longrightarrow \GL(
V_{x_{0}})(\overline \ff)
$.
\end{corollary}
\begin{proof} 
Let $m, n \in N_{G'}(M')(F)_{[y]_{M'}}$.
We will prove that $o_{nm} = o_{n} \circ o_{m}$.
We have
\begin{align*}
o_{nm} &=
\left(
V_{x_{0}} \xrightarrow{nm} V_{nm   x_{0}} \xrightarrow{I_{x_{0} \mid nm   x_{0}}} V_{x_{0}}
\right) 
= 
\left(
V_{x_{0}} \xrightarrow{m} V_{m   x_{0}} \xrightarrow{n} V_{nm   x_{0}} \xrightarrow{I_{x_{0} \mid nm   x_{0}}} V_{x_{0}}
\right) \\
&= 
\left(
V_{x_{0}} \xrightarrow{m} V_{m   x_{0}}  \xrightarrow{I_{x_{0} \mid m   x_{0}}} V_{x_{0}} \xrightarrow{I_{m   x_{0} \mid x_{0}}} V_{m   x_{0}}
\xrightarrow{n} V_{nm   x_{0}} \xrightarrow{I_{x_{0} \mid nm   x_{0}}} V_{x_{0}}
\right) \\
&= \left(
V_{x_{0}} \xrightarrow{I_{m   x_{0} \mid x_{0}}} V_{m   x_{0}}
\xrightarrow{n} V_{nm   x_{0}} \xrightarrow{I_{x_{0} \mid nm   x_{0}}} V_{x_{0}}
\right) \circ o_{m}.
\end{align*}
By applying Lemma~\ref{lemmacommutativityreplacingpoint} for $x = x_{0}$ and $y = m x_{0}$, we obtain
\begin{align*}
\left(
V_{x_{0}} \xrightarrow{I_{m   x_{0} \mid x_{0}}} V_{m   x_{0}}
\xrightarrow{n} V_{nm   x_{0}} \xrightarrow{I_{x_{0} \mid nm   x_{0}}} V_{x_{0}}
\right) &=
\left(
V_{x_{0}} \xrightarrow{n} V_{n   x_{0}} \xrightarrow{I_{nm   x_{0} \mid n   x_{0}}} V_{nm   x_{0}} \xrightarrow{I_{x_{0} \mid nm   x_{0}}} V_{x_{0}}
\right) \\
&= \left(
V_{x_{0}} \xrightarrow{n} V_{n   x_{0}}  \xrightarrow{I_{x_{0} \mid n   x_{0}}} V_{x_{0}} 
\right) = o_{n}.
\end{align*}
Thus, we conclude that $o_{nm} = o_{n} \circ o_{m}$.
\end{proof}

\subsection{Extension of $\epsilon^{G/G'}_{x_0}/_{G'} \epsilon^{\protect\sym, \protect\ram}_{s}$} \label{subsection:character13}

By Lemma \ref{rewritethetwopiecesofepsion}, the restriction of the character $_{G'} \epsilon_{\sym, \ram} \cdot {_{G'} \epsilon_{0}}$ to ${N_{G'}(M')(F)_{x_{0}}}$ factors through the morphism $n \mapsto o_n$. We will first define a subgroup $\widetilde{\ort}(V_{x_{0}})$ of $\GL(V_{x_0})(\overline \ff)$, see Definition \ref{definitionofOtilde}, that contains $o_n$ for all $n \in  N_{G'}(M')(F)_{[x_{0}]_{M'}}$, see Proposition \ref{propositionon}, and use this to extend $(_{G'} \epsilon_{\sym, \ram} \cdot {_{G'} \epsilon_{0}})\restriction_{N_{G'}(M')(F)_{x_{0}}}$, and hence $\epsilon_{x_0}^{G/G'}$, to $N_{G'}(M')(F)_{[x_{0}]_{M'}}$.
To do so, let 
$
[\alpha_{M'}] \in \left(
\Phi'_{M'}/ I_{F}
\right) / \Gal(\overline{\ff}/\ff) = \Phi'_{M'} / \Gal(\overline{F}/F)
$.
Since $x_{0} \in \mathcal{B}(G, F)$, the set $J(O_{M'}; x_{0})$ does not depend on the choice of $I_F$-orbit $O_{M'} \in [\alpha_{M'}] \subset \Phi'_{M'}/ I_{F}$. 
We write $J([\alpha_{M'}]; x_{0})$ for $J(O_{M'}; x_{0})$.
We define
\[
\widetilde{R}_{x_{0}} =
\left\{
([\alpha_{M'}], t) \mid [\alpha_{M'}] \in \Phi'_{M'} / \Gal(\overline{F}/F), t \in J([\alpha_{M'}]; x_{0})
\right\},
\]
\[
\widetilde{R}_{x_{0}, \asym} =
\left\{
([\alpha_{M'}], t) \mid [\alpha_{M'}] \in \Phi'_{M', \asym} / \Gal(\overline{F}/F), t \in J([\alpha_{M'}]; x_{0})
\right\},
\]
and
\[
\widetilde{R}_{x_{0}, \sym} =
\left\{
([\alpha_{M'}], t) \mid [\alpha_{M'}] \in \Phi'_{M', \sym} / \Gal(\overline{F}/F), t \in J([\alpha_{M'}]; x_{0})
\right\}.
\]
For $([\alpha_{M'}], t) \in \widetilde{R}_{x_{0}}$, we set 
$
V_{(x_{0}, [\alpha_{M'}], t)} = \bigoplus_{O_{M'} \in [\alpha_{M'}]} V_{(x_{0}, O_{M'}, t)}
$, where the sum is taken over the $I_F$-orbits in $[\alpha_{M'}]$.
According to \cite[Corollary~3.11]{MR3849622}, we can define the action of $\{\pm 1\}$ on $\widetilde{R}_{x_{0}}$ by
$
-1 \cdot ([\alpha_{M'}], t) = (- [\alpha_{M'}], - t)
$,
and this action preserves $\widetilde{R}_{x_{0}, \asym}$ and $\widetilde{R}_{x_{0}, \sym}$.
For $[([\alpha_{M'}], t)] \in \widetilde{R}_{x_{0}} / \{\pm1\}$, we define
\begin{align}
\label{plusminusdecomposition}
V_{(x_{0}, [([\alpha_{M'}], t)])} &= \bigoplus_{([\alpha'_{M'}], t') \in [([\alpha_{M'}], t)]} V_{(x_{0}, [\alpha'_{M'}], t')} \\
&= 
\begin{cases}
V_{(x_{0}, [\alpha_{M'}], t)} \oplus V_{(x_{0}, - [\alpha_{M'}], -t)} & \left(
([\alpha_{M'}], t) \neq (- [\alpha_{M'}], - t)
\right), \notag \\
V_{(x_{0}, [\alpha_{M'}], t)} & \left(
([\alpha_{M'}], t) = (- [\alpha_{M'}], - t)
\right).
\end{cases}
\end{align}
For later use, we note that if $([\alpha_{M'}], t) \in \widetilde{R}_{x_{0}, \asym}$, then $([\alpha_{M'}], t) \neq (- [\alpha_{M'}], - t)$ and $
\dim_{\ff}\left(
V_{(x_{0}, [([\alpha_{M'}], t)])}(\ff)
\right)
$
is even.
The spaces $V_{(x_{0}, [([\alpha_{M'}], t)])}$ are defined over $\ff$, and noting Property~\ref{propertyofvarphi}, we have the orthogonal decomposition
\begin{align}
\label{orthogonaldecompositionofvx0}
V_{x_{0}} &=
\bigoplus_{[([\alpha_{M'}], t)] \in \widetilde{R}_{x_{0}} / \{\pm1\}} V_{(x_{0}, [([\alpha_{M'}], t)])} 
\\
&=
\Biggl(
\bigoplus_{[([\alpha_{M'}], t)] \in \widetilde{R}_{x_{0}, \asym} / \{\pm1\}} V_{(x_{0}, [([\alpha_{M'}], t)])}
\Biggr)
\oplus
\Biggl(
\bigoplus_{[([\alpha_{M'}], t)] \in \widetilde{R}_{x_{0}, \sym} / \{\pm1\}} V_{(x_{0}, [([\alpha_{M'}], t)])}
\Biggr).
\notag
\end{align}
For $[([\alpha_{M'}], t)] \in \widetilde{R}_{x_{0}} / \{\pm1\}$, let
$
\zeta_{[([\alpha_{M'}], t)] } \colon V_{(x_{0}, [([\alpha_{M'}], t)])} \longrightarrow V_{(x_{0}, [([\alpha_{M'}], t)])}
$
denote the multiplication by $\zeta \in \ff'$.
We also write $\zeta_{[([\alpha_{M'}], t)] }$ for the element of $\GL(V_{x_{0}})(\ff')$ that acts on $V_{(x_{0}, [([\alpha_{M'}], t)])}$ by $\zeta_{[([\alpha_{M'}], t)] }$ and acts on the other direct summands of the decomposition~\eqref{orthogonaldecompositionofvx0} by the identity map.

\begin{definition} \label{definitionofOtilde}
We define the group $\GL^+(V_{x_{0}})$ to be the subgroup of $\GL(V_{x_{0}})(\overline \ff)$ that consists of all
$g \in \GL(V_{x_{0}})(\overline \ff)$ that satisfy the following property: For every element $[([\alpha_{M'}], t)]$ of $\widetilde{R}_{x_{0}, \asym} / \{\pm1\}$, resp.,\ $\widetilde{R}_{x_{0}, \sym} / \{\pm1\}$, there exists an element $[([\alpha'_{M'}], t')]$ of $\widetilde{R}_{x_{0}, \asym} / \{\pm1\}$, resp.,\ $\widetilde{R}_{x_{0}, \sym} / \{\pm1\}$ such that
\[
g V_{(x_{0}, [([\alpha_{M'}], t)])} = V_{(x_{0}, [([\alpha'_{M'}], t')])}.
\]
We define $\widetilde{\ort}(V_{x_{0}})$ to be the subgroup of $\GL(V_{x_{0}})(\ff')$ generated by $\GL^+(V_{x_{0}}) \cap \ort(V_{x_{0}}, \varphi)(\ff)$ and $\zeta_{[([\alpha_{M'}], t)] }$ for $[([\alpha_{M'}], t)]  \in \widetilde{R}_{x_{0}, \asym} / \{\pm1\}$, i.e., \spacingatend{check vertical spacing}
\[
\widetilde{\ort}(V_{x_{0}}) = \left \langle
\GL^+(V_{x_{0}}) \cap \ort(V_{x_{0}}, \varphi)(\ff), \zeta_{[([\alpha_{M'}], t)] } \mid [([\alpha_{M'}], t)]  \in \widetilde{R}_{x_{0}, \asym} / \{\pm1\}
\right \rangle \subset \GL(V_{x_{0}})(\ff').
\]
\end{definition}
\begin{proposition} \label{propositionon}
For any $n \in N_{G'}(M')(F)_{[x_{0}]_{M'}}$, the element $o_{n}$ is contained in the group $\widetilde{\ort}(V_{x_{0}})$.
\end{proposition}
\begin{proof}
\addtocounter{equation}{-1}
\begin{subequations}
Let $n \in N_{G'}(M')(F)_{[x_{0}]_{M'}}$, $O_{M'} \in \Phi'_{M'}/ I_{F}$, and $t \in J(O_{M'}; x_{0})$.
The definition of the integer $z(x_{0} - n x_{0}; O_{M'}; t)$ implies that it actually only depends on the $\Gal(\overline{\ff}/\ff)$-orbit of $O_{M'}$, and we have
\begin{align}
\label{-differencedifferenceof-}
z(x_{0} - n x_{0}; O_{M'}; t) = - z(x_{0} - n x_{0}; - O_{M'}; - t).
\end{align}
Hence we can define the subset $\widetilde{R}_{x_{0}}(n) / \{\pm1\}$ of $\widetilde{R}_{x_{0}} / \{\pm1\}$ by
\[
\widetilde{R}_{x_{0}}(n) / \{\pm1\}
=
\Bigl\{
[([\alpha_{M'}], t)]  \in \widetilde{R}_{x_{0}} / \{\pm1\} 
\Bigr. \, \Bigl| \,
z(x_{0} - n x_{0}; O_{M'}; t)\equiv 1 \bmod 2 \text{ for } O_{M'} \subset [\alpha_{M'}]
\Bigr\}.
\]
By Remark~\ref{remarksymmetricvanish} we have
$
z(x_{0} - n x_{0}; O_{M'}; t)  = 0
$
for all $O_{M'} \in \Phi'_{M', \sym} / I_{F}$.
Hence, we have
$
\widetilde{R}_{x_{0}}(n) / \{\pm1\} \subset \widetilde{R}_{x_{0}, \asym} / \{\pm1\}
$.
It follows from the definition of $o_{n}$ and our choice of uniformizers that the element 
\(
\Biggl(
\prod_{[([\alpha_{M'}], t)] \in \widetilde{R}_{x_{0}}(n) / \{\pm1\}} \zeta_{[([\alpha_{M'}], t)] }
\Biggr)
\circ o_{n}
\)
is defined over $\ff$ and contained in the group $\GL^+(V_{x_{0}})$. Moreover, this element preserves the quadratic form $\varphi$ on $V_{x}$ by Equation~\eqref{-differencedifferenceof-} and Property \ref{propertyofvarphi}.
Hence, we have
\(
\Biggl(
\prod_{[([\alpha_{M'}], t)] \in \widetilde{R}_{x_{0}}(n) / \{\pm1\}} \zeta_{[([\alpha_{M'}], t)] }
\Biggr)
\circ o_{n} \in \GL^+(V_{x_{0}}) \cap \ort(V_{x_{0}}, \varphi)(\ff)
\), and therefore
$
o_{n} \in \widetilde{\ort}(V_{x_{0}})
$.
\end{subequations}
\end{proof}
To extend our quadratic characters, we fix a square root $\sqrt{\smash[b]{\sgn_{\ff}(-1)}}$ of $\sgn_{\ff}(-1)$ in $\Coeff^{\times}$ and recall that by  Equation~\eqref{plusminusdecomposition} $ 
\dim_{\ff}\left(
V_{(x_{0}, [([\alpha_{M'}], t)])}(\ff)
\right)
$
is even when $[([\alpha_{M'}], t)]  \in \widetilde{R}_{x_{0}, \asym} $.
\begin{proposition} \label{Proposition-extension-epsilon13}
	The character 
$
( \sgn_{\ff} \circ \sn )
\restriction_{\GL^+(V_{x_{0}}) \cap \ort(V_{x_{0}}, \varphi)(\ff)}
$
extends to a unique character $\widetilde{\sn}$ of $\widetilde{\ort}(V_{x_{0}})$ that satisfies \spacingatend{}
\[
\widetilde \sn\left(
\zeta_{[([\alpha_{M'}], t)]}
\right)
= \sqrt{\smash[b]{\sgn_{\ff}}(-1)}^{\dim_{\ff}\left(
	V_{(x_{0}, [([\alpha_{M'}], t)])}(\ff)
	\right) / 2}
\]
for all $[([\alpha_{M'}], t)]  \in \widetilde{R}_{x_{0}, \asym} / \{\pm1\}$.
\end{proposition}

\begin{proof}
By the definition of $\widetilde{\ort}(V_{x_{0}})$, if such an extension $\widetilde{\sn}$ exists, it is unique. In order to prove existence, we define the character $\sn_{\zeta}$ of the abelian group
\[A_\zeta \coloneqq  
\Bigl \langle
\zeta_{[([\alpha_{M'}], t)] }
\Bigr. \,\Bigl|\,
[([\alpha_{M'}], t)]  \in \widetilde{R}_{x_{0}, \asym} / \{\pm1\}
\Bigr\rangle
\]
to be the unique character that satisfies
\[
\sn_{\zeta}\left(
\zeta_{[([\alpha_{M'}], t)]}
\right)
= \sqrt{\smash[b]{\sgn_{\ff}}(-1)}^{\dim_{\ff}\left(
V_{(x_{0}, [([\alpha_{M'}], t)])}(\ff)
\right) / 2}
\]
for all $[([\alpha_{M'}], t)]  \in \widetilde{R}_{x_{0}, \asym} / \{\pm1\}$.
Since $\zeta^{2} = -1$, Property~\ref{propertyofvarphi}, Equation~\eqref{plusminusdecomposition}, and \cite[Chapter~9, Example~3.5]{MR770063} imply that
\begin{align*}
\sgn_{\ff} \left(
\sn (\zeta^{2}_{[([\alpha_{M'}], t)]})
\right)
&=  \sgn_{\ff} \left(
(-1)^{\dim_{\ff}\left(
V_{(x_{0}, [([\alpha_{M'}], t)])}(\ff)
\right) / 2}
\right) \\
&= \left(
\sgn_{\ff}(-1)
\right)^{\dim_{\ff}\left(
V_{(x_{0}, [([\alpha_{M'}], t)])}(\ff)
\right) / 2} 
= \sn_{\zeta}\left(
\zeta^{2}_{[([\alpha_{M'}], t)]}
\right).
\end{align*}
Hence, we obtain that the restrictions of $\sgn_{\ff} \circ \sn$ and $\sn_{\zeta}$ to the group
\begin{align*}
&\quad 
\GL^+(V_{x_{0}}) \cap \ort(V_{x_{0}}, \varphi)(\ff) \cap
A_\zeta 
= \Bigl \langle
\zeta^{2}_{[([\alpha_{M'}], t)] }
\Bigr. \,\Bigl|\,
[([\alpha_{M'}], t)]  \in \widetilde{R}_{x_{0}, \asym} / \{\pm1\}
\Bigr \rangle
\end{align*}
agree.

Let $g \in \GL^+(V_{x_{0}})$ and $[([\alpha_{M'}], t)]  \in \widetilde{R}_{x_{0}, \asym} / \{\pm1\}$.
If we write $g [([\alpha_{M'}], t)]$ for the element of $\widetilde{R}_{x_{0}, \asym} / \{\pm1\}$ such that
$
g V_{(x_{0}, [([\alpha_{M'}], t)])} = V_{(x_{0}, g [([\alpha_{M'}], t)])}
$, then the definition of $\zeta_{[([\alpha_{M'}], t)]}$ implies that
$
g \circ \zeta_{[([\alpha_{M'}], t)]} = \zeta_{g [([\alpha_{M'}], t)]} \circ g
$.
Hence, we see that the group $\GL^+(V_{x_{0}})$ normalizes the group $A_\zeta$ and the character $\sn_{\zeta}$.
Thus, we can define the character $\widetilde \sn$ of  $\widetilde{\ort}(V_{x_{0}})$ by
\(
\widetilde{\sn}(g \cdot h) = \sgn_{\ff} \left(
\sn(g)
\right) \cdot \sn_{\zeta}(h)
\)
for $g \in \GL^+(V_{x_{0}}) \cap \ort(V_{x_{0}}, \varphi)(\ff)$ and $h \in A_\zeta$.
\end{proof}
\begin{remark}
If $\sgn_{\ff}(-1) = 1$, then the character $\widetilde{\sn}$ is a quadratic character.
However, if $\sgn_{\ff}(-1) = - 1$, the character $\widetilde{\sn}$ takes values in $\mu_{4}$ and is not always a quadratic character. 
\end{remark}
\begin{definition} \label{definitionofepsilontilde}
	We define the character $_{G'}\widetilde\epsilon_{\sym, \ram, 0}$\index{notation-ky}{epsilontildesymram0@$_{G'}\widetilde\epsilon_{\sym, \ram, 0}$}
of $N_{G'}(M')(F)_{[x_{0}]_{M'}}$ as the composition
\[
N_{G'}(M')(F)_{[x_{0}]_{M'}}
\xrightarrow{n \mapsto o_{n}}
\widetilde{\ort}(V_{x_{0}}) 
\stackrel{\widetilde{\sn}}{\longrightarrow}
\mu_{4},
\]
and the character $\widetilde{\epsilon}^{G/G'}_{x_{0}} \colon N_{G'}(M')(F)_{[x_{0}]_{M'}} \longrightarrow \mu_{4}$\index{notation-ky}{epsilontildeGG'x@$\widetilde{\epsilon}^{G/G'}_{x_{0}}$}
by $\widetilde{\epsilon}^{G/G'}_{x_{0}} = {_{G'} \widetilde{\epsilon}^{\sym, \ram}_{s}} \cdot {_{G'}\widetilde\epsilon_{\sym, \ram, 0}}$.

\end{definition}
\begin{corollary}
\label{extensionepsilon13}
The characters $_{G'}\widetilde\epsilon_{\sym, \ram, 0}$ and $\widetilde{\epsilon}^{G/G'}_{x_{0}}$ are extensions of the characters
\[
\left(
_{G'} \epsilon_{\sym, \ram} \cdot {_{G'} \epsilon_{0}}
\right)\restriction_{N_{G'}(M')(F)_{x_{0}}}
\quad
\text{ and }
\quad
\epsilon^{G/G'}_{x_{0}}
\restriction_{N_{G'}(M')(F)_{x_{0}}}
\]
to the group $N_{G'}(M')(F)_{[x_{0}]_{M'}}$.
\end{corollary}
\begin{proof}
	This follows from Lemma~\ref{rewritethetwopiecesofepsion}, Proposition \ref{Proposition-extension-epsilon13}, Lemma~\ref{extensionepsilon2} and the definition $\epsilon^{G/G'}_{x_{0}}={_{G'} \epsilon_{\sym, \ram}} \cdot {_{G'}\epsilon}^{\sym, \ram}_{s} \cdot {_{G'} \epsilon_{0}}$.
\end{proof}

\subsection{Extension of the quadratic character for general reductive groups}
\label{subsec:extension-general}

Now we drop the condition that $G$ is adjoint, and we assume (as before) that there exists a $G_{\ad}$-good element $X \in \Lie^{*}(G'_{\scn, \ab})(F)_{-r}$ of depth $-r$ in the sense of \cite[Section~3]{FKS}. This is, for example, implied by the existence of a $G$-generic character of $G'(F)$ of depth $r$ relative to $x_0$ by \cite[Remark~4.1.3]{FKS}.
Then we obtain from \cite[Lemma~4.1.2]{FKS} the quadratic character\index{notation-ky}{epsilonGG'x@$\epsilon^{G/G'}_{x_0}$}%
\[
\epsilon^{G/G'}_{x_0} \colon G'(F)_{[x_0]_{G}} \rightarrow \{ \pm 1 \}, 
\]
\index{notation-ky}{epsilonGG'x@$\epsilon^{G/G'}_{x}$}%
which is defined as the precomposition of $\epsilon^{G_{\ad}/G'_{\ad}}_{x_{0}}$ with the map $G'(F)_{[x_{0}]_{G}} \rightarrow G'_{\ad}(F)_{[x_{0}]_{G}}$ that is the restriction of the map $G(F) \rightarrow G_{\ad}(F)$ arising from the adjoint quotient $G \twoheadrightarrow G_{\ad}$. Here  $[x_{0}]_{G} \in \cB^{\red}(G,F) = \cB(G_{\ad}, F)$.

We note that the diagram $\{\iota\}$ from page \pageref{pageiota} induces a commutative diagram 
\[
\xymatrix{
	\cB(G'_{\ad}, F) \ar@{^{(}->}[r] \ar@{}[dr]|\circlearrowleft & \cB(G_{\ad}, F)
	\\
	\cB(M'_{\ad}, F) \ar@{^{(}->}[u] \ar@{^{(}->}[r] & \cB(M_{\ad}, F) \ar@{^{(}->}[u]
}
\]
of admissible embeddings, the point $[x_{0}]_{G}$ lies in the image of $\cB(M'_{\ad}, F)$,
and the embedding $\cB(M_{\ad}, F) \longrightarrow \cB(G_{\ad}, F)$ is $\frac{r}{2}$-generic relative to $[x_{0}]_{G}$.
Hence we obtain from Definition \ref{definitionofepsilontilde} the character 
 $\widetilde{\epsilon}^{G_{\ad}/G'_{\ad}}_{[x_{0}]_{G}} \colon N_{G'_{\ad}}(M'_{\ad})(F)_{[x_{0}]_{M'}} \longrightarrow \mu_{4}$.
 \begin{definition}\label{definitiontildeepsilon}
 	We define $\widetilde{\epsilon}^{G/G'}_{x_{0}}$ to be the composition	
 	\[
 	\widetilde{\epsilon}^{G/G'}_{x_{0}} \colon N_{G'}(M')(F)_{[x_{0}]_{M'}}  \longrightarrow N_{G'_{\ad}}(M'_{\ad})(F)_{[x_{0}]_{M'}} \longrightarrow \mu_{4},
 	\]\index{notation-ky}{epsilontildeGG'x@$\widetilde{\epsilon}^{G/G'}_{x_{0}}$}%
 	where the first map is the restriction of $G(F) \rightarrow G_{\ad}(F)$ and the second is $\widetilde{\epsilon}^{G_{\ad}/G'_{\ad}}_{[x_{0}]_{G}}$.
 \end{definition}

\begin{theorem}
	\label{twistextension}
	The restrictions of the characters $\epsilon^{G/G'}_{x_{0}}$ and $\widetilde{\epsilon}^{G/G'}_{x_{0}}$ to $N_{G'}(M')(F)_{[x_{0}]_{G}}$ agree.
\end{theorem}
\begin{proof}
This follows from the definitions of $\epsilon^{G/G'}_{x_{0}}$ and $\widetilde{\epsilon}^{G/G'}_{x_{0}}$ and Corollary \ref{extensionepsilon13}.
\end{proof}


\section{The non-twisted and twisted Heisenberg--Weil constructions}
\label{sec:twistedweil}

In this section, we will construct a compact, open subgroup $K_{x}$ of $G(F)$ and an irreducible smooth representation $\kappa_{x}$ of $K_{x}$ from a Heisenberg--Weil datum that will be introduced in Definition~\ref{definition:Heisenberg-Weil}. Here $x$ denotes a point of the Bruhat--Tits building of (a twisted Levi subgroup of) $G$.
These representations will play the roles of $\kappa_{M}$ in
Axiom~\ref{HAI-axiomaboutKM0vsKM} of \cite{HAI}
and $\kappa_{x}$ in
Axioms~\ref{HAI-axiomaboutK0vsK}
and \ref{HAI-axiomextensionoftheinductionofkappa}
of \cite{HAI}
in the setting of this paper.
To construct $\kappa_{x}$, we will follow the construction of \cite[Section~4]{Yu} that uses the theory of Heisenberg--Weil representations, but incorporating a quadratic twist introduced by \cite[Section~4]{FKS} and allowing more general coefficient fields using \cite[Section~2.3]{MR4460255}. 
Key new results in this section include: \spacingatend{}
\begin{enumerate}[(1)]
\item 
\label{enumeratecomparisonresultofkappax}
We will prove a comparison result between $\kappa_{x}$ and $\kappa_{y}$ for $x$ and $y$ sufficiently close to each other, see Proposition~\ref{propositioninductiontwistedsequencever}.
The quadratic twist is essential for this comparison.
The result will be used to verify
\cite[Axiom~\ref{HAI-axiomextensionoftheinductionofkappa}\eqref{HAI-axiomextensionoftheinductionofkappacompatibilitywiththecompactind}]{HAI}
in the setting of this paper, see Lemma \ref{proofofaxiomextensionoftheinductionofkappa}. 
\item
\label{enumerateextensionresultofkappaM}
We will prove an extension result for the (non-twisted) Heisenberg--Weil representations, see Proposition~\ref{proposition:extensionofuntwistedWeilHeisenbergrepresentation}.  
This result combined with the extension result for the quadratic character of \cite{FKS}, Theorem \ref{twistextension}, yields an extension of $\kappa_M$, which will be used to verify
Axiom~\ref{HAI-axiomaboutKM0vsKM}\eqref{HAI-axiomaboutKM0vsKMaboutanextension}
of \cite{HAI},
in the setting of this paper, see Proposition \ref{propproofofaxiomaboutKM0vsKM}.
\end{enumerate}
We note that the parts concerning the construction of the representation $\kappa_{x}$ in Sections~\ref{sec:compact-open-subgroups}, \ref{subsec:untwistedconstruction}, and the first part of Section~\ref{subsec:repfromHeisenbergWeil} essentially replicate the procedure in \cite{Yu} in our more general setting. We have included these details for the convenience of the reader and to establish the relevant notation.

\subsection{Compact open subgroups from functions on the root system}
\label{sec:compact-open-subgroups}

Let $x \in \cB(G,F)$.
Let $T$ be a maximal torus of $G$ such that
the Galois splitting field $E$ of $T$ is tamely ramified over $F$
and such that $x \in \cA(G, T, E)$. 
According to the discussion in the beginning of \cite[Section~2]{Yu}, such a torus exists.

Given any function 
\[
f \colon \Phi(G,T) \cup \{0\} \longrightarrow \widetilde{\mathbb{R}},
\]
with $f(0) \ge 0$,
following Bruhat and Tits (\cite[6.4.42]{MR327923}) we define the compact, open subgroup $G(E)_{x,f}$ of $G(E)$ by
\[
G(E)_{x, f} =
\left\langle
U_{\alpha}(E)_{x, f(\alpha)} \mid \alpha \in \Phi(G, T) \cup \{0\}
\right\rangle, 
\]
where we use the convention that $U_0=T$ and $T(E)_{x, \infty}=\{1\}$.
If $f$ is \emph{concave} in the sense of \cite[6.4.3]{MR327923},
then according to \cite[6.4.9(ii),(iii)]{MR327923},
we have
\[
G(E)_{x,f} \cap U_\alpha(E) = U_\alpha(E)_{x,f(\alpha)}
\]
for all $\alpha \in \Phi(G,T) \cup \{0\}$.
If $x$ is fixed and $f$ is preserved by the actions of $\Gal(E/F)$
on $\cA(G, T,E)$ and on $\Phi(G,T)$,
then we will define the compact, open subgroup $G(F)_{x,f}$
of $G(F)$ by
\[
G(F)_{x,f} = G(F) \cap G(E)_{x,f}.
\]

\subsection{The non-twisted construction}
\label{subsec:untwistedconstruction}
In this subsection we recall Yu's construction of $\phi'_{i-1}$ \cite[Section~4]{Yu} but allowing more general coefficient fields using \cite[Section~2.3]{MR4460255}.
The representation $\phi'_{i-1}$ is attached to a twisted Levi subgroup $G' \subseteq G$ and a generic character thereof as we explain below. In Section \ref{A twist of the construction} we will twist this construction by a quadratic character, and in Section \ref{subsec:repfromHeisenbergWeil} we will attach a representation to a whole sequence of twisted Levi subgroups $G^0 \subseteq G^1 \subseteq \hdots \subseteq G^n$ and generic characters of those groups by inductively repeating the below construction.

Let $G'$ be a (not necessarily proper) twisted Levi subgroup of $G$
that splits over a
tamely ramified
field extension of $F$.
We fix an 
 admissible embedding
$
\cB(G', F) \hookrightarrow \cB(G, F)
$
and identify a point of $\cB(G', F)$ with its image in $\cB(G, F)$.
We fix a positive real number $r > 0$.
Let $x \in \cB(G', F)$.
Let $T$ be a maximal torus of $G'$ such that the splitting field $E$ of $T$ is tamely ramified over $F$ and such that $x \in \cA(G', T, E)$.

Following Yu (\cite[Section~9]{Yu}), but recording the point $x$ as index in the notation, we define the compact, open subgroups $J_{x}$ and $J_{x,+}$ of $G(F)$ 
by
\index{notation-ky}{Jx@$J_x$}%
\index{notation-ky}{Jx+@$J_{x,+}$}%
\[
J_x = G(F)_{x,f}
\qquad
\text{and}
\qquad
J_{x,+} = G(F)_{x,f^+}
\]
where $f$ and $f^+$ are the (concave) functions on $\Phi(G,T) \cup \{ 0\}$ given by
\[
f(\alpha) = 
\begin{cases}
r & \alpha \in \Phi(G',T) \cup \{0\}, \\
\frac{r}{2} & \text{otherwise},
\end{cases}
\qquad
\text{and}
\qquad
f^+(\alpha) = 
\begin{cases}
r & \alpha \in \Phi(G',T) \cup \{0\}, \\
\frac{r}{2}+ & \text{otherwise}.
\end{cases}
\]
As explained in \cite[Section~1, Section~2]{Yu}, the groups $J_{x}$ and $J_{x, +}$ are independent of the choice of a maximal torus $T$ of $G'$.

We also recall the following notation from \cite[Section~2]{Yu}
\[
(G', G)(F)_{x, (r+, \frac{r}{2}+)} = 
G(F)_{x,f^{++}},
\qquad
\text{where}
\qquad
f^{++}(\alpha) = 
\begin{cases}
r+ & \alpha \in \Phi(G',T) \cup \{0\}, \\
\frac{r}{2}+ & \text{otherwise}.
\end{cases}
\]
According to \cite[6.4.44]{MR327923}, we have
$
[J_{x}, J_{x}] \subset J_{x, +}
$.
Hence, $J_{x, +}$ is a normal subgroup of $J_{x}$, and the quotient group $J_{x}/J_{x, +}$ is an abelian group.
Moreover, the definition of the Moy--Prasad filtration subgroups implies that for all $j \in J_{x}$, we have $j^{p} \in J_{x, +}$.
Thus, $J_{x}/J_{x, +}$ is an abelian group of exponent $p$.
We regard $J_{x}/J_{x, +}$ as an $\mathbb{F}_{p}$-vector space. 

Let $K'_{x}$ be an open subgroup of $G'(F)_{[x]_{G}}$.
\index{notation-ky}{Kayp x@$K'_x$}%
We write $K'_{x, 0+} = K'_{x} \cap G'(F)_{x, 0+}$.
\index{notation-ky}{Kayp x 0+@$K'_{x,0+}$}%
According to \cite[Remark~3.5]{Yu}, the group $K'_{x}$ normalizes the groups $J_{x}$ and $J_{x, +}$.
Hence, we can define the open subgroup $K_{x}$ of $G(F)$ as
\index{notation-ky}{Kay x@$K_x$}%
$
K_{x} = K'_{x} \cdot J_{x}
$.

Let $\phi$ be a character of $G'(F)$ that is $G$-generic of depth $r$ relative to $x$ in the sense of \cite[Definition~3.5.2]{Fintzen-IHES}.
We will construct an irreducible smooth representation $\phi'_{x}$ of $K_{x}$ following \cite[Section~4, Section~11]{Yu} and \cite[Section~2.3]{MR4460255}.
\begin{remark}
In \cite{Yu}, the representation $\phi'_{x}$ was constructed in the case of $\Coeff = \bC$.
All the arguments there can be applied to the general cases except for the existence of the Heisenberg--Weil representation.
We will replace the complex Heisenberg--Weil representation with the mod-$\ell$ Heisenberg--Weil representation constructed in \cite[Section~2.3]{MR4460255} if the coefficient field $\Coeff$ has positive characteristic (see Notation~\ref{notationcomplexandmodellweilheisenberg} below). 
\end{remark}
Let $\widehat{\phi}_{x}$ denote the character of $\left(
G'(F)_{[x]_{G}}
\right) \cdot J_{x, +}$ defined as in \cite[Section~4]{Yu}.
This means $\widehat{\phi}_{x}$ is the character of $\left(
G'(F)_{[x]_{G}}
\right) \cdot J_{x, +}$ that is trivial on $(G', G)(F)_{x, (r+ , \frac{r}{2}+)}$ and agrees with $\phi$ on $G'(F)_{[x]_{G}}$.
We write $N_{x}$ for the kernel of $\widehat{\phi}_{x} \restriction_{J_{x, +}}$.

Since $[J_{x}, J_{x}] \subset J_{x, +}$, we can define the pairing
$
J_{x} \times J_{x} \longrightarrow \Coeff^{\times}
$
by
$
(j_{1}, j_{2}) \mapsto \widehat{\phi}_{x}(j_{1} j_{2} j_{1}^{-1} j_{2}^{-1})
$.
Moreover, according to \cite[Lemma~11.1]{Yu}, the pairing above
induces a non-degenerate symplectic pairing
\index{notation-ky}{<,>x@$\langle \phantom{x},\phantom{x} \rangle_{x}$}%
\[
\langle \phantom{x},\phantom{x} \rangle_{x} \colon J_{x}/J_{x, +} \times J_{x}/J_{x, +} \longrightarrow \mu_{p},
\]
where 
$
\mu_{p} = \{
\zeta \in \Coeff \mid \zeta^{p} = 1
\}
$.
We fix an isomorphism
\index{notation-ky}{iotap@$\iota_p$}%
$
\iota_{p} \colon \mu_{p} \isoarrow \mathbb{F}_{p}
$
and regard $\langle \phantom{x},\phantom{x} \rangle_{x}$ as an $\mathbb{F}_{p}$-valued pairing.
Thus, we can regard $J_{x}/J_{x, +}$ as a symplectic space over $\mathbb{F}_{p}$.
Let $(J_{x}/J_{x, +})^{\#}$ denote the Heisenberg group of $J_{x}/J_{x, +}$, that is, $(J_{x}/J_{x, +})^{\#}$ is the set $(J_{x}/J_{x, +}) \times \mathbb{F}_{p}$ equipped with the group law
\[
(v, a) \cdot (w, b) = \left(
v + w, a + b + \frac{\left\langle v, w \right\rangle_{x}}{2}
\right).
\]
The symplectic group $\Sp(J_{x}/J_{x, +})$ acts on $(J_{x}/J_{x, +})^{\#}$ as
$
g (v, a) = (g v, a)
$
for $g \in \Sp(J_{x}/J_{x, +})$ and $(v, a) \in (J_{x}/J_{x, +})^{\#}$.  
\begin{notation}
\label{notationcomplexandmodellweilheisenberg}
We define the irreducible $\Coeff$-representation $\omega_{x}$ of the group $
\Sp \left(J_{x}/J_{x, +}\right) \ltimes (J_{x}/J_{x, +})^{\#}
$ such that the center
$
\{
(0, a) \mid a \in \mathbb{F}_{p}
\}
$
of $(J_{x}/J_{x, +})^{\#}$ acts by the character
$
(0, a) \mapsto \iota_{p}^{-1} (a)
$ as follows.%
\index{notation-ky}{omegax@$\omega_{x}$}%
\begin{enumerate}
\item
Suppose that the coefficient field $\Coeff$ has characteristic zero. Then let $\omega_{x}$ denote the complex Heisenberg--Weil representation obtained from \cite[Section~2]{MR460477}.
We note that although the representation $\omega_{x}$ is a priori defined over $\bC$, it is already defined over $\overline{\bQ}$ since the group $
\Sp \left(J_{x}/J_{x, +}\right) \ltimes (J_{x}/J_{x, +})^{\#}
$ is finite, and hence yields a representation with $\Coeff$-coefficients.
\item
Suppose that the coefficient field $\Coeff$ has characteristic $\ell > 0$. Then let $\omega_{x}$ denote the mod-$\ell$ Heisenberg--Weil representation obtained from \cite[Section~2.3]{MR4460255}.
\end{enumerate}
\end{notation}

We let 
$
j_{x} \colon J_{x}/N_{x} \longrightarrow (J_{x}/J_{x, +})^{\#}
$
be the isomorphism from \cite[Proposition~11.4]{Yu}, which satisfies the following properties. 
\begin{itemize}
\item For $j \in J_{x, +}$, we have
\(
j_{x}(j + N_{x}) = (0, \iota_{p} \circ \widehat{\phi}_{x}(j)).
\)
\item For $j \in J_{x}$, the first factor of 
\(
j_{x}(j + N_{x}) \in (J_{x}/J_{x, +})^{\#} = (J_{x}/J_{x, +}) \times \mathbb{F}_{p}
\)
is equal to
\(
j + J_{x, +} \in J_{x}/J_{x, +}.
\)
\end{itemize}
We use the same notation $j_{x}$
\index{notation-ky}{j_x@$j_x$}%
for its precomposition with the 
surjection 
$
J_{x} \twoheadrightarrow J_{x} / N_{x}
$.
Moreover, \cite[Proposition~11.4]{Yu} also implies that the conjugation action of $K'_{x}$ on $J_{x}$ induces a homomorphism
\[
f_{x} \colon K'_{x} \longrightarrow \Sp(J_{x}/J_{x, +}),
\]
and the maps $j_{x}$ and $f_{x}$ induce a group homomorphism
\begin{align*}
f_{x} \ltimes j_{x} \colon K'_{x} \ltimes J_{x} \longrightarrow \Sp \left(J_{x}/J_{x, +}\right) \ltimes (J_{x}/J_{x, +})^{\#}.
\end{align*}
Thus, following Yu, we can define the irreducible representation $\widetilde{\phi}_{x}$ of $K'_{x} \ltimes J_{x}$ as the pull back of $\omega_{x}$ via $f_{x} \ltimes j_{x}$.
The construction of $\widetilde{\phi}_{x}$ implies that the restriction of $\widetilde{\phi}_{x}$ to $1 \ltimes J_{x, +}$ is $1 \ltimes \widehat{\phi}_{x} \restriction_{J_{x, +}}$-isotypic.
On the other hand, according to \cite[6.4.44]{MR327923}, we have
$
[G'(F)_{x, 0+}, J_{x}] \subset J_{x, +}
$.
Hence, the image of $K'_{x, 0+}$ via $f_{x}$ is trivial, and the restriction of $\widetilde{\phi}_{x}$ to $K'_{x, 0+} \ltimes 1$ is $1$-isotypic.
We define the character $\inf(\phi_{x})$ of $K'_{x} \ltimes J_{x}$ as the inflation of the character $\phi_{x} \coloneqq  \phi \restriction_{K'_{x}}$ via the map
\[
K'_{x} \ltimes J_{x} \twoheadrightarrow (K'_{x} \ltimes J_{x}) / J_{x} \isoarrow K'_{x}.
\]
Since the restriction of $\widetilde{\phi}_{x}$ to $1 \ltimes J_{x, +}$ is $1 \ltimes \widehat{\phi}_{x} \restriction_{J_{x, +}}$-isotypic, the restriction of $\widetilde{\phi}_{x}$ to $K'_{x, 0+} \ltimes 1$ is $1$-isotypic, and since we have
\[
K'_{x} \cap J_{x} = K'_{x} \cap G'(F)_{x, r} = K'_{x, 0+} \cap J_{x, +},
\]
the set
\[
\{
k \ltimes k^{-1} \mid k \in K'_{x} \cap J_{x}
\} \subset K'_{x} \ltimes J_{x}
\]
is contained in the kernel of the representation $\inf(\phi_{x}) \otimes \widetilde{\phi}_{x}$ of $K'_{x} \ltimes J_{x}$ .
Thus, the representation $\inf(\phi_{x}) \otimes \widetilde{\phi}_{x}$ factors through the surjection
\(
K'_{x} \ltimes J_{x} \twoheadrightarrow K'_{x} \cdot J_{x} = K_{x}.
\)
We define the irreducible representation $\phi'_{x}$
\index{notation-ky}{phi'x@$\phi'_{x}$}%
of $K_{x}$ to be the representation whose inflation to $K'_{x} \ltimes J_{x}$ is $\inf(\phi_{x}) \otimes \widetilde{\phi}_{x}$.
\begin{remark}
\label{remarkinthecaseG'=G}
In the case $G' = G$, we have $K_{x} = K'_{x} \cdot G(F)_{x, r}$ and $\phi'_{x} = \phi \restriction_{K_{x}}$.
\end{remark}

We record some properties of $\phi'_{x}$.
\begin{lemma}
\label{lemmaphiisotypic}
The restriction of $\phi'_{x}$ to $J_{x}$ is irreducible, and the restriction of $\phi'_{x}$ to the group $K'_{x, 0+} \cdot J_{x, +}$ is $\widehat{\phi}_{x} \restriction_{K'_{x, 0+} \cdot  J_{x, +}}$-isotypic.
In particular, the restriction of $\phi'_{x}$ to $G(F)_{x, r+}$ is trivial.
\end{lemma}
\begin{proof}
The first claim follows from the fact that the restriction of $\phi'_{x}$ to $J_{x}$ is the pull back of the Heisenberg representation of $(J_{x}/J_{x, +})^{\#}$ via the surjection $j_{x}$ (see also \cite[Lemma~2.5]{MR4460255} for the positive characteristic case).
The construction of $\phi'_{x}$ implies that the restriction of $\phi'_{x}$ to $J_{x, +}$ is $\widehat{\phi}_{x} \restriction_{J_{x, +}}$-isotypic, and the restriction of $\phi'_{x}$ to $K'_{x, 0+}$ is $\phi \restriction_{K'_{x, 0+}} = \widehat{\phi}_{x} \restriction_{K'_{x, 0+}}$-isotypic.
Thus, we obtain the second claim.
The last claim follows from the second claim and the fact that the character $\widehat{\phi}_{x}$ is trivial on the group $G(F)_{x, r+} \subset (G', G)(F)_{x, (r+ , \frac{r}{2}+)}$.
\end{proof}

\begin{lemma}
\label{lemma:phi'xisunitary}
Suppose that $\Coeff$ admits a nontrivial involution,
with respect to which the character $\phi_{x}$ of $K'_{x}$ is unitary.
Then the representation $\phi'_{x}$ is unitary.
\end{lemma}
\begin{proof}
Since the representation $\widetilde{\phi}_{x}$ factors through the map $f_{x} \ltimes j_{x}$, and the group $\Sp \left(J_{x}/J_{x, +}\right) \ltimes (J_{x}/J_{x, +})^{\#}$ is finite, the representation $\widetilde{\phi}_{x}$ is unitary.
Then the lemma follows from the definition of $\phi'_{x}$ and the fact that the character $\phi_{x}$ is unitary.
\end{proof}

\subsection{A relative construction}
\label{subsection: relative case}

In this subsection, we will construct an irreducible smooth representation $(\phi')^{y}_{x}$ of an open subgroup $K^{y}_{x}$ of $G(F)$ from two points $x, y \in \cB(G', F)$. We will use them to compare the representations $\phi'_x$ and $\phi'_y$ for $x$ and $y$ sufficiently close in Section \ref{Section-compatibility-with-compact-induction}, see Corollary \ref{corollaryinductionuntwisted}.
For this, we first prove that the genericity condition of a character $\phi$ of $G'(F)$ does not depend on the choice of points.  \begin{lemma}
\label{lemmaindependenceofthegenericity}
Let $r > 0$ and $\phi$ be a character of $G'(F)$ that is $G$-generic of depth $r$ relative to a point $x \in \cB(G', F)$.
Then the character $\phi$ is $G$-generic of depth $r$ relative to $y$ for any point $y \in \cB(G', F)$.
\end{lemma}
\begin{proof}
We fix a maximal split torus $S$ of $G'$ such that $x, y \in \mathcal{A}(G', S, F)$.
Since the character $\phi$ is $G$-generic of depth $r$ relative to the point $x$, the restriction of $\phi$ to the group $G'(F)_{x, r+}$ is trivial, and there exists an element $X^{*} \in \Lie^{*}(G')^{G'}(F)$ which is $G$-generic of depth $-r$ in the sense of \cite[Definition~3.5.2]{Fintzen-IHES} such that the restriction of $\phi$ to 
\[
G'(F)_{x, r}/G'(F)_{x, r+} \simeq \Lie(G')(F)_{x, r} / \Lie(G')(F)_{x, r+} 
\]
is given by $\Psi \circ X^{*}$.
To prove the lemma, it suffices to show that the character $\phi$ is trivial on $G'(F)_{y, r+}$ and the restriction of $\phi$ to 
\[
G'(F)_{y, r}/G'(F)_{y, r+} \simeq \Lie(G')(F)_{y, r} / \Lie(G')(F)_{y, r+} 
\]
is also given by $\Psi \circ X^{*}$.
We fix $U \in \mathcal{U}(Z_{G'}(S))$.
According to \cite[4.3~Proposition (a)]{Kim-Yu}, we have
\[
G'(F)_{y, r+} = \left(
G'(F)_{y, r+} \cap U(F)
\right) \cdot (G'(F)_{y, r+} \cap Z_{G'}(S)(F)) \cdot \left(
G'(F)_{y, r+} \cap \overline{U}(F)
\right).
\]
Since $\phi$ is a character of $G'(F)$, it is trivial on the groups $U(F)$ and $\overline{U}(F)$.
Moreover, since $x, y \in \mathcal{A}(G', S, F)$,
from \cite[Proposition 1.9.1]{adler-thesis},
we have
\[
G'(F)_{y, r+} \cap Z_{G'}(S)(F) = Z_{G'}(S)(F)_{y, r+} = Z_{G'}(S)(F)_{x, r+} \subset G'(F)_{x, r+}.
\]
Hence, the character $\phi$ is also trivial on the group $G'(F)_{y, r+} \cap Z_{G'}(S)(F)$.
Thus, we have proved that the character $\phi$ is trivial on the group $G'(F)_{y, r+}$.
Next, we will prove that the restriction of $\phi$ to the group $G'(F)_{y, r}/G'(F)_{y, r+}$ is given by $\Psi \circ X^{*}$.
By using \cite[4.3~Proposition (a)]{Kim-Yu} again, we obtain that
\[
G'(F)_{y, r} = \left(
G'(F)_{y, r} \cap U(F)
\right) \cdot (G'(F)_{y, r} \cap Z_{G'}(S)(F)) \cdot \left(
G'(F)_{y, r} \cap \overline{U}(F)
\right).
\]
Since $X^{*}$ is fixed by the coadjoint action of $G'$ (hence of $S$) on $\Lie^{*}(G')$, we obtain that $X^{*}$ is trivial on the Lie algebras of $U$ and $\overline{U}$.
Since the character $\phi$ is trivial on the groups $U(F)$ and $\overline{U}(F)$, it suffices to prove that the restriction of $\phi$ to the group $Z_{G'}(S)(F)_{y, r} / Z_{G'}(S)(F)_{y, r+}$ is given by $\Psi \circ X^{*}$.
Now, the claim follows from the facts that 
\[
Z_{G'}(S)(F)_{y, r} / Z_{G'}(S)(F)_{y, r+} = Z_{G'}(S)(F)_{x, r} / Z_{G'}(S)(F)_{x, r+} 
\]
and the restriction of $\phi$ to the group $G'(F)_{x, r}/G'(F)_{x, r+}$ is given by $\Psi \circ X^{*}$.
\end{proof}
Let $x, y \in \nobreak \cB(G', F)$. We use the notation from the previous section, i.e., we fix a positive real number $r > 0$, and  $\phi$ is a character of $G'(F)$ that is $G$-generic of depth $r$ relative to $x$, and hence also relative to $y$.
\begin{lemma}
\label{lemmacompatibilityofcharacters}
We have
\[
\widehat{\phi}_{x}\restriction_{J_{x, +} \cap J_{y, +}} = \widehat{\phi}_{y}\restriction_{J_{x, +} \cap J_{y, +}}.
\]
\end{lemma}
\begin{proof}
Let $z \in \{x, y\}$. Then
$\widehat{\phi}_{z}\restriction_{J_{z, +}}$ is given by composing the map	
	\begin{eqnarray*} 
	J_{z,+} \twoheadrightarrow  J_{z,+}/G(F)_{z,r+} & \simeq & (\Lie(G')(F)_{z,r}\oplus \fr(F)_{z,\frac{r}{2}+})/(\Lie(G)(F)_{z,r+}) \\
	& \twoheadrightarrow & \Lie(G')(F)_{z,r}/\Lie(G')(F)_{z,r+} \simeq 
	G'(F)_{z,r}/G'(F)_{z,r+}
	\end{eqnarray*}
with the restriction $\phi\restriction_{G'(F)_{z,r}}$ (that factors through $G'(F)_{z,r}/G'(F)_{z,r+}$), where the subspace $\fr(F)\subset \Lie(G)(F)$ is defined by
$\fr(F)= \Lie(G)(F) \cap \bigoplus_{\alpha \in \Phi(G,T) \smallsetminus \Phi(G', T)} \Lie(G)(E)_\alpha$ for some maximal torus $T$ of $G'$ that splits over a tamely ramified extension $E$ and for which $x, y \in \cA(G', T, E)$. 
Now the claim follows from the equation $\widehat{\phi}_{x}\restriction_{G'(F)_{x,r} \cap G'(F)_{y,r}} = \widehat{\phi}_{y}\restriction_{G'(F)_{x,r} \cap G'(F)_{y,r}} = \phi\restriction_{G'(F)_{x,r} \cap G'(F)_{y,r}}$.
\end{proof}

As before, let $K'_{x}$, resp.\ $K'_{y}$, be an open subgroup of $G'(F)_{[x]_{G}}$, resp.\ $G'(F)_{[y]_{G}}$.
\index{notation-ky}{Kayp x@$K'_x$}%
We recall our notation $K_{x} = K'_{x} \cdot J_{x}$ and $K_{y} = K'_{y} \cdot J_{y}$.
We also write%
\index{notation-ky}{Kayp x y@$K'_{x, y}$}%
\index{notation-ky}{Jxy@$J^y_{x}$}%
\index{notation-ky}{Kayy x@$K^y_x$}%
\[
\begin{cases}
K'_{x, y} &= K'_{x} \cap K'_{y}, \\
J^{y}_{x} &= \left(
J_{x} \cap J_{y}
\right) \cdot J_{y, +}, \\
K^{y}_{x} &= K'_{x, y} \cdot J^{y}_{x}.
\end{cases}
\]
We will construct an irreducible representation $(\phi')^{y}_{x}$ of $K^{y}_{x}$ below.
We define $\mathbb{F}_{p}$-subspaces
$\Vyx$ and $\Vyxp$ of $J_{y} / J_{y, +}$ as
\index{notation-ky}{Vyx@$\Vyx$}%
\index{notation-ky}{Vyxp@$\Vyxp$}%
$
\Vyx
=J^y_x/J_{y,+}
 = \left(
J_{x} \cap J_{y}
\right) \cdot J_{y, +}/J_{y, +}
$
and
$
\Vyxp  =
\left(
J_{x, +} \cap J_{y}
\right) \cdot J_{y, +}/J_{y, +}
$.
(Note that our group $\Vyxp$ is denoted $U^{y}_{x}$ in \cite{FKS}.) The following lemma is a generalization of a result of Yu to our more general setting.
\begin{lemma}[\cite{Yu}]
\label{lemmatotallyisotropic}
The $\mathbb{F}_{p}$-subspace $\Vyxp $ is totally isotropic, and the space $\Vyx$ is the orthogonal complement of $\Vyxp $ in $J_{y}/J_{y, +}$ with respect to $\langle \phantom{x},\phantom{x} \rangle_{y}$.
\end{lemma}
\begin{proof}
	The proof is analogous to Yu's considerations in \cite[Section~12 and 13]{Yu} where he treats the case that $y$ is in the $G'(F)$-orbit of $x$. For the convenience of the interested reader we have spelt out more details in Appendix \ref{appendixYufortwopoints}.
\end{proof}

Since $\Vyx$ is the orthogonal complement of the totally isotropic subspace $\Vyxp $ in $J_{y}/J_{y, +}$, the symplectic pairing $\langle \phantom{x},\phantom{x} \rangle_{y}$ induces a non-degenerate symplectic pairing
on $\Vyx/\Vyxp $.
We use the same notation $\langle \phantom{x},\phantom{x} \rangle_{y}$ for this pairing on $\Vyx/\Vyxp $.
Hence, we can define the Heisenberg--Weil representation $\omega^{y}_{x}$ of 
$
\Sp\left(
\Vyx/\Vyxp 
\right) \ltimes \left(
\Vyx/\Vyxp 
\right)^{\#}
$
associated with the central character
 $
 (0, a) \mapsto \iota_{p}^{-1}(a)
 $
 for 
 \(
 (0, a) \in \{0\} \times \mathbb{F}_{p} \subset \left(
\Vyx/\Vyxp 
\right)^{\#}
 \) as in Notation \ref{notationcomplexandmodellweilheisenberg}.
 We define the subgroup $(\Vyx)^{\#}$ of $(J_{y}/J_{y, +})^{\#}$
\index{notation-ky}{Vxysharp@$(\Vyx)^{\#}$}%
by
\spacingatend{}
 \[
 (\Vyx)^{\#} = \left\{
 (v, a) \in (J_{y}/J_{y, +})^{\#} \mid v \in \Vyx
 \right\}.
 \]
 Then we can define the surjective homomorphism
 \begin{align}
 \label{surjectionheisenberg}
 (\Vyx)^{\#} \twoheadrightarrow \left(
\Vyx/\Vyxp 
\right)^{\#}
 \end{align}
 by
 $
 (v, a) \mapsto (v+ \Vyxp , a)
 $.
 We define the subgroup
\index{notation-ky}{Pyx@$\Pyx$}%
$\Pyx$ of $\Sp(J_{y}/J_{y, +})$ as
 \[
 \Pyx = \{
 g \in \Sp(J_{y}/J_{y, +}) \mid g \Vyxp  \subset \Vyxp 
 \}.
 \]
 Since every element $g \in \Pyx$ preserves $\Vyxp $ and hence also its orthogonal complement $\Vyx$, restriction of $g$ to $\Vyx \subset J_y/J_{y,+}$ yields the surjection
 \begin{align}
 \label{surjectionsymplectic}
 \Pyx \twoheadrightarrow \Sp\left(
\Vyx/\Vyxp 
\right).
\end{align}
 Combining \eqref{surjectionheisenberg} with \eqref{surjectionsymplectic}, we obtain a surjective homomorphism
 \begin{align}
 \label{pullbackxy}
 \Pyx \ltimes (\Vyx)^{\#} \twoheadrightarrow \Sp\left(
\Vyx/\Vyxp 
\right) \ltimes \left(
\Vyx/\Vyxp 
\right)^{\#},
 \end{align}
 and we can pull back $\omega^{y}_{x}$ to $\Pyx \ltimes (\Vyx)^{\#}$ via \eqref{pullbackxy} and denote the resulting representation of $\Pyx \ltimes (\Vyx)^{\#}$ also by $\omega^{y}_{x}$.
\index{notation-ky}{omegayx@$\omega^y_x$}%
 Moreover, since the image of $K'_{x, y} \ltimes J^{y}_{x}$ via $f_{y} \ltimes j_{y}$ is contained in the group $\Pyx \ltimes (\Vyx)^{\#}$, we can pull back $\omega^{y}_{x}$ to a representation $\widetilde{\phi}^{y}_{x}$ of $K'_{x, y} \ltimes J^{y}_{x}$.
 Let $\inf(\phi_{x, y})$ denote the character of $K'_{x, y} \ltimes J^{y}_{x}$ obtained by the inflation of the character $\phi_{x, y} \coloneqq  \phi \restriction_{K'_{x, y}}$ via the map
 \[
K'_{x, y} \ltimes J^{y}_{x} \longrightarrow (K'_{x, y} \ltimes J^{y}_{x}) / J^{y}_{x} \simeq K'_{x, y}.
 \]
Then by using the same argument as in the construction of $\phi'_{x}$ in Section~\ref{subsec:untwistedconstruction}, we see that the representation $\inf(\phi_{x, y}) \otimes \widetilde{\phi}^{y}_{x}$ factors through the surjection
\[
K'_{x, y} \ltimes J^{y}_{x} \twoheadrightarrow K'_{x, y} \cdot J^{y}_{x} = K^{y}_{x}.
\]
We define $(\phi')^{y}_{x}$ as the representation of $K^{y}_{x}$ whose inflation to $K'_{x, y} \ltimes J^{y}_{x} $ is $\inf(\phi_{x, y}) \otimes \widetilde{\phi}^{y}_{x}$.
\subsection{Compatibility with respect to compact induction} \label{Section-compatibility-with-compact-induction}
We keep the notation from the previous subsections.
In this subsection, we will study the relationship between the constructions of representations in Section~\ref{subsec:untwistedconstruction} and Section~\ref{subsection: relative case} with the goal to to compare the representations $\phi'_x$ and $\phi'_y$ for $x$ and $y$ sufficiently close to each other, see Corollary \ref{corollaryinductionuntwisted}.

We define the character $\chi^{\Vyxp }$ of $\Pyx$ by
\[
\chi^{\Vyxp }(g) = \sgn_{\mathbb{F}_{p}}\left(
\det\nolimits_{\mathbb{F}_{p}}(g\restriction_{\Vyxp })
\right) \quad \text{ for } g \in \Pyx ,
\]
where $\det\nolimits_{\mathbb{F}_{p}}(g\restriction_{\Vyxp })$ denotes the determinant of the the $\mathbb{F}_{p}$-linear map 
\(
g\restriction_{\Vyxp } \colon \Vyxp  \longrightarrow \Vyxp ,
\)
and
$
\sgn_{\mathbb{F}_{p}} \colon \mathbb{F}_{p}^{\times} \longrightarrow \{ \pm 1 \}
$
denotes the unique non-trivial quadratic character of $\mathbb{F}_{p}^{\times}$.

Note that
$
K'_{x, y} \cap J^{y}_{x}
\subset K'_{y} \cap J_{y} \subset K'_{y, 0+}
$
is contained in the kernel of the map
\(
f_{y}\restriction_{K'_{x, y}} \colon 
K'_{x, y} \longrightarrow \Pyx \subset \Sp(J_{y}/J_{y, +})
\),
 hence we can introduce the following notation.
\begin{notation}
Following \cite[Definition~4.1.1]{FKS} we denote by
$\delta^{y}_{x}$ the character of $K'_{x, y}$ that is the precomposition of the character $\chi^{\Vyxp }$ with the map
\( f_{y} \).
We also denote by $\delta^{y}_{x}$
\index{notation-ky}{deltaxy@$\delta_y^x$}%
the inflation of this character to  $K^{y}_{x}$ via the surjection
\(
K^{y}_{x} = K'_{x, y} \cdot J^{y}_{x} \twoheadrightarrow K'_{x, y} \cdot J^{y}_{x}/J^{y}_{x} \isoarrow K'_{x, y}/(K'_{x, y} \cap J^{y}_{x}).
\)
\end{notation}
\begin{proposition}
\label{propuntwistedinduction}
We have an isomorphism
\[
\phi'_{y}\restriction_{K'_{x, y} \cdot J_{y}}
\isoarrow
\ind_{K^{y}_{x}}^{K'_{x, y} \cdot J_{y}}
\left(
(\phi')^{y}_{x} \otimes \delta^{y}_{x}
\right).
\]
\end{proposition}
To prove Proposition~\ref{propuntwistedinduction}, we prepare a general lemma.
\begin{lemma}
\label{lemmainduction}
Let $N$ be a locally profinite group and $N_{2} \subset N_{1}$ be open subgroups of $N$.
Let $K$ be a closed subgroup of $N$ that normalizes $N_{1}$ and $N_{2}$.
We suppose that
\[
N_{1} \cap KN_{2} = N_{2}.
\]
Let $\tau$ be a smooth representation of $K N_{2}$, and we write $\inf(\tau)$ for the inflation of $\tau$ to $K \ltimes N_{2}$ via the natural map
$
K \ltimes N_{2} \longrightarrow K N_{2}
$.
Then the representation $\ind_{K \ltimes N_{2}}^{K \ltimes N_{1}}(\inf(\tau))$ is isomorphic to the inflation of $\ind_{K N_{2}}^{K N_{1}} (\tau)$ to $K \ltimes N_{1}$ via the natural map
$
K \ltimes N_{1} \longrightarrow K N_{1}
$.
\end{lemma}
\begin{proof}
Using Mackey decomposition and $N_{1} \cap KN_{2} = N_{2}$, we have 
\[
\ind_{K \ltimes N_{2}}^{K \ltimes N_{1}}(\inf(\tau))\restriction_{N_{1}} \simeq \ind_{N_{2}}^{N_{1}} (\tau\restriction_{N_{2}})
= \ind_{N_{1} \cap K N_{2}}^{N_{1}} (\tau\restriction_{N_{1} \cap K N_{2}})
\simeq \ind_{K N_{2}}^{K N_{1}} (\tau)\restriction_{N_{1}} .
\] 
Moreover, a straightforward calculation implies that the actions of $K$ on the representation spaces of these two representations agree.
\end{proof}
\begin{proof}[Proof of Proposition~\ref{propuntwistedinduction}] 
According to \cite[Theorem~2.4(b)]{MR460477}
and \cite[Lemma~3.2]{MR4460255} for the positive characteristic case, we have
\[
\omega_{y} \restriction_{\Pyx \ltimes (J_{y}/J_{y, +})^{\#}} \simeq \ind_{\Pyx \ltimes (\Vyx)^{\#}}^{\Pyx \ltimes (J_{y}/J_{y, +})^{\#}} \left(
\omega^{y}_{x} \otimes (\chi^{\Vyxp } \ltimes 1)
\right).
\]
(Note that the statement of \cite[Theorem~2.4(b)]{MR460477}
omits the character $\chi^{\Vyxp } \ltimes 1$,
but this is only a typographical error.
For more details, see Footnote 1 of \cite{MR4357723}).
Hence, the definitions of $\widetilde{\phi}_{y}$ and $\widetilde{\phi}^{y}_{x}$ and the observation that $j_y: J_y \longrightarrow (J_y/J_{y,+})^{\#}$ and $j_y: J_x^y \longrightarrow (\Vyx)^{\#}$ are surjective imply that
\[
\widetilde{\phi}_{y} \restriction_{K'_{x, y} \ltimes J_{y}} \simeq \ind_{K'_{x, y} \ltimes J^{y}_{x}}^{K'_{x, y} \ltimes J_{y}} \left(
\widetilde{\phi}^{y}_{x} \otimes (\delta^{y}_{x} \ltimes 1)
\right).
\]
Moreover, the definitions of $\inf(\phi_{y})$ and $\inf(\phi_{x, y})$ imply that 
$
\inf(\phi_{y})\restriction{
K'_{x, y} \ltimes J^{y}_{x}}= \inf(\phi_{x, y})
$.
Hence, we have
\[
\left(
\inf(\phi_{y}) \otimes
\widetilde{\phi}_{y}
\right)
\restriction_{K'_{x, y} \ltimes J_{y}} \simeq \ind_{K'_{x, y} \ltimes J^{y}_{x}}^{K'_{x, y} \ltimes J_{y}} \left(
\inf(\phi_{x, y}) \otimes
\widetilde{\phi}^{y}_{x} \otimes (\delta^{y}_{x} \ltimes 1)
\right).
\]
Since \spacingatend{also remove vertical space at end}
\begin{align*}
J_{y} \cap K'_{x, y} \cdot J^{y}_{x} &= J_{y} \cap \left(
(K'_{x} \cap K'_{y}) \cdot (J_{x} \cap J_{y}) \cdot J_{y, +}
\right)
= \left(
J_{y} \cap K'_{x} \cap K'_{y}
\right) \cdot (J_{x} \cap J_{y}) \cdot J_{y, +}\\
&= \left(
G'(F)_{y, r} \cap K'_{x} \cap K'_{y}
\right) \cdot (J_{x} \cap J_{y}) \cdot J_{y, +}
= \left(
J_{y, +} \cap K'_{x} \cap K'_{y}
\right) \cdot (J_{x} \cap J_{y}) \cdot J_{y, +}\\
&= (J_{x} \cap J_{y}) \cdot J_{y, +}
= J^{y}_{x},
\end{align*}
Lemma~\ref{lemmainduction} and the definitions of $\phi'_{y}$ and $(\phi')^{y}_{x}$ imply that
\[
\phi'_{y}\restriction_{K'_{x, y} \cdot J_{y}} \simeq 
\ind_{K^{y}_{x}}^{K'_{x, y} \cdot J_{y}} \left(
(\phi')^{y}_{x} \otimes \delta^{y}_{x}
\right).
\qedhere
\]
\end{proof}

Next, we consider the following special case that will be used in Corollary~\ref{corollaryinductionuntwisted}, in which we can compare  $\phi'_x$ and $\phi'_y$, and which, via Corollary~\ref{corollaryofpropositioninductiontwistedsequencever}, will be used in Lemma~\ref{proofofaxiomextensionoftheinductionofkappa} below to verify Axiom~\ref{HAI-axiomextensionoftheinductionofkappa} of \cite{HAI}
in the setting of this paper.
We suppose for the rest of this subsection that  
\begin{align}
\label{conditionsonxandXaboutK'}
G'(F)_{x, r}, G'(F)_{y, r} & \subseteq K'_{x} \subseteq K'_{y}
\end{align}
and
\begin{align}
\label{conditionsonxandXaboutJ}
G(F)_{y, \frac{r}{2}+} \subseteq G(F)_{x, \frac{r}{2}+} \subseteq G(F)_{x, \frac{r}{2}} \subseteq G(F)_{y, \frac{r}{2}}.
\end{align}
According to Condition \eqref{conditionsonxandXaboutK'}, we have $K'_{x} = K'_{x, y}$.
We can also prove $K_{x} = K^{y}_{x}$ as follows. 
\begin{lemma}
\label{lemmaKx=KXx}
We have \spacingatend{}
\[
K'_{x} \cdot J_{x} = K'_{x} \cdot \left(
J_{x} \cap J_{y}
\right) = K'_{x} \cdot \left(
J_{x} \cap J_{y}
\right) \cdot J_{y, +}.
\]
In particular, we have $K_{x} = K^{y}_{x}$.
\end{lemma}
\begin{proof}
\addtocounter{equation}{-1}
\begin{subequations}
According to Condition \eqref{conditionsonxandXaboutJ} and the definition of $J_{x}$ and $J_{y}$, we have
\begin{align}
\label{JxsubsetGxrJXandJy+subsetGyrJx+}
J_{x} \subseteq G'(F)_{x, r} \cdot J_{y}
\qquad 
\text{ and }
\qquad
J_{y, +} \subseteq G'(F)_{y, r} \cdot J_{x, +}.
\end{align}
According to the first inclusion in \eqref{JxsubsetGxrJXandJy+subsetGyrJx+}, we have
$
J_{x} = G'(F)_{x, r} \cdot \left(
J_{x} \cap J_{y}
\right)
$.
Hence, the first equality of the lemma follows from Condition \eqref{conditionsonxandXaboutK'}.

We will prove the second equality.
According to the second inclusion in \eqref{JxsubsetGxrJXandJy+subsetGyrJx+}, we have
$
J_{y, +} = G'(F)_{y, r} \cdot \left(
J_{x, +} \cap J_{y, +}
\right)
$.
Hence, using Condition \eqref{conditionsonxandXaboutK'}, we have
\[
K'_{x} \cdot \left(
J_{x} \cap J_{y}
\right) \cdot J_{y, +} =
K'_{x} \cdot J_{y, +} \cdot \left(
J_{x} \cap J_{y}
\right) 
= K'_{x} 
\cdot \left(
J_{x, +} \cap J_{y, +}
\right) \cdot
\left(
J_{x} \cap J_{y}
\right) 
= K'_{x} \cdot \left(
J_{x} \cap J_{y}
\right).
\]
The last claim follows from the calculation
\[
K_{x} = K'_{x} \cdot J_{x} 
= K'_{x} \cdot \left(
J_{x} \cap J_{y}
\right) \cdot J_{y, +} 
= K'_{x, y} \cdot \left(
J_{x} \cap J_{y}
\right) \cdot J_{y, +} 
= K^{y}_{x}.
\qedhere
\]
\end{subequations}
\end{proof}

Now, we have two representations $\phi'_{x}$ and $(\phi')^{y}_{x}$ of the group $K_{x} = K^{y}_{x}$.
We will prove in Proposition~\ref{propositionisomofPhixandPhiXx} below that these two representations are isomorphic.
To do this, we first prepare the following lemma.
\begin{lemma}
\label{lemmaaboutxandX}
\mbox{}

\begin{enumerate}[(1)]
\item
\label{lemmaaboutxandXJxandJXcapJx}
The inclusion $J_{x} \cap J_{y} \subseteq J_{x}$ induces an isomorphism
\[
\left(
J_{x} \cap J_{y}
\right) / \left(
J_{x, +} \cap J_{y}
\right) \isoarrow J_{x} / J_{x, +}.
\]
\item
\label{lemmaaboutxandXJX+smallerthanJx+}
The inclusion $J_{x} \cap J_{y} \subseteq \left(
J_{x} \cap J_{y}
\right) \cdot J_{y, +}$ induces an isomorphism
\[
\left(
J_{x} \cap J_{y}
\right) / \left(
J_{x, +} \cap J_{y}
\right) \isoarrow \left(
J_{x} \cap J_{y}
\right) \cdot J_{y, +} / \left(
J_{x, +} \cap J_{y}
\right) \cdot J_{y, +}.
\]
\end{enumerate}
\end{lemma}
\begin{proof}
According to Condition \eqref{conditionsonxandXaboutJ} and the definition of $J_{x}$ and $J_{y}$, we have
\[
J_{x} =G'(F)_{x, r} \cdot \left(
J_{x} \cap J_{y}
\right) 
\subseteq J_{x, +} \cdot \left(
J_{x} \cap J_{y}
\right)
\subseteq J_{x}.
\]
Thus, we obtain that $J_{x} = J_{x, +} \cdot \left(
J_{x} \cap J_{y}
\right)$, which implies the first claim of the lemma.

We will prove the second claim.
It suffices to show that
$
\left(
J_{x} \cap J_{y}
\right)
\cap
\left(\left(
J_{x, +} \cap J_{y}
\right) \cdot J_{y, +}\right) = 
J_{x, +} \cap J_{y}
$.
According to Condition \eqref{conditionsonxandXaboutJ} and the definition of $J_{x}$ and $J_{y}$, we have
$
J_{y, +} =
G'(F)_{y, r} \cdot \nobreak \left(
J_{x, +} \cap J_{y, +}
\right)
$.
Hence, we have 
\begin{align*}
J_{x} \cap J_{y, +} = 
J_{x} \cap \left(
G'(F)_{y, r} \cdot \left(
J_{x, +} \cap J_{y, +}
\right)\right) 
&= \left(
J_{x} \cap G'(F)_{y, r}
\right) \cdot \left(
J_{x, +} \cap J_{y, +}
\right) \\
&= 
\left(
G'(F)_{x, r}
\cap G'(F)_{y, r}
\right) \cdot \left(
J_{x, +} \cap J_{y, +}
\right) 
= J_{x, +} \cap J_{y, +}.
\end{align*}
Thus, we obtain that
\[
\left(
J_{x} \cap J_{y}
\right)
\cap
\left(\left(
J_{x, +} \cap J_{y}
\right) \cdot J_{y, +} \right)
= 
\left(
J_{x, +} \cap J_{y}
\right) \cdot \left(
J_{x} \cap J_{y, +}
\right) 
= 
\left(
J_{x, +} \cap J_{y}
\right) \cdot \left(
J_{x, +} \cap J_{y, +}
\right) 
=
J_{x, +} \cap J_{y}.
\qedhere
\]
\end{proof}
\begin{proposition}
\label{propositionisomofPhixandPhiXx}
We have an isomorphism  
$\phi'_{x} \isoarrow (\phi')^{y}_{x}$
as representations of $K_{x} = K^{y}_{x}$.
\end{proposition}
\begin{proof}
Combining the isomorphisms in Lemma~\ref{lemmaaboutxandX}, we have isomorphisms
\[
J_{x} / J_{x, +}  \isoarrowleft
\left(
J_{x} \cap J_{y}
\right) / \left(
J_{x, +} \cap J_{y}
\right) 
\isoarrow \left(
J_{x} \cap J_{y}
\right) \cdot J_{y, +} / \left(
J_{x, +} \cap J_{y}
\right) \cdot J_{y, +} 
 \isoarrow \Vyx / \Vyxp .
\]
We identify the space $J_{x} / J_{x, +}$ with the space $\Vyx / \Vyxp $ via this isomorphism.
According to Lemma~\ref{lemmacompatibilityofcharacters}, the symplectic pairing $\langle \phantom{x},\phantom{x} \rangle_{x}$ on $J_{x} / J_{x, +}$ agrees with the symplectic pairing $\langle \phantom{x},\phantom{x} \rangle_{y}$ on $\Vyx / \Vyxp $.
Hence, we obtain an isomorphism
$
\omega_{x} \restriction_{(J_{x} / J_{x, +})^{\#}} \isoarrow \omega^{y}_{x} \restriction_{(\Vyx / \Vyxp )^{\#}}
$.
Then the definition of $\phi'_{x}$ and $(\phi')^{y}_{x}$ implies that
$
\phi'_{x} \restriction_{J_{x} \cap J_{y}} \simeq (\phi')^{y}_{x} \restriction_{J_{x} \cap J_{y}}
$.
Moreover, since the isomorphisms in Lemma~\ref{lemmaaboutxandX} are compatible with the conjugate actions of $K'_{x}$, we also have
$
\phi'_{x} \restriction_{K'_{x} \cdot \left(
J_{x} \cap J_{y}
\right)} \simeq (\phi')^{y}_{x} \restriction_{K'_{x} \cdot \left(
J_{x} \cap J_{y}
\right)}
$.
Now, the lemma follows from Lemma~\ref{lemmaKx=KXx}.
\end{proof}
\begin{corollary}
\label{corollaryinductionuntwisted}
Let $x, y \in \cB(G', F)$ and $\phi$ be a character of $G'(F)$ that is $G$-generic of depth $r$ relative to $x$.
Let $K'_{x}$ (resp.\ $K'_{y}$) be an open subgroup of $G'(F)_{[x]_{G}}$ (resp.\ $G'(F) \cap G(F)_{[y]_{G}}$).
We suppose Conditions \eqref{conditionsonxandXaboutK'} and \eqref{conditionsonxandXaboutJ}.
Then we have an isomorphism
\[
\phi'_{y} \restriction_{K'_{x} \cdot J_{y}}
\isoarrow
\ind_{K_{x}}^{K'_{x} \cdot J_{y}}
\left(
\phi'_{x} \otimes \delta^{y}_{x}
\right).
\]
\end{corollary}
\begin{proof}
The corollary follows from Proposition~\ref{propuntwistedinduction} and Proposition~\ref{propositionisomofPhixandPhiXx}.
\end{proof}

\subsection{A twist of the construction} 
\label{A twist of the construction}

In this subsection,
we will twist the construction of $\phi'_{x}$ in Section~\ref{subsec:untwistedconstruction} by a quadratic character to define the irreducible smooth representation $\phi^{+}_{x}$ of $K_{x}$, see Definition \ref{definitionphiplus}.
Then the relation between the representations attached to two nearby points,
i.e.\ the analogue of Corollary~\ref{corollaryinductionuntwisted}, will be simpler, more precisely, it will no longer require an auxiliary character, see Corollary~\ref{corollaryinductiontwisted} below. 
For $x \in \cB(G', F)$, recall that  
\[
\epsilon^{G/G'}_{x} \colon G'(F)_{[x]_{G}} \longrightarrow \{ \pm 1 \}
\]
denotes the quadratic character of $G'(F)_{[x]_{G}}$ defined in \cite[Lemma~4.1.2]{FKS}, which is
trivial on the group $G'(F)_{x, 0+}$.
Since
$
K'_{x} \cap J_{x} \subseteq G'(F)_{x, 0+}
$,
we can inflate the restriction of $\epsilon^{G/G'}_{x}$
to $K_{x}$ via the map
\(
K_{x} = K'_{x} \cdot J_{x}
\longrightarrow K'_{x} \cdot J_{x}/J_{x}
\simeq K'_{x}/\left(
K'_{x} \cap J_{x}
\right), 
\) and we denote the resulting character again by $\epsilon^{G/G'}_{x}$.
\index{notation-ky}{epsilonGG'x@$\epsilon^{G/G'}_{x}$}%
The following lemma follows from \cite[Lemma~4.1.2]{FKS}.
\begin{lemma}[\cite{FKS}]
\label{lemmathatfollowsfromFKS412}
For all $x, y \in \cB(G', F)$, we have
\[
\left(
\epsilon^{G/G'}_{x} \cdot \delta^{x}_{y}
\right)\restriction_{K^{y}_{x} \cap K^{x}_{y}} = 
\left(
\epsilon^{G/G'}_{y} \cdot \delta^{y}_{x}
\right)\restriction_{K^{y}_{x} \cap K^{x}_{y}} .
\]
\end{lemma}
\begin{proof}
Since
$
K^{y}_{x} \cap K^{x}_{y} = \left(
K'_{x} \cap K'_{y}
\right) \cdot \left(
J^{y}_{x} \cap J^{x}_{y}
\right)
$,
and the characters $\epsilon^{G/G'}_{x}$, $\delta^{y}_{x}$, $\epsilon^{G/G'}_{y}$, and $\delta^{x}_{y}$ are trivial on the group $J^{y}_{x} \cap J^{x}_{y}$, it suffices to show that
\[
\left(
\epsilon^{G/G'}_{x} \cdot \delta^{x}_{y}
\right)\restriction_{K'_{x} \cap K'_{y}} = 
\left(
\epsilon^{G/G'}_{y} \cdot \delta^{y}_{x}
\right)\restriction_{K'_{x} \cap K'_{y}},
\]
which holds by \cite[Lemma~4.1.2]{FKS}.
\end{proof}

\begin{corollary}
\label{corollaryoflemmathatfollowsfromFKS412}
Let $M'$ be a Levi subgroup of $G'$ and $x, y \in \cB(M', F)$.
Let $M$ be the centralizer of $A_{M'}$ in $G$.
Let $\{\iota\}$ be a commutative diagram
\[
\xymatrix{
\cB(G', F) \ar@{^{(}->}[r] \ar@{}[dr]|\circlearrowleft & \cB(G, F)
\\
\cB(M', F) \ar@{^{(}->}[u] \ar@{^{(}->}[r] & \cB(M, F) \ar@{^{(}->}[u]
}
\]
of admissible embeddings of buildings, which we use to identify $\cB(M', F)$, $\cB(M, F)$, and $\cB(G', F)$ with their images in $\cB(G, F)$.
We assume that the images of the points $x$ and $y$ under the projection to
$\cB^{\red}(M', F)$ agree and that the embedding $\iota \colon \cB(M, F) \longrightarrow \cB(G, F)$ is $r/2$-generic relative to $x$ and $y$ in the sense of \cite[Definition~3.2]{Kim-Yu}.
Then we have $\epsilon^{G/G'}_{x} \restriction_{K^{y}_{x} \cap K^{x}_{y}}  = \epsilon^{G/G'}_{y} \restriction_{K^{y}_{x} \cap K^{x}_{y}} $.
\end{corollary}
\begin{proof}
Since the embedding $\iota \colon \cB(M, F) \longrightarrow \cB(G, F)$ is $r/2$-generic relative to $y$, we have $J_{y} / J_{y, +} = \left(
J_{y} \cap M(F)
\right) / \left(
J_{y, +} \cap M(F)
\right)$.
On the other hand, since the images of the points $x$ and $y$ under the projection to
$\cB^{\red}(M', F)$ agree, we have $J_{x} \cap M(F) = J_{y} \cap M(F)$.
Thus, we have 
\[
J_{y} / J_{y, +} \supseteq \Vyx = \left(
J_{x} \cap J_{y}
\right) \cdot J_{y, +}/J_{y, +} 
\supseteq \left(
J_{y} \cap M(F)
\right) / \left(
J_{y, +} \cap M(F)
\right) =
J_{y} / J_{y, +},
\]
hence $\Vyx=J_{y} / J_{y, +}$
and $
\Vyxp  = \left(
\Vyx 
\right)^{\perp} = \{0\}
$.
Thus, the definition of $\delta^{y}_{x}$ implies that $\delta^{y}_{x} = 1$.
Similarly, we have $\delta^{x}_{y} = 1$.
Now, the claim follows from Lemma~\ref{lemmathatfollowsfromFKS412}.
\end{proof}

\begin{definition} \label{definitionphiplus}
For $x \in \cB(G', F)$, we define the irreducible smooth representation $\phi_{x}^{+}$
\index{notation-ky}{phi+x@$\phi^{+}_{x}$}%
of $K_{x}$ as
\[
\phi_{x}^{+} = \phi'_{x} \otimes \epsilon^{G/G'}_{x}.
\]
\end{definition}
\begin{remark}
\label{remarkinthecaseG'=Gtwistedver}
In the case $G' = G$, using  Remark~\ref{remarkinthecaseG'=G}, we have $\epsilon^{G/G'}_{x} = 1$ and $\phi_{x}^{+} = \phi \restriction_{K'_{x} \cdot G(F)_{x, r}}$.
\end{remark}
Now we can rewrite Lemma~\ref{lemmaphiisotypic}, Lemma~\ref{lemma:phi'xisunitary} and Corollary~\ref{corollaryinductionuntwisted} in terms of $\phi^{+}_{x}$.
\begin{lemma}
\label{lemmaphiisotypictwisted}
The restriction of $\phi^{+}_{x}$ to $J_{x}$ is irreducible, and the restriction of $\phi^{+}_{x}$ to the group $K'_{x, 0+} \cdot J_{x, +}$ is $\widehat{\phi}_{x} \restriction_{K'_{x, 0+} \cdot  J_{x, +}}$-isotypic.
In particular, the restriction of $\phi^{+}_{x}$ to $G(F)_{x, r+}$ is trivial.
\end{lemma}
\begin{proof}
	This follows from Lemma~\ref{lemmaphiisotypic} and $\epsilon^{G/G'}_{x}$ being trivial on the group $K'_{x, 0+} \cdot J_{x}$
\end{proof}
\begin{lemma}
\label{lemma:phi+xisunitary}
Suppose that $\Coeff$ admits a nontrivial involution,
with respect to which the character $\phi_{x}$ of $K'_{x}$ is unitary.
Then the representation $\phi_{x}^{+}$ is unitary.
\end{lemma}
\begin{proof}
	This follows from Lemma~\ref{lemma:phi'xisunitary} and $\epsilon^{G/G'}_{x}$ being a quadratic character.
\end{proof}
\begin{corollary}
\label{corollaryinductiontwisted}
Let $x, y \in \cB(G', F)$ and $\phi$ be a character of $G'(F)$ that is $G$-generic of depth $r$ relative to $x$.
Let $K'_{x}$ (resp.\ $K'_{y}$) be an open subgroup of $G'(F)_{[x]_{G}}$
(resp.\ $G'(F)_{[y]_{G}}$).
We suppose Conditions \eqref{conditionsonxandXaboutK'} and \eqref{conditionsonxandXaboutJ}.
Then we have an isomorphism
\[
\phi_{y}^{+}\restriction_{K'_{x} \cdot J_{y}} \isoarrow \ind_{K_{x}}^{K'_{x} \cdot J_{y}} \left(
\phi^{+}_{x} 
\right).
\]
\end{corollary}
\begin{proof}
According to Corollary~\ref{corollaryinductionuntwisted}, we have
\(
\phi'_{y} \restriction_{K'_{x} \cdot J_{y}} \simeq \ind_{K_{x}}^{K'_{x} \cdot J_{y}} \left(
\phi'_{x} \otimes \delta^{y}_{x}
\right).
\)
Hence, Lemmas~\ref{lemmaKx=KXx} and \ref{lemmathatfollowsfromFKS412} and the definition of $\phi_{x}^{+}$ and $\phi_{y}^{+}$ imply that
\begin{align*}
\phi_{y}^{+}\restriction_{K'_{x} \cdot J_{y}}
 = \left(
\phi'_{y} \otimes \epsilon^{G/G'}_{y}
\right)\restriction_{K'_{x} \cdot J_{y}}
& \simeq \ind_{K_{x}}^{K'_{x} \cdot J_{y}} \left(
\phi'_{x} \otimes \delta^{y}_{x}
\right) \otimes \left(
\epsilon^{G/G'}_{y}\restriction_{K'_{x} \cdot J_{y}}
\right)\\
& = \ind_{K_{x}}^{K'_{x} \cdot J_{y}} \left(
\phi^{+}_{x} \otimes (\epsilon^{G/G'}_{x})^{-1}
\otimes \delta^{y}_{x}
\right) \otimes \left(
\epsilon^{G/G'}_{y}\restriction_{K'_{x} \cdot J_{y}}
\right)\\
& \simeq \ind_{K_{x}}^{K'_{x} \cdot J_{y}} \left(
\phi^{+}_{x} \otimes (\epsilon^{G/G'}_{x})^{-1}
\otimes \left(
\epsilon^{G/G'}_{y}\restriction_{K_{x}}
\right)
\otimes \delta^{y}_{x}
\right)\\
& \simeq \ind_{K_{x}}^{K'_{x} \cdot J_{y}} \left(
\phi^{+}_{x} \otimes (\delta^{x}_{y})^{-1}
\right).
\end{align*}
To prove the corollary, it suffices to show that $\delta^{x}_{y} = 1$.
According to Lemma~\ref{lemmaaboutxandX}(\ref{lemmaaboutxandXJxandJXcapJx}),
we have
\[
\Vxy = \left(
J_{x} \cap J_{y}
\right) \cdot J_{x, +}/J_{x, +} 
= J_{x} / J_{x, +}.
\]
Hence, we obtain that
$
\Vxyp = \left(
\Vxy 
\right)^{\perp} = \{0\}
$.
Thus, the definition of $\delta^{x}_{y}$ implies that $\delta^{x}_{y} = 1$.
\end{proof}

\subsection{Representations from Heisenberg--Weil data}
\label{subsec:repfromHeisenbergWeil}

In this subsection, we will generalize the construction of $\phi_{x}^{+}$ and the comparison result Corollary~\ref{corollaryinductiontwisted}, to start with the following more general input. 
\begin{definition}
\label{definition:Heisenberg-Weil}
A Heisenberg--Weil datum is a $5$-tuple
$\bigl((\overrightarrow{G}), \overrightarrow{r}, (x, \{\iota\}), K^{0}_{x}, \overrightarrow{\phi}\bigr)$
where
\begin{enumerate}[(1)]
\item
$\overrightarrow{G}=\left(G^0 \subseteq G^1 \subseteq\ldots \subseteq G^{n}=G\right)$ with $G^i$ a tamely ramified twisted Levi subgroup of $G$ for $0 \leq i \leq n-1$ and some $n \in \bZ_{\geq 1}$. 
\item
$\overrightarrow{r}=\left(r_0, \ldots , r_{n-1} \right)$ is a sequence of real numbers satisfying
$
0<r_0<r_1<\cdots <r_{n-1}
$.
\item
$x$ is a point of $\cB(G^{0}, F)$, and $\{\iota\}$ is a collection of compatible admissible embeddings of buildings
\[
\cB(G^{0}, F) \hookrightarrow \cB(G^{1}, F) \hookrightarrow \cdots \hookrightarrow \cB(G^{n}, F).
\]
We identify points in $\cB(G^{0}, F)$ with their images via the embeddings $\{ \iota \}$.
\item
$K^{0}_{x}$ is an open subgroup of $G^{0}(F)_{[x]_{G}}$.
\index{notation-ky}{Kay0 x@$K^{0}_{x}$}%
\item
$\overrightarrow{\phi}=\left(\phi_0, \ldots , \phi_{n-1} \right)$ is a sequence of characters, where $\phi_i$ is a character of $G^i(F)$.
We suppose that $\phi_i$ is $G^{i+1}$-generic of depth $r_i$ relative to $x$ in the sense of \cite[Definition~3.5.2]{Fintzen-IHES} for all $0\le i\le n-1$. 
\end{enumerate}
\end{definition}

A Heisenberg--Weil datum is a generalization of the part of the datum that Yu (\cite[Section~3]{Yu}) uses to construct the positive-depth factor
of his supercuspidal types via the theory of Heisenberg--Weil representations of finite groups. More precisely, we allow that some of the twisted Levi subgroups $G^i$ are equal (using the notion of generic characters as in \cite[Definition~3.5.2.]{Fintzen-IHES} in the case of equal twisted Levi subgroups), that $Z(G^{0})/Z(G)$ is isotropic, that $x$ is any point in $\cB(G^0,F)$, and we work with the more general group $K_x^0 \subseteq G^{0}(F)_{[x]_{G}}$ and the more general coefficients $\Coeff$. 
Following the construction in \cite[Section~4]{Yu}, which we will recall due to our more general set-up and for the convenience of the reader, we will construct from a Heisenberg--Weil datum an irreducible representation $\kappa_{x}$ of a compact, open subgroup $K_{x}$ of $G(F)$. We will then prove a comparison result for the representations $\kappa_{x}$ attached to different appropriate points $x$, see Proposition~\ref{propositioninductiontwistedsequencever} below.
This result will then be used in Section~\ref{Section-KimYutypes} as follows.
In Section~\ref{Section-KimYutypes}, we will construct an irreducible representation $\rho_{x}$ of $K_{x}$ following the construction in \cite[Section~4]{Yu} and \cite[Section~7]{Kim-Yu} from a $G$-datum $\Sigma$ (see \cite[7.2]{Kim-Yu} and Definition~\ref{definitionofGdatum} below for the definition of $G$-datum).
To define $\rho_{x}$, we will use the representation $\kappa_{x}$ constructed from the Heisenberg--Weil datum obtained as a part of $\Sigma$, and Proposition~\ref{propositioninductiontwistedsequencever} is the key to proving that the representations $\rho_{x}$ satisfy
Axiom~\ref{HAI-axiomextensionoftheinductionofkappa}\eqref{HAI-axiomextensionoftheinductionofkappacompatibilitywiththecompactind}
of \cite{HAI},
see Lemma~\ref{proofofaxiomextensionoftheinductionofkappa}.

Let $\left((\overrightarrow{G}), \overrightarrow{r}, (x, \{\iota\}), K^{0}_{x}, \overrightarrow{\phi}\right)$ be a Heisenberg--Weil datum.
We fix a maximal torus $T$ of $G^{0}$ such that the splitting field $E$ of $T$ is tamely ramified over $F$ and such that $x \in \cA(G^{0}, T, E)$. We will now roughly repeat the construction of Section \ref{A twist of the construction} for each pair $G^{i-1} \subseteq G^{i}$  ($1 \le i \le n$). More precisely, for $1 \le i \le n$,
we define the concave functions $f_i$ and $f_i^+$
on $\Phi(G^i,T) \cup \{0\}$ by 
\[
f_i(\alpha) = 
\begin{cases}
r_{i-1} & \alpha \in \Phi(G^{i-1},T) \cup \{0\}, \\
\frac{r_{i-1}}{2} & \text{otherwise},
\end{cases}
\qquad
\text{and}
\qquad
f_i^+(\alpha) = 
\begin{cases}
r_{i-1} & \alpha \in \Phi(G^{i-1},T) \cup \{0\}, \\
\frac{r_{i-1}}{2} + & \text{otherwise}.
\end{cases}
\]
We define the compact, open subgroups $J^{i}_{x}$ and $J^{i}_{x, +}$ of $G^{i}(F)$ as
\[
J^i_x \coloneqq  G^{i}(F)_{x,f_i}
\qquad
\text{and}
\qquad
J^i_{x,+} \coloneqq  G^{i}(F)_{x,f_i^+}.
\]
As explained in \cite[Section~1, Section~2]{Yu}, the groups $J^{i}_{x}$ and $J^{i}_{x, +}$ are independent of the choice of a maximal torus $T$ of $G^{0}$.
For $1 \le i \le n$, we write
$
J^{\le i}_{x} = J^{1}_{x} J^{2}_{x} \cdots J^{i}_{x}
$
and 
$
J^{\le i}_{x, +} = J^{1}_{x, +} J^{2}_{x, +} \cdots J^{i}_{x, +}
$.
We also write $J^{\le 0}_{x} = J^{\le 0}_{x, +} = \{1\}$.
For $0 \le i \le n$, we define the open subgroups $K^{i}_{x}$ and $K^{i}_{x, +}$ of $G^{i}(F)$ as
\[
K^{i}_{x} = K^{0}_{x} \cdot J^{\le i}_{x}
\quad 
\text{ and }
\quad
K^{i}_{x, +} = K^{0}_{x, +} \cdot J^{\le i}_{x, +},
\]
where $K^{0}_{x, +} = K^{0}_{x} \cap G^{0}(F)_{x, 0+}$.
\index{notation-ky}{Kay0 x +@$K^{0}_{x, +}$}%
We note that
$
K^{i}_{x} = K^{i-1}_{x} \cdot J^{i}_{x}
$
and
$
K^{i}_{x, +} = K^{i-1}_{x, +} \cdot J^{i}_{x, +}
$
for all $1 \le i \le n$.
We write $K_{x} \coloneqq  K^{n}_{x}$, $K_{x, +} \coloneqq K^{n}_{x, +}$,
and $J_{x} \coloneqq  J^{\le n}_{x}$.
\index{notation-ky}{Kay x +@$K_{x, +}$}%
\index{notation-ky}{Kay x@$K_x$}%
\index{notation-ky}{Kay x 0+@$K_{x, 0+}$}%
We also define the compact, open, pro-$p$ subgroup $K_{x, 0+}$ of $K_{x}$ by
\[
K_{x, 0+} 
= K_{x} \cap G(F)_{x, 0+}
= K^{0}_{x, +} \cdot J_{x}.
\]
We note that we have $K_{x} = K^{0}_{x} \cdot K_{x, 0+}$.

Applying the construction in Section~\ref{A twist of the construction} to 
\begin{equation}
	\label{applyingtwostepcaseforGi}
		G = G^{i}, \quad
		G' = G^{i-1}, \quad
		\phi = \phi_{i-1}, \quad
		K'_{x} = K^{i-1}_{x}, 
	\end{equation}
we obtain an irreducible representation $\phi^{+}_{i-1, x}$ of $K^{i}_{x}$ for all $1 \le i \le n$.
According to Lemma~\ref{lemmaphiisotypictwisted}, the restriction of $\phi^{+}_{i-1, x}$ to the group $G^{i}(F)_{x, r_{i-1}+}$ is trivial.
Since we have
\begin{align*}
K^{i}_{x} \cap J^{i+1}_{x} J^{i+2}_{x} \cdots J^{n}_{x} \subseteq G^{i}(F)_{x, r_{i}} \subseteq G^{i}(F)_{x, r_{i-1}+}
\end{align*}
for $1 \le i \le n-1$,
we can inflate the representation $\phi^{+}_{i-1, x}$ to an irreducible representation $\inf(\phi^{+}_{i-1, x})$ of the group $K_{x}$ via the surjection
\[
K_{x} = K^{i}_{x} \cdot J^{i+1}_{x} J^{i+2}_{x} \cdots J^{n}_{x} 
\longrightarrow K^{i}_{x} \cdot J^{i+1}_{x} J^{i+2}_{x} \cdots J^{n}_{x} / 
J^{i+1}_{x} J^{i+2}_{x} \cdots J^{n}_{x} 
\simeq K^{i}_{x} / \left(
K^{i}_{x} \cap J^{i+1}_{x} J^{i+2}_{x} \cdots J^{n}_{x}
\right).
\]
We define the irreducible representation
\index{notation-ky}{kappax@$\kappa_{x}$}%
$\kappa_{x}$
\index{notation-ky}{phi+x@$\phi^{+}_{x}$}%
of $K_{x}$ as
\[
\kappa_{x} = \inf(\phi^{+}_{0, x}) \otimes \inf(\phi^{+}_{1, x}) \otimes \cdots \otimes \inf(\phi^{+}_{n-2, x}) \otimes \phi^{+}_{n-1, x}.
\]
We note that the representation $\phi^{+}_{x}$ in Section~\ref{A twist of the construction} agrees with the representation $\kappa_{x}$ constructed from the Heisenberg--Weil datum
$
\bigl(
(G' \subseteq G), (r), (x, \{\iota\}), K'_{x}, (\phi)
\bigr)
$, i.e. the construction of $\kappa_x$ is a generalization of the construction of $\phi^{+}_{x}$ in Section~\ref{A twist of the construction}.
If $n>1$, then we also write $\kappa_x^{n-1}= \inf(\phi^{+}_{0, x}) \otimes \inf(\phi^{+}_{1, x}) \otimes \cdots \otimes \inf(\phi^{+}_{n-2, x})$.
We note that the representation $\kappa^{n-1}$ is trivial on $J^{n}_{x}$, and the restriction of $\kappa_x^{n-1}$ to $K^{n-1}_{x}$ agrees with the representation obtained by applying the above construction to the Heisenberg--Weil datum 
\[
\left(
\left(
G^0\subseteq G^1 \subseteq\ldots \subseteq G^{n-1}
\right),
\left(r_0, \ldots , r_{n-2} \right), (x, \{\iota\}), K^{0}_{x}, \left(\phi_0, \ldots , \phi_{n-2} \right)
\right).
\]

We define another representation, $\kappa^{\nt}_{x}$, of $K_{x}$ as follows:
\begin{notation}
\label{notationuntwistedconstruction}
For $1 \le i \le n$, let $\phi'_{i-1, x}$ denote the irreducible representation of $K^{i}_{x}$ obtained by applying the construction in Section~\ref{subsec:untwistedconstruction} to \eqref{applyingtwostepcaseforGi}.
Replacing $\phi^{+}_{i-1, x}$ by $\phi'_{i-1, x}$ in the definition of $\kappa_{x}$, we can define another representation $\kappa^{\nt}_{x}$
\index{notation-ky}{kappantx@$\kappa^{\nt}_{x}$}%
of $K_{x}$.
We call the construction of the representation $\kappa_{x}$ (resp.\ $\kappa^{\nt}_{x}$) from the Heisenberg--Weil datum $\left((\overrightarrow{G}), \overrightarrow{r}, (x, \{\iota\}), K^{0}_{x}, \overrightarrow{\phi}\right)$ the \emph{twisted Heisenberg--Weil construction}
\index{definition}{Heisenberg--Weil construction,twisted@twisted Heisenberg--Weil construction}%
\index{definition}{Heisenberg--Weil construction,non-twisted@non-twisted Heisenberg--Weil construction}%
(resp.\ \emph{non-twisted Heisenberg--Weil construction}).\label{pageuntwistedHeisenbergWeilconstruction}
\end{notation}

The difference between $\kappa_{x}$ and $\kappa^{\nt}_{x}$ can be described by using the characters $\epsilon^{G^{i}/G^{i-1}}_{x}$ defined in Section~\ref{A twist of the construction} as follows.
\begin{notation}
\label{notationepsilonoverrightarrowG}
Let $\epsilon^{\overrightarrow{G}}_{x}$
\index{notation-ky}{epsilonGarrowx@$\epsilon^{\overrightarrow{G}}_{x}$}%
be the character of $K_{x}$ that is trivial on $J_{x}$ and satisfies 
\[
\epsilon^{\overrightarrow{G}}_{x} \restriction_{K^{0}_{x}} = \prod_{i=1}^{n} \epsilon^{G^{i}/G^{i-1}}_{x} \restriction_{K^{0}_{x}}.
\]
\end{notation}
The definitions of $\kappa_{x}$ and $\kappa^{\nt}_{x}$ imply that we have
\[
\kappa_{x} = \kappa^{\nt}_{x} \otimes \epsilon^{\overrightarrow{G}}_{x}.
\]

We record some properties of $\kappa_{x}$.
\begin{lemma}
\label{irreducibilityofkapparestrictiontoJ}
The restriction of the representation $\kappa_{x}$ to the groups $J_{x}$ and $K_{x, 0+}$ are irreducible.
\end{lemma}
\begin{proof}
We will prove that the restriction of $\kappa_{x}$ to $J_{x}$ is irreducible by induction on $n$.
When $n = 1$, the claim follows from Lemma~\ref{lemmaphiisotypictwisted}.
Suppose that $n > 1$.
Since the group $J_{x}$ is compact, it suffices to show that
\[
\dim_{\Coeff}\left(
\End_{J_{x}} \left(
\kappa_{x} \restriction_{J_{x}}
\right)
\right) = 1.
\]
We note that $J_{x} = J^{\le n-1}_{x} \cdot J^{n}_{x}$.
Since the representation
$\kappa_x^{n-1}$
is trivial on $J^{n}_{x}$, and the restriction of the representation $\phi^{+}_{n-1, x}$ to $J^{n}_{x}$ is irreducible by Lemma~\ref{lemmaphiisotypictwisted},
we have the isomorphism
\[
\End_{J^{\le n-1}_{x}} \left(
\kappa^{n-1}_{x} \restriction_{J^{\le n-1}_{x}}
\right) \isoarrow \End_{J_{x}} \left(
\left(
\kappa^{n-1}_{x} \otimes \phi^{+}_{n-1, x}
\right) \restriction_{J_{x}}
\right) =
\End_{J_{x}} \left(
\kappa_{x} \restriction_{J_{x}}
\right)
\]
defined by $\Phi \mapsto \Phi \otimes \id_{\phi^{+}_{n-1, x}}$.
By combining this isomorphism with the induction hypothesis $\dim_{\Coeff}\left(
\End_{J^{\le n-1}_{x}} \left(
\kappa^{n-1}_{x} \restriction_{J^{\le n-1}_{x}}
\right)
\right) = 1$, we obtain the claim.
Since $J_{x} \subseteq K_{x, 0+}$, the restriction of $\kappa_{x}$ to $K_{x, 0+}$ is also irreducible.
\end{proof}

\begin{lemma}
\label{lemma:kappaxisunitary}
Suppose that $\Coeff$ admits a nontrivial involution,
with respect to which the restrictions of the characters $\phi_{i}$
to $K^{i}_{x}$ are unitary
for all $0 \le i \le n-1$.
Then the representations $\kappa_{x}$ and $\kappa^{\nt}_{x}$ are unitary.
\end{lemma}
\begin{proof}
The lemma follows from Lemma~\ref{lemma:phi'xisunitary}, Lemma~\ref{lemma:phi+xisunitary}, and the definition of $\kappa_{x}$ and $\kappa^{\nt}_{x}$.
\end{proof}

\begin{notation}
\label{notation:theta_x}
For $0 \le i \le n-1$, let $\widehat{\phi}_{i, x}$ denote the character of $K^{0}_{x} \cdot G^{i}(F)_{x, 0} \cdot G(F)_{x, \frac{r_{i}}{2} +}$ defined as in \cite[Section~4]{Yu}.
We define the character $\theta_{x}$
\index{notation-ky}{thetax@$\theta_x$}%
of $K_{x, +}$ as
\[
\theta_{x} = \prod_{i=0}^{n-1} \widehat{\phi}_{i, x} \restriction_{K_{x, +}}.
\]
\end{notation}
\begin{lemma}
\label{lemmakapparestrictiontoK+}
The restriction of the representation $\kappa_{x}$ to the group $K_{x, +}$ is $\theta_{x}$-isotypic.
\end{lemma}
\begin{proof}
The lemma follows from Lemma~\ref{lemmaphiisotypictwisted} and the definition of $\kappa_{x}$ and $\theta_{x}$ (see also \cite[Proposition~4.4]{Yu}).
\end{proof}

According to Lemma~\ref{lemmaphiisotypictwisted} and the construction of $\kappa_{x}$, the representation $\kappa_{x}$ is trivial on the group $G(F)_{x, r_{n-1}+}$.
For later use, we will prove a stronger version of this claim.
\begin{lemma}
\label{lemmakappaxistrivialonGXr}
Let $y \in \cB(G^{0}, F)$ and $r_{n-1} < r$ such that $G(F)_{y, r} \subseteq K_{x, +}$.
Then the representation $\kappa_{x}$ is trivial on the group $G(F)_{y, r}$.
\end{lemma}
\begin{proof}
Let $S^{0}$ be a maximal split torus of $G^{0}$ such that $x, y \in \cA(G^{0}, S^{0}, F)$.
We fix $U \in \cU(Z_{G}(S^{0}))$.
According to \cite[4.3~Proposition (a)]{Kim-Yu}, we have
\[
G(F)_{y, r} = \left(
G(F)_{y, r} \cap U(F)
\right) \cdot \left(
G(F)_{y, r} \cap Z_{G}(S^{0})(F)
\right) \cdot 
\left(
G(F)_{y, r} \cap \overline{U}(F)
\right).
\]
Hence, to prove the lemma, it suffices to show that the representation $\kappa_{x}$ is trivial on the groups $G(F)_{y, r} \cap U(F)$, $G(F)_{y, r} \cap Z_{G}(S^{0})(F)$, and $G(F)_{y, r} \cap \overline{U}(F)$.
Since $x, y \in \cA(G^{0}, S^{0}, F)$, we obtain that
\[
G(F)_{y, r} \cap Z_{G}(S^{0})(F) = Z_{G}(S^{0})(F)_{y, r} = Z_{G}(S^{0})(F)_{x, r} \subset G(F)_{x, r_{n-1}+}.
\]
Hence, the representation $\kappa_{x}$ is trivial on the group $G(F)_{y, r} \cap Z_{G}(S^{0})(F)$.

We will prove that the representation $\kappa_{x}$ is trivial on the group $G(F)_{y, r} \cap U(F)$ for all $U \in \cU(Z_{G}(S^{0}))$.
Since $G(F)_{y, r} \subseteq K_{x, +}$ and
 the restriction of $\kappa_{x}$ to $K_{x, +}$ is $\theta_{x}$-isotypic, it suffices to show that the character $\theta_{x}$ is trivial on the group $K_{x, +} \cap U(F)$.
Since $\phi_{i}$ is a character of $G^{i}(F)$, we obtain that $\phi_{i}$ is trivial on the group $G^{i}(F) \cap U(F)$ for each $0 \le i \le n-1$.
Then the claim that $\theta_{x}$ is trivial on the group $K_{x, +} \cap U(F)$ follows from the definition of $\widehat{\phi}_{i, x}$ and $\theta_{x} = \prod_{i=0}^{n-1} \widehat{\phi}_{i, x} \restriction_{K_{x, +}}$.
\end{proof}

Let $y \in \cB(G^{0}, F)$ 
 and let $K^{0}_{y}$ be an open subgroup of $G^{0}(F) \cap G(F)_{[y]_{G}}$
such that
\begin{align} 
\label{conditionsonxandysequenceveraboutK0}
G^{0}(F)_{y, r_{0}}, G^{0}(F)_{x, r_{0}} 
 \subseteq K^{0}_{x, +}
 \subseteq K^{0}_{x} \subseteq K^{0}_{y}
\end{align}
and \spacingatend{} 
\begin{align}
\label{conditionsonxandysequenceveraboutJ}
G^{i}(F)_{y, \frac{r_{i-1}}{2} +} \subseteq G^{i}(F)_{x, \frac{r_{i-1}}{2} +} \subseteq G^{i}(F)_{x, \frac{r_{i-1}}{2}} \subseteq G^{i}(F)_{y, \frac{r_{i-1}}{2}}
\end{align}
for all $1 \le i \le n$.
According to Lemma~\ref{lemmaindependenceofthegenericity}, the $5$-tuple $\left((\overrightarrow{G}), \overrightarrow{r}, (y, \{\iota\}), K^{0}_{y}, \overrightarrow{\phi}\right)$ is a Heisenberg--Weil datum.
Condition \eqref{conditionsonxandysequenceveraboutJ} ensures that Condition \eqref{conditionsonxandXaboutJ} is satisfied for the setting of \eqref{applyingtwostepcaseforGi} for all $1 \le i \le n$.
We will prove that Condition \eqref{conditionsonxandXaboutK'} is also satisfied in these settings. 
For later use, we will prove the following stronger statement.   
\begin{lemma}
\label{lemmaconditionsforcorollaryinductiontwisted}
We have 
\[
G^{i}(F)_{x, r_{i}}, G^{i}(F)_{y, r_{i}} \subseteq K^{i}_{x, +}
\]
for all $0 \le i \le n-1$
and
\[
K^{i}_{x} \subseteq K^{0}_{x} \cdot J^{\le i}_{y} \subseteq K^{i}_{y}
\]
for all $0 \le i \le n$.
In particular, Condition \eqref{conditionsonxandXaboutK'} is satisfied for the setting of \eqref{applyingtwostepcaseforGi} for all $1 \le i \le n$.
\end{lemma} 
\begin{proof} 
We will prove the lemma by induction on $i$.
When $i = 0$, the claims of the lemma follow from \eqref{conditionsonxandysequenceveraboutK0}. 
Suppose that $i > 0$.  
The definitions of $J^{i}_{x}$ and $K^{i}_{x}$ imply that
\(
G^{i}(F)_{x, r_{i}} \subseteq G^{i}(F)_{x, r_{i-1}} \subseteq J^{i}_{x, +} \subseteq K^{i}_{x, +}.
\)
We will prove that $G^{i}(F)_{y, r_{i}} \subseteq K^{i}_{x, +}$.
According to \eqref{conditionsonxandysequenceveraboutJ}, we have $J^{i}_{y, +} \subseteq G^{i-1}(F)_{y, r_{i-1}} \cdot J^{i}_{x, +}$.
Hence, we have
\[
G^{i}(F)_{y, r_{i}} \subseteq G^{i}(F)_{y, r_{i-1}} 
\subseteq J^{i}_{y, +} 
\subseteq G^{i-1}(F)_{y, r_{i-1}} \cdot J^{i}_{x, +}.
\]
Moreover, the induction hypothesis implies that $G^{i-1}(F)_{y, r_{i-1}} \subseteq K^{i-1}_{x, +}$.
Thus, we conclude that
\[
G^{i}(F)_{y, r_{i}} \subseteq K^{i-1}_{x, +} \cdot J^{i}_{x, +} = K^{i}_{x, +}.
\]

Next, we will prove that $K^{i}_{x} \subseteq K^{0}_{x} \cdot J^{\le i}_{y}$.
According to \eqref{conditionsonxandysequenceveraboutJ}, we have
$
J^{i}_{x} \subseteq G^{i-1}(F)_{x, r_{i-1}} \cdot J^{i}_{y}
$.
Moreover, the induction hypothesis implies that 
\(
G^{i-1}(F)_{x, r_{i-1}} \subseteq K^{i-1}_{x} \subseteq K^{0}_{x} \cdot J^{\le i-1}_{y}.
\)
Thus, we obtain that
\[
K^{i}_{x} = K^{i-1}_{x} \cdot J^{i}_{x}
\subseteq K^{i-1}_{x} \cdot G^{i-1}(F)_{x, r_{i-1}} \cdot J^{i}_{y} 
\subseteq K^{0}_{x} \cdot J^{\le i-1}_{y} \cdot J^{i}_{y} 
= K^{0}_{x} \cdot J^{\le i}_{y}. 
\]

Finally, since $K^{0}_{x} \subseteq K^{0}_{y}$, we have $K^{0}_{x} \cdot J^{\le i}_{y} \subseteq K^{0}_{y} \cdot J^{\le i}_{y} = K^{i}_{y}$.
\end{proof} 

The following proposition is a generalization of Corollary~\ref{corollaryinductiontwisted} to the general twisted Heisenberg--Weil construction.
\begin{proposition}
\label{propositioninductiontwistedsequencever}
Let $\left((\overrightarrow{G}), \overrightarrow{r}, (x, \{\iota\}), K^{0}_{x}, \overrightarrow{\phi}\right)$ be a Heisenberg--Weil datum.
Let $y$ and $K^0_y$ be as above, i.e. $y \in \cB(G^{0}, F)$ and $K^{0}_{y}$ is an open subgroup of $G^{0}(F) \cap G(F)_{[y]_{G}}$ such that \eqref{conditionsonxandysequenceveraboutK0} and \eqref{conditionsonxandysequenceveraboutJ} hold.
Then we have an isomorphism
\[
\kappa_{y} \restriction_{K^{0}_{x} \cdot J_{y}}
\isoarrow
\ind_{K_{x}}^{K^{0}_{x} \cdot J_{y}} (\kappa_{x}).
\]
\end{proposition}
\begin{proof}
\addtocounter{equation}{-1}
\begin{subequations}
We will prove the proposition by the induction on $n$.
When $n = 1$, the claim follows from \eqref{conditionsonxandysequenceveraboutK0}, \eqref{conditionsonxandysequenceveraboutJ}, Lemma~\ref{lemmaconditionsforcorollaryinductiontwisted}, and Corollary~\ref{corollaryinductiontwisted}.
Suppose that $n > 1$.
Then for $z \in \{x, y\}$, we have by definition and the induction hypothesis that
\begin{align}
\label{bydefinitionofkappaxandyandinductionhypo}
\kappa_{z} \isoarrow \kappa^{n-1}_{z} \otimes \phi^{+}_{n-1, z},
\quad \text{ and } \quad
\kappa^{n-1}_{y} \restriction_{K^{0}_{x} \cdot J^{\le n-1}_{y}} \, \simeq \, \ind_{K^{n-1}_{x}}^{K^{0}_{x} \cdot J^{\le n-1}_{y}} \left(
\kappa^{n-1}_{x} \restriction_{K^{n-1}_{x}}
\right).
\end{align}
Moreover, according to \eqref{conditionsonxandysequenceveraboutK0}, \eqref{conditionsonxandysequenceveraboutJ}, Lemma~\ref{lemmaconditionsforcorollaryinductiontwisted}, and Corollary~\ref{corollaryinductiontwisted}, we obtain that
\begin{align}
\label{applyingcorollaryinductiontwistedtoGd}
\phi_{n-1, y}^{+}\restriction_{K^{n-1}_{x} \cdot J^{n}_{y}} \simeq \ind_{K_{x}}^{K^{n-1}_{x} \cdot J^{n}_{y}} \left(
\phi^{+}_{n-1, x} 
\right).
\end{align} 

To prove the proposition, we will inflate the representation $\kappa^{n-1}_{x}$ to $K_{x} \cdot J^{n}_{y}$ as follows.
According to Lemma~\ref{lemmaconditionsforcorollaryinductiontwisted}, we have
\(
K^{n-1}_{x} \cap J^{n}_{y} \subseteq G^{n-1}(F) \cap J^{n}_{y} = G^{n-1}(F)_{y, r_{n-1}} \subseteq K^{n-1}_{x, +}.
\)
Then applying Lemma~\ref{lemmakappaxistrivialonGXr} to $\kappa^{n-1}_{x} \restriction_{K^{n-1}_{x}}$, we obtain that the representation $\kappa^{n-1}_{x}$ is trivial on the group $K^{n-1}_{x} \cap J^{n}_{y} \subseteq G^{n-1}(F)_{y, r_{n-1}}$.
Hence, we can inflate the representation $\kappa^{n-1}_{x}$ to the group $K_{x} \cdot J^{n}_{y}$ via the map
\(
K_{x} \cdot J^{n}_{y} \twoheadrightarrow K_{x} \cdot J^{n}_{y} / J^{n}_{y} 
\simeq K_{x} / \left(
K_{x} \cap J^{n}_{y}
\right) 
=  K_{x} / \left(
K^{n-1}_{x} \cap J^{n}_{y}
\right) \cdot (J^{n}_{x} \cap J^{n}_{y}),
\)
where the last equality follows from Lemma~\ref{lemmaKx=KXx}.
We write $\inf(\kappa^{n-1}_{x})$ for this representation of $K_{x} \cdot J^{n}_{y} = K^{n-1}_{x} \cdot J^{n}_{y}$.

Now, combining \eqref{bydefinitionofkappaxandyandinductionhypo} 
 with \eqref{applyingcorollaryinductiontwistedtoGd}, we obtain that
\begin{align*}
\ind_{K_{x}}^{K^{0}_{x} \cdot J_{y}} \kappa_{x} & \simeq 
\ind_{K^{n-1}_{x} \cdot J^{n}_{y}}^{K^{0}_{x} \cdot J_{y}}
\left(
\ind_{K_{x}}^{K^{n-1}_{x} \cdot J^{n}_{y}} \kappa_{x}
\right) 
\simeq \ind_{K^{n-1}_{x} \cdot J^{n}_{y}}^{K^{0}_{x} \cdot J_{y}}
\left(
\ind_{K_{x}}^{K^{n-1}_{x} \cdot J^{n}_{y}} \left(
\kappa^{n-1}_{x} \otimes \phi^{+}_{n-1, x}
\right)
\right) \\
& \simeq \ind_{K^{n-1}_{x} \cdot J^{n}_{y}}^{K^{0}_{x} \cdot J_{y}}
\left(
\inf(\kappa^{n-1}_{x}) \otimes
\ind_{K_{x}}^{K^{n-1}_{x} \cdot J^{n}_{y}} 
\phi^{+}_{n-1, x}
\right) 
\simeq \ind_{K^{n-1}_{x} \cdot J^{n}_{y}}^{K^{0}_{x} \cdot J_{y}}
\left(
\inf(\kappa^{n-1}_{x}) \otimes
\phi^{+}_{n-1, y} \restriction_{K^{n-1}_{x} \cdot J^{n}_{y}}
\right) \\
& \simeq 
\ind_{K^{n-1}_{x} \cdot J^{n}_{y}}^{K^{0}_{x} \cdot J_{y}}
\left(
\inf(\kappa^{n-1}_{x}) 
\right) \otimes
\phi^{+}_{n-1, y} \restriction_{K^{0}_{x} \cdot J_{y}} 
\simeq 
\inf \left(
\ind_{K^{n-1}_{x}}^{K^{0}_{x} \cdot J^{\le n-1}_{y}} \left(
\kappa^{n-1}_{x} \restriction_{K^{n-1}_{x}}
\right)
\right) \otimes
\phi^{+}_{n-1, y} \restriction_{K^{0}_{x} \cdot J_{y}} \\
& \simeq 
\inf \left(
\kappa^{n-1}_{y} \restriction_{K^{0}_{x} \cdot J^{\le n-1}_{y}}
\right) \otimes
\phi^{+}_{n-1, y} \restriction_{K^{0}_{x} \cdot J_{y}} 
\simeq 
\kappa^{n-1}_{y} \restriction_{K^{0}_{x} \cdot J_{y}} \otimes
\phi^{+}_{n-1, y} \restriction_{K^{0}_{x} \cdot J_{y}} \\
&= \left(
\kappa^{n-1}_{y} \otimes \phi^{+}_{n-1, y}
\right)\restriction_{K^{0}_{x} \cdot J_{y}} 
= \kappa_{y} \restriction_{K^{0}_{x} \cdot J_{y}},
\end{align*}
where $\inf \left(
\ind_{K^{n-1}_{x}}^{K^{0}_{x} \cdot J^{\le n-1}_{y}} \left(
\kappa^{n-1}_{x} \restriction_{K^{n-1}_{x}}
\right)
\right)$ and $\inf \left(
\kappa^{n-1}_{y} \restriction_{K^{0}_{x} \cdot J^{\le n-1}_{y}}
\right)$ denote the inflations of the representations 
\(
\ind_{K^{n-1}_{x}}^{K^{0}_{x} \cdot J^{\le n-1}_{y}} \left(
\kappa^{n-1}_{x} \restriction_{K^{n-1}_{x}}
\right) \simeq \kappa^{n-1}_{y} \restriction_{K^{0}_{x} \cdot J^{\le n-1}_{y}}
\)
of $K^{0}_{x} \cdot J^{\le n-1}_{y}$ to the group $K^{0}_{x} \cdot J_{y}$ via the map
\[
K^{0}_{x} \cdot J_{y}  = K^{0}_{x} \cdot J^{\le n-1}_{y} \cdot J^{n}_{y} 
\twoheadrightarrow K^{0}_{x} \cdot J^{\le n-1}_{y} \cdot J^{n}_{y} / J^{n}_{y} 
\simeq K^{0}_{x} \cdot J^{\le n-1}_{y} / \left(
K^{0}_{x} \cdot J^{\le n-1}_{y} \cap J^{n}_{y}
\right) 
= K^{0}_{x} \cdot J^{\le n-1}_{y} / G^{n-1}(F)_{y, r_{n-1}}.
\]
\end{subequations}
\end{proof}  

We record an immediate corollary of Proposition~\ref{propositioninductiontwistedsequencever} that will be used to prove Lemma~\ref{proofofaxiomextensionoftheinductionofkappa} below, which in turn is used to verify that the types constructed by Kim and Yu satisfy the axioms of \cite{HAI}.
\begin{corollary}
\label{corollaryofpropositioninductiontwistedsequencever}
Let $\left((\overrightarrow{G}), \overrightarrow{r}, (x, \{\iota\}), K^{0}_{x}, \overrightarrow{\phi}\right)$ be a Heisenberg--Weil datum.
Let $y \in \cB(G^{0}, F)$ and $K^{0}_{y}$ be an open subgroup of $G^{0}(F) \cap G(F)_{[y]_{G}}$.
Suppose that $G^{0}(F)_{y, 0+} \subseteq G^{0}(F)_{x, 0+} \subseteq K^{0}_{x} \subseteq K^{0}_{y}$ and
\(
G^{i}(F)_{y, \frac{r_{i-1}}{2} +} \subseteq G^{i}(F)_{x, \frac{r_{i-1}}{2} +} \subseteq G^{i}(F)_{x, \frac{r_{i-1}}{2}} \subseteq G^{i}(F)_{y, \frac{r_{i-1}}{2}}
\)
for all $1 \le i \le n$.
Then we have an isomorphism
\[
\kappa_{y} \restriction_{K^{0}_{x} \cdot K_{y, 0+}}
\isoarrow
\ind_{K_{x}}^{K^{0}_{x} \cdot K_{y, 0+}} (\kappa_{x}).
\]
\end{corollary}
\begin{proof}
According to the assumptions, we have 
\(
G^{0}(F)_{y, r_{0}}, G^{0}(F)_{x, r_{0}} \subseteq G^{0}(F)_{x, 0+} = K^{0}_{x, +} \subseteq K^{0}_{x} \subseteq K^{0}_{y}.
\)
Moreover, we have $K^{0}_{x} \cdot K_{y, 0+} = K^{0}_{x} \cdot (K^{0}_{y} \cap G^{0}(F)_{y, 0+}) \cdot J_{y} = K^{0}_{x} \cdot J_{y}$.
Thus, the claim follows from Proposition~\ref{propositioninductiontwistedsequencever}.
\end{proof}

\subsection{An extension of the non-twisted Heisenberg--Weil representation} \label{Section-untwisted-Heisenberg-Weil-extension}

Let $\HW(\Sigma) = \bigl((\overrightarrow{G}), \overrightarrow{r}, (x, \{\iota\}), K^{0}_{x}, \overrightarrow{\phi}\bigr)$ be a Heisenberg--Weil datum and $M^{0}$ be a Levi subgroup of $G^0$.
Let $M^{i}$ denote the centralizer of $A_{M^0}$ in $G^i$ for $0 \le i \le n$. 
According to \cite[2.4~Lemma~(a), (b)]{Kim-Yu}, $M^i$ is a Levi subgroup of $G^i$ and a tamely ramified twisted Levi subgroup of $M \coloneqq M^{n}$. Note that by construction $Z(M^0)/Z(M^i)$ is anisotropic for all $0 \le i \le n$.
We write $\overrightarrow{M}=\left(M^0 \subseteq M^1 \subseteq \ldots \subseteq M^{n}\right)$.
We fix a commutative diagram $\{ \iota \}$
\[
\xymatrix{
\cB(G^{0}, F) \ar[r] \ar@{}[dr]|\circlearrowleft & \cB(G^{1}, F) \ar[r] &
\cdots \ar@{}[d]|\circlearrowleft  \ar[r] & \cB(G^{n}, F)
\\
\cB(M^{0}, F) \ar[u] \ar[r] & \cB(M^{1}, F) \ar[u] \ar[r] &
\cdots \ar[r] & \cB(M^{n}, F) \ar[u]
}
\]
of admissible embeddings of Bruhat--Tits buildings and identify a point in $\cB(M^{0}, F)$ with its images via the embeddings $\{ \iota \}$.
We assume that $x \in \cB(M^{0}, F)$.
We define the open subgroup $K_{M^0}$ of $M^{0}(F)_{[x]_{M}}$ by $K_{M^0} = M^{0}(F) \cap K^{0}_{x}$.\index{notation-ky}{Kay M0@$K_{M^0}$}
The following lemma is proved in \cite[5.3 Lemma]{Kim-Yu} for a slightly less general notion of generic characters and assuming the additional hypothesis that $p$ is not a torsion prime for the dual absolute root datum of $G$. 

\begin{lemma}
\label{lemma:genericityforGimpliesgenericityforM}
The characters $\phi_{i} \restriction_{M^{i}(F)}$ are $M^{i+1}$-generic of depth $r_i$ relative to $x$ for all $0\le i\le n-1$.
Thus, the $5$-tuple
$\HW(\Sigma)_{M} = \bigl((\overrightarrow{M}), \overrightarrow{r}, (x, \{\iota\}), K_{M^0}, (\phi_{0} \restriction_{M^{0}(F)}, \ldots , \phi_{n-1} \restriction_{M^{n-1}(F)} )\bigr)$ is a Heisenberg--Weil datum.
\end{lemma} 
\begin{proof}
Let $0 \le i \le n-1$.
Since the character $\phi_{i}$ is $G^{i+1}$-generic of depth $r_{i}$ relative to the point $x$, the character $\phi_{i}$ is trivial on $G^{i}(F)_{x, r_{i} +}$, and there exists an element $X^{*}_{i} \in \Lie^{*}(G^{i})^{G^{i}}(F)$ which is $G^{i+1}$-generic of depth $-r_{i}$ in the sense of \cite[Definition~3.5.2]{Fintzen-IHES} such that the restriction of $\phi_{i}$ to 
\(
G^{i}(F)_{x, r_{i}}/G^{i}(F)_{x, r_{i} +} \simeq \Lie(G^{i})(F)_{x, r_{i}} / \Lie(G^{i})(F)_{x, r_{i} +} 
\)
is given by $\Psi \circ X^{*}_{i}$.
Let $X^{*}_{M, i} \in \Lie^{*}(M^{i})^{M^{i}}(F)$ denote the restriction of $X^{*}_{i}$ to $\Lie(M^{i})$.
Since the restriction of $\phi_{i} \restriction_{M^{i}(F)}$ to $M^{i}(F)_{x, r_{i}}/M^{i}(F)_{x, r_{i} +} $ is given by $\Psi \circ X^{*}_{M, i}$, it suffices to show that $X^{*}_{M, i}$ is $M^{i+1}$-generic of depth $-r_{i}$.
Since $X^{*}_{i}$ has depth $-r_{i}$ at $x$, the restriction of $\psi \circ X^{*}_{i}$ to $\Lie(G^{i})_{x, r_{i}}$ is non-trivial.
Moreover, the restriction map from $\Lie^*(G^{i})$ to $\Lie^*(M^{i})$ yields, by definition of $M^i$, an isomorphism of the subspace $\Lie^*(G^{i})^{A_{M^{0}}} \subset \Lie^*(G^{i})$  with $\Lie^*(M^{i})$. Hence, since $X^*_i \in \Lie^*(G^{i})^{A_{M^{0}}} \subset \Lie^*(G^{i})$, the restriction of $\psi \circ X^{*}_{i}$ to $\Lie(M^{i})_{x, r_{i}}$ is also non-trivial. Thus, $X^{*}_{M, i}$ satisfies Condition {\bf (GE0)} of \cite[Definition~3.5.2]{Fintzen-IHES}.
Condition {\bf (GE1)} of \cite[Definition~3.5.2]{Fintzen-IHES} follows from the same condition for $X^*_i$ because $\Phi(M^{i+1}, T) \smallsetminus \Phi(M^{i}, T) \subseteq \Phi(G^{i+1}, T) \smallsetminus \Phi(G^{i}, T)$ for $T$ some maximal torus of $M^{i}$ that splits over a tame extension $E/F$.
Moreover, Condition {\bf (GE2)} of \cite[Definition~3.5.2]{Fintzen-IHES} follows from the genericity of $X^{*}_{i}$ and the fact that the absolute Weyl group of $M^{i}$ relative to a maximal torus $T$ of $M^{i}$ is the intersection of the absolute Weyl groups of $M^{i+1}$ and $G^{i}$ relative to $T$.

The last claim follows from the first claim and the fact that $M^{i}$ are tamely ramified twisted Levi subgroups of $M$ for all $0\le i\le n$.
\end{proof}

Let $\kappa^{\nt}_{M}$
\index{notation-ky}{kappaMnt@$\kappa_M^{\nt}$}%
denote the representation obtained from the Heisenberg--Weil datum $\HW(\Sigma)_{M}$ via the non-twisted Heisenberg--Weil construction.
Replacing the Heisenberg--Weil datum $\HW(\Sigma)$ with $\HW(\Sigma)_{M}$ in the definitions of $K_{x}$, $K_{x, 0+}$, $J^{i}_{x}$, and $J^{i}_{x, +}$ we define the open subgroups $K_{M}$ and $K_{M, 0+}$ of $M(F)$ and the compact, open subgroups $J^{i}_{M}$ and $J^{i}_{M, +}$ of $M^{i}(F)$ for $1 \le i \le n$.\index{notation-ky}{Kay M@$K_{M}$}\index{notation-ky}{Kay M 0+@$K_{M, 0+}$}

\begin{lemma}
\label{lemma:NG0M0xM0normalizescompactopensungroups}
The group $N_{G^{0}}(M^{0})(F)_{[x]_{M^{0}}}$ normalizes the groups $J^{i}_{M}$ and $J^{i}_{M, +}$ for $1 \le i \le n$. 
If the group $N_{G^{0}}(M^{0})(F)_{[x]_{M^{0}}}$ normalizes the group $K_{M^0}$, it also normalizes the groups $K_{M}$ and $K_{M, 0+}$.

\end{lemma}
\begin{proof}
The first claim follows from the definitions of $N_{G^{0}}(M^{0})(F)_{[x]_{M^{0}}}$, $J^{i}_{M}$, and $J^{i}_{M, +}$, and the observation that $N_{G^{0}}(M^{0})(F)_{[x]_{M^{0}}}=N_{G^{0}}(M^{0})(F)_{[x]_{M^{i}}}$, because $Z(M^0)/Z(M^{i})$ is anisotropic.
The second claim follows from the first claim and the definition of $K_{M}$ and $K_{M, 0+}$.
\end{proof}

Let $0 \le i \le n-1$. Since $\phi_i$ is a character of $G^i(F) \supseteq G^0(F)$ and since $N_{G^{0}}(M^{0})(F)_{[x]_{M^{0}}}$ normalizes $(M^i, M^{i+1})(F)_{x, r_i+, \frac{r_i}{2}+}$, the character $\hat{\phi_{i}}|_{J^{i+1}_{M, +}}$ can be extended to a character of  $N_{G^{0}}(M^{0})(F)_{[x]_{M^{0}}} \cdot J^{i+1}_{M, +}$.
Hence the conjugation action of $N_{G^{0}}(M^{0})(F)_{[x]_{M^{0}}}$ on $J^{i+1}_{M}$ arising from Lemma~\ref{lemma:NG0M0xM0normalizescompactopensungroups} induces a group homomorphism
\begin{equation}
\label{grouphomfromNG0M0xM0tosymplecticgroups}
N_{G^{0}}(M^{0})(F)_{[x]_{M^{0}}} \longrightarrow \Sp \left(
J^{i+1}_{M}/J^{i+1}_{M, +}
\right).
\end{equation}
We define the representation $\widetilde \phi'_{i, M}$ of $N_{G^{0}}(M^{0})(F)_{[x]_{M^{0}}}$ by composing \eqref{grouphomfromNG0M0xM0tosymplecticgroups} with the Weil representation of $\Sp \left(
J^{i+1}_{M}/J^{i+1}_{M, +}
\right)$ associated with the central character $\iota_{p}^{-1}$ and taking the tensor product of the resulting representation with $\phi_{i} \restriction_{N_{G^{0}}(M^{0})(F)_{[x]_{M^{0}}}}$.

Now, we assume that the group $N_{G^{0}}(M^{0})(F)_{[x]_{M^{0}}}$ normalizes the group $K_{M^0}$.
This assumption is satisfied, for instance, if we take $K_{M^0} = M^0(F)_{x}$ or $K_{M^0} = M^0(F)_{x, 0}$.
According to Lemma~\ref{lemma:NG0M0xM0normalizescompactopensungroups}, $N_{G^{0}}(M^{0})(F)_{[x]_{M^{0}}} \cdot K_{M}$ is an open subgroup of $M(F)$.
\begin{proposition}
\label{proposition:extensionofuntwistedWeilHeisenbergrepresentation}
There exists a unique extension $\widetilde \kappa_{M}^{\nt}$ of $\kappa_{M}^{\nt}$ to $N_{G^{0}}(M^{0})(F)_{[x]_{M^{0}}} \cdot K_{M}$ such that the restriction of $\widetilde \kappa_{M}^{\nt}$ to $N_{G^{0}}(M^{0})(F)_{[x]_{M^{0}}}$ is given by the representation $\widetilde \phi'_{0, M} \otimes \widetilde \phi'_{1, M} \otimes \cdots \otimes \widetilde \phi'_{d, M}$.
\end{proposition}
\begin{proof}
The proposition follows from the definitions of $\kappa_{M}^{\nt}$ and $\widetilde \phi'_{i, M}$.
\end{proof}

The following lemma is an analogue of Lemma~\ref{lemma:kappaxisunitary} and will be used in the proof of Proposition~\ref{propproofofaxiomaboutKM0vsKM} below.
\begin{lemma}
\label{lemma:kappatildeisunitary}
Suppose that $\Coeff$ admits a nontrivial involution,
with respect to which the restrictions of the characters $\phi_{i}$
to $K^{i}_{x}$ are unitary
for all $0 \le i \le n-1$.
Then the representation $\widetilde \kappa_{M}^{\nt}$ is unitary.
\end{lemma}
\begin{proof}
Using the same argument as in Lemma~\ref{lemma:phi'xisunitary}, we obtain that the representations $\widetilde \phi'_{i, M}$ of $N_{G^{0}}(M^{0})(F)_{[x]_{M^{0}}}$ are unitary for all $0 \le i \le n-1$.
Hence $\widetilde \kappa_{M}^{\nt}$ restricted to $N_{G^{0}}(M^{0})(F)_{[x]_{M^{0}}}$ is unitary. Note that $N_{G^{0}}(M^{0})(F)_{[x]_{M^{0}}} \cdot K_{M} = N_{G^{0}}(M^{0})(F)_{[x]_{M^{0}}} \cdot K_{M, 0+}$ and $K_{M, 0+}$ is a normal, compact, pro-$p$ subgroup of $N_{G^{0}}(M^{0})(F)_{[x]_{M^{0}}} \cdot K_{M, 0+}$. Thus, by integrating an $N_{G^{0}}(M^{0})(F)_{[x]_{M^{0}}}$-invariant Hermitian form over $K_{M, 0+}$, we obtain an $N_{G^{0}}(M^{0})(F)_{[x]_{M^{0}}} \cdot K_{M}$-invariant Hermitian form.
\end{proof}


\section{Hecke algebras for the types constructed by Kim and Yu}
\label{Section-KimYutypes}
In this section we will show that the types $(K_{x_0}, \rho_{x_0})$ constructed by Kim and Yu (\cite{Kim-Yu}), but twisted by a quadratic character introduced in \cite{FKS} and allowing more general coefficients $\Coeff$, see Section \ref{subsec:construction} for the definitions,
satisfy all of the axioms necessary for \cite[Theorems \ref{HAI-theoremstructureofhecke} and \ref{HAI-thm:isomorphismtodepthzero} ]{HAI} to hold
so that 
 we obtain an isomorphism between the Hecke algebra attached to $(K_{x_0}, \rho_{x_0})$ and a depth-zero Hecke algebra (see Theorem \ref{heckealgebraisomforKim-Yutype}) as well as an explicit description of those Hecke algebras as a semi-direct product of an affine Hecke algebra with a twisted group algebra (see Theorem \ref{theoremstructureofheckeforKimYu}). The results obtained along the way might be of independent interest.

\subsection{A twist of the construction by Kim and Yu}
\label{subsec:construction}
We begin by recalling the notion of a $G$-datum following Kim and Yu (\cite{Kim-Yu}), but adjusted to our more general coefficient field $\Coeff$. From those data we will afterwards construct compact, open subgroups and representations of some of those following the construction of Kim and Yu (\cite{Kim-Yu}), but including a twist by the quadratic character of \cite{FKS}.
Let $d \in \mathbb Z_{\geq 0}$ and let $\overrightarrow{G}=\left(G^0\subsetneq G^1 \subsetneq\ldots \subsetneq G^d=G\right)$ be a sequence of twisted Levi subgroups of $G$. 
\index{notation-ky}{G0@$G^0$}%
Let $M^0$ be a Levi subgroup of $G^0$.
\index{notation-ky}{M0@$M^0$}%
We denote by $M^{i}$ the centralizer of $A_{M^0}$ in $G^i$ for $0 \le i \le d$.
According to \cite[2.4~Lemma~(a), (b)]{Kim-Yu}, $M^i$ is a Levi subgroup of $G^i$ and a tamely ramified twisted Levi subgroup of $M \coloneqq M^d$.
\index{notation-ky}{M@$M$}%
Note that by construction $Z(M^0)/Z(M^i)$ is anisotropic for all $0 \le i \le d$.

\begin{definition}[{\cite[7.2]{Kim-Yu}}] \label{definitionofGdatum}
A \emph{$G$-datum}
is a $5$-tuple
$\bigl(
(\overrightarrow{G}, M^0), \overrightarrow{r}, (x_{0}, \{\iota\}), (K_{M^0}, \rho_{M^0}), \overrightarrow{\phi}
\bigr)$
satisfying the following:
\begin{description}
\item[\bf{D1}] 
$\overrightarrow{G}=\left(G^0\subsetneq G^1 \subsetneq\ldots \subsetneq G^d=G\right)$ is a sequence of twisted Levi subgroups of $G$ that split over a tamely ramified extension of $F$ for some $d  \in \mathbb Z_{\geq 0}$,
and $M^0$ is a Levi subgroup of $G^0$.
Let $M^0 \subseteq M^1 \subseteq \ldots \subseteq M^d$ be as constructed above.
\item[\bf{D2}]
$\overrightarrow{r}=\left(r_0, \ldots , r_d\right)$
is a sequence of real numbers satisfying
$0\leq r_0\leq r_1\leq \cdots \leq r_{d}$, where all of the inequalities
are strict except for the last one.
We also write $r_{-1} = 0$.
\item[\bf{D3}]
$x_{0}$ is a point of $\cB(M^0, F)$, and $\{\iota\}$ is a commutative diagram 
\[
\xymatrix{
\cB(G^{0}, F) \ar[r] \ar@{}[dr]|\circlearrowleft & \cB(G^{1}, F) \ar[r] &
\cdots \ar@{}[d]|\circlearrowleft  \ar[r] & \cB(G^{d}, F)
\\
\cB(M^{0}, F) \ar[u] \ar[r] & \cB(M^{1}, F) \ar[u] \ar[r] &
\cdots \ar[r] & \cB(M^{d}, F) \ar[u]
}
\]
of admissible embeddings of buildings that is $\overrightarrow{r}/2$-generic relative to $x_{0}$ in the sense of \cite[3.5~Definition]{Kim-Yu},
where $\overrightarrow{r}/2 = (0, \frac{r_{0}}{2}, \cdots, \frac{r_{d-1}}{2})$.
We identify a point in $\cB(M^{0}, F)$ with its images via the embeddings $\{ \iota \}$.
\item[\bf{D4}]
\index{notation-ky}{Kay M0@$K_{M^0}$}%
\index{notation-ky}{rhoM0@$\rho_{M^0}$}%
$K_{M^0}$
is a compact, open subgroup of $M^0(F)_{x_{0}}$ containing $M^0(F)_{x_{0}, 0}$, and $\rho_{M^0}$ is an irreducible smooth representation of $K_{M^0}$ such that $((G^0, M^0), (x_{0}, \iota:\cB(M^0,F) \longrightarrow \cB(G^0,F)), (K_{M^0}, \rho_{M^0}))$ is a depth-zero $G^0$-datum as in \cite[Definition~\ref{HAI-definition:depthzeroGdatum}]{HAI} (following \cite[7.1]{Kim-Yu}), where $\iota: \cB(M^0,F) \longrightarrow \cB(G^0,F)$ is the embedding from $\{ \iota \}$.
\item[\bf{D5}]
$\overrightarrow{\phi}=\left(\phi_0, \ldots , \phi_d\right)$ is a sequence of characters, where $\phi_i$ is a character of $G^i(F)$.
We assume that $\phi_{d} = 1$ if $r_{d-1} = r_{d}$ and otherwise $\phi_{d}$ is of depth $r_{d}$. 
We also assume that $\phi_i$ is $G^{i+1}$-generic of depth $r_i$ relative to $x_0$ 
 in the sense of \cite[Definition~3.5.2]{Fintzen-IHES} for all $0\le i\le d-1$. 
\end{description}
\end{definition}

From now on, we let
\index{notation-ky}{SZigma@$\Sigma$}%
$\Sigma = \bigl((\overrightarrow{G}, M^0), \overrightarrow{r}, (x_{0}, \{\iota\}), (K_{M^0}, \rho_{M^0}), \overrightarrow{\phi}\bigl)$ be a $G$-datum with $r_0>0$. The case $r_0=0$ corresponds to a depth-zero $G$-datum that was already treated in \cite[Section~\ref{HAI-sec:depth-zero}]{HAI} and the reduction-to-depth-zero results in this case are trivial. 
According to Lemma \ref{lemmaindependenceofthegenericity}, for each
\index{notation-ky}{A x0@$\cA_{x_{0}}$}%
$x \in \cA_{x_{0}} \coloneqq  x_{0} + \left(
X_{*}(A_{M}) \otimes_{\mathbb{Z}} \mathbb{R}
\right)$
such that the diagram $\{ \iota \}$ is $\overrightarrow{r}/2$-generic relative to $x$, the tuple
\index{notation-ky}{SZigmax@$\Sigma_x$}%
$\Sigma_{x} = \bigl((\overrightarrow{G}, M^0), \overrightarrow{r}, (x, \{\iota\}), (K_{M^0}, \rho_{M^0}), \overrightarrow{\phi}\bigl)$
is also a $G$-datum.
From this $G$-datum $\Sigma_{x}$, we will now construct pairs $(K_{x}^0, \rho_{x}^0)$, $(K_{x}, \rho_{x})$, and $(K_M, \rho_M)$ of compact subgroups and irreducible representations thereof following the construction of Kim and Yu (\cite[Section~7]{Kim-Yu}), which is based on \cite[Section~4]{Yu}, but twisted by the quadratic character from \cite[Section~4]{FKS} and allowing more general coefficients. If $\Coeff=\bC$, these will be types for finite products of Bernstein blocks for the groups $G^0$, $G$, and $M$, respectively. If $\Coeff=\bC$ and $K_{M^0}=M(F)_{x_0}$, then the resulting types 
are types for single Bernstein blocks.

The pair $(K_{x}^0, \rho_{x}^0)$ is the depth-zero pair attached to the depth-zero $G^0$-datum $ \bigl((G^0, M^0), \overrightarrow{r}, (x_{0}, \iota),$ $(K_{M^0}, \rho_{M^0})\bigl)$ in \cite[Section~\ref{HAI-subsec:constructionofdepthzero}]{HAI}, which we now briefly recall for the convenience of the reader.
For $x \in \cA_{x_{0}}$, we define  \spacingatend{}
\index{notation-ky}{Kay0 x@$K^{0}_{x}$}%
\index{notation-ky}{Kay0 x +@$K^{0}_{x, +}$}%
\begin{equation} \label{definitionofcompactopensubgroupsdepthzerocase}
K^{0}_{x} = K_{M^0} \cdot G^0(F)_{x, 0}
\quad
\text{ and }
\quad K^{0}_{x, +} = G^{0}(F)_{x, 0+} .
\end{equation}
If $\iota \colon \cB(M^0,F) \longrightarrow \cB(G^0,F)$ is 0-generic relative to $x$, then 
we define the irreducible smooth representation $\rho^{0}_{x}$
\index{notation-ky}{rho0x@$\rho^0_{x}$}%
of $K^{0}_{x} / K^{0}_{x, +}$ as the composition of $\rho_{M^{0}}$ with the inverse of the isomorphism $K_{M^{0}} / M^{0}(F)_{x, 0+} \isoarrow K^{0}_{x}/K^{0}_{x, +}$ that comes from the inclusion  $K_{M^{0}} \subseteq K^{0}_{x}$ .
We also regard $\rho^{0}_{x}$ as an irreducible smooth representation of $K^{0}_{x}$ that is trivial on $K^{0}_{x, +}$.

In order to construct the pairs $(K_{x}, \rho_{x})$ and $(K_{M}, \rho_{M})$,
we let $x \in \cA_{x_{0}}$ and define several compact, open subgroups of $G(F)$ and $M(F)$ as follows 
\label{page:moredefs}%
\index{notation-ky}{Kay x@$K_x$}%
\index{notation-ky}{Kay x 0+@$K_{x, 0+}$}%
\index{notation-ky}{Kay x +@$K_{x, +}$}%
\index{notation-ky}{Kay M@$K_{M}$}%
\index{notation-ky}{Kay M 0+@$K_{M, 0+}$}%
\begin{equation}
\label{definitionofcompactopensubgroupsKimYucase}
\renewcommand{\arraystretch}{1.2} 
\begin{array}{l c l}
K_{x}
&=&
K^{0}_{x} G^{1}(F)_{x,\frac{r_{0}}{2}} G^{2}(F)_{x, \frac{r_{1}}{2}}
		\cdots G^{d}(F)_{x, \frac{r_{d-1}}{2}}, \\
K_{x, 0+}
&=&
G^{0}(F)_{x, 0+} G^{1}(F)_{x,\frac{r_{0}}{2}} G^{2}(F)_{x, \frac{r_{1}}{2}}
		\cdots G^{d}(F)_{x, \frac{r_{d-1}}{2}}, \\
K_{x, +}
&=&
G^{0}(F)_{x, 0+} G^{1}(F)_{x, \frac{r_{0}}{2}+} G^{2}(F)_{x, \frac{r_{1}}{2}+}
		\cdots G^{d}(F)_{x, \frac{r_{d-1}}{2}+}, \\
K_{M}
&=&
K_{M^{0}} M^{1}(F)_{x,\frac{r_{0}}{2}} M^{2}(F)_{x, \frac{r_{1}}{2}}
		\cdots M^{d}(F)_{x, \frac{r_{d-1}}{2}}, \\
K_{M, 0+} 
&=&
M^{0}(F)_{x_{0}, 0+} M^{1}(F)_{x_{0},\frac{r_{0}}{2}} M^{2}(F)_{x_{0}, \frac{r_{1}}{2}}
		\cdots M^{d}(F)_{x_{0}, \frac{r_{d-1}}{2}}.
\end{array}
\end{equation}
We note that $K_{x} = K^{0}_{x} \cdot K_{x, 0+}$, $K_{M} = K_{M^{0}} \cdot K_{M, 0+}$, the groups $K_{M}$ and $K_{M, 0+}$ do not depend on the point $x \in \cA_{x_{0}}$, and all the above groups are the same as the groups introduced in Sections \ref{subsec:repfromHeisenbergWeil} and \ref{Section-untwisted-Heisenberg-Weil-extension} attached to the 
Heisenberg--Weil data
\index{notation-ky}{HWSigmax@$\HW(\Sigma)_{x}$}%
\[
\HW(\Sigma)_{x} = \!
\begin{cases} \!
	\bigl( 
	\left(
	G^0 \subsetneq G^1 \subsetneq \ldots \subsetneq G^d
	\right),
	(r_0, \ldots , r_{d-1}), (x, \{\iota\}), K^{0}_{x}, (\phi_0, \ldots , \phi_{d-1})
	\bigr) & \! \! \!  (r_{d-1} = r_{d}), \\
	\bigl(
	\left(
	G^0 \subsetneq G^1 \subsetneq \ldots \subsetneq G^d \subseteq G^{d+1} \coloneqq  G^{d}
	\right),
	(r_0, \ldots , r_{d}), (x, \{\iota\}), K^{0}_{x}, (\phi_0, \ldots , \phi_{d})
	\bigr) & \! \! \!  (r_{d-1} < r_{d}),
\end{cases}
\]
and $\HW(\Sigma)_{M} =$
\[
\index{notation-ky}{HWSigmaM@$\HW(\Sigma)_{M}$}%
\begin{cases}
	\bigl(
	\left(
	M^0 \subseteq M^1 \subseteq \ldots \subseteq M^d 
	\right),
	(r_0, \ldots , r_{d-1}), (x, \{\iota\}), K_{M^{0}}, (\phi_{0} \restriction_{M^{0}(F)}, \ldots , \phi_{d-1} \restriction_{M^{d-1}(F)} )
	\bigr) & \! \! (r_{d-1} = r_{d}), \\
	\bigl(
	\left(
	M^0 \subseteq  \ldots \subseteq M^d \subseteq M^{d+1} \coloneqq  M^{d}
	\right),
	(r_0, \ldots , r_{d}), (x, \{\iota\}), K_{M^{0}}, (\phi_{0} \restriction_{M^{0}(F)}, \ldots , \phi_{d} \restriction_{M^{d}(F)} )
	\bigr) & \! \!  (r_{d-1} < r_{d})
\end{cases}
\]
i.e.\ our notation in this section is consistent with the notation in the Section \ref{sec:twistedweil}.

\begin{remark}
	\label{remark:conventionsofGdatum}
	Since we follow the conventions of \cite{Yu} for the definition of a $G$-datum, we have to distinguish the above two cases when extracting the Heisenberg--Weil datum $\HW(\Sigma)_{x_0}$ from $\Sigma$, which will be used below to construct the representations $\rho_{x_0}$. On the other hand, if we followed the conventions of \cite{Fi-exhaustion, MR4357723} to define the $G$-datum, then the Heisenberg--Weil datum $\HW(\Sigma)_{x_0}$ would be obtained by simply removing the depth-zero piece from the $G$-datum $\Sigma$, see \cite[Remark~2.4]{MR4357723}. 
\end{remark}

From now on suppose that the diagram $\{ \iota \}$ is $\overrightarrow{r}/2$-generic relative to $x$, so that $\Sigma_x$ is a $G$-datum. 
\index{notation-ky}{kappantx@$\kappa^{\nt}_{x}$}%
\index{notation-ky}{kappaMnt@$\kappa_M^{\nt}$}%
We denote by $\kappa^{\nt}_{x}$ and $\kappa_{M}^{\nt}$ the representations of $K_{x}$ and $K_{M}$ defined from the Heisenberg--Weil data $\HW(\Sigma)_{x}$ and $\HW(\Sigma)_{M}$ via the non-twisted Heisenberg--Weil construction as in Notation~\ref{notationuntwistedconstruction}, respectively. In other words, $\kappa^{\nt}_{x}$, resp.,\ $\kappa_{M}^{\nt}$ is the irreducible smooth representation of $K_{x}$, resp.,\ $K_{M}$ defined via the theory of Heisenberg--Weil representations as in \cite[Section~7]{Kim-Yu}, \cite[Section~4]{Yu}, and allowing a more general coefficient field $\Coeff$ using the arguments in \cite[Section~2]{MR4460255}.

We define the irreducible representations
\index{notation-ky}{kappax@$\kappa_{x}$}%
\index{notation-ky}{kappaM@$\kappa_M$}%
\label{kappaM}
$\kappa_{x}$ of $K_{x}$ and $\kappa_{M}$ of $K_{M}$ by $\kappa^{\nt}_{x} \otimes \epsilon^{\overrightarrow{G}}_{x}$ and $\kappa_{M} = \kappa_{M}^{\nt} \otimes \bigl(
\epsilon^{\overrightarrow{G}}_{x} \restriction_{K_{M}}
\bigr)$, where $\epsilon^{\overrightarrow{G}}_{x}$ denotes the quadratic character of \cite[Section~4.1]{FKS}.
More precisely, $\epsilon^{\overrightarrow{G}}_{x}$ is the character of $K_{x}$ introduced in Notation~\ref{notationepsilonoverrightarrowG} constructed from the Heisenberg--Weil datum $\HW(\Sigma)_{x}$. 
In other words, the representation $\kappa_{x}$ is the representation defined from the Heisenberg--Weil datum $\HW(\Sigma)_{x}$ by the twisted Heisenberg--Weil construction as in Notation~\ref{notationuntwistedconstruction}.
\begin{remark}
Let $\epsilon^{\overrightarrow{M}}_{x}$ denote the character of $K_{M}$ introduced in Notation~\ref{notationepsilonoverrightarrowG} constructed from the Heisenberg--Weil datum $\HW(\Sigma)_{M}$ which is based on \cite[Section~4.1]{FKS}.
Then the restriction of the character $\epsilon^{\overrightarrow{G}}_{x}$ to $K_{M}$ does not necessarily agree with the character $\epsilon^{\overrightarrow{M}}_{x}$, see Remark~\ref{remarkexampleofepsilonGneqepsilonM} below for an example.
Thus, the representation $\kappa_{M}$ does not necessarily agree with the representation constructed from the Heisenberg--Weil datum $\HW(\Sigma)_{M}$
via the twisted Heisenberg--Weil construction.
\end{remark}
\begin{lemma}
\label{lemmaindependencyofkappaMntandkappaM}
The representations $\kappa_{M}^{\nt}$ and $\kappa_{M}$ do not depend on the point $x \in \cA_{x_{0}}$ such that the diagram $\{ \iota \}$ is $\overrightarrow{r}/2$-generic relative to $x$, and we have $\kappa^{\nt}_{x} \restriction_{K_{M}} = \kappa_{M}^{\nt}$ and $\kappa_{x} \restriction_{K_{M}} = \kappa_{M}$ for all such $x$. 
\end{lemma}
\begin{proof}
The claim that $\kappa_{M}^{\nt}$ is independent of the points follows from the construction of $\kappa_{M}^{\nt}$.
Moreover, according to Corollary~\ref{corollaryoflemmathatfollowsfromFKS412}, we have $\epsilon^{\overrightarrow{G}}_{x} \restriction_{K_{M}} = \epsilon^{\overrightarrow{G}}_{y} \restriction_{K_{M}}$ for all $x, y \in \cA_{x_{0}}$ such that $\{ \iota \}$ is $\overrightarrow{r}/2$-generic relative to $x$ and $y$. 
Thus, the representation $\kappa_{M}$ is also independent of the choice of $x$.
The remaining claims follow from the definitions.
\end{proof}

Now, we define the irreducible representations $\rho^{\nt}_{x}$ and $\rho_{x}$ of $K_{x}$ and $\rho_{M}$ of $K_{M}$ by
\index{notation-ky}{rhontx@$\rho^{\nt}_{x}$}%
\index{notation-ky}{rhox@$\rho_{x}$}%
\index{notation-ky}{rhoM@$\rho_{M}$}%
\[
\rho^{\nt}_{x} = \inf \left(
\rho^{0}_{x}
\right) \otimes \kappa^{\nt}_{x},
\qquad
\qquad
\rho_{x} =  \inf \left(
\rho^{0}_{x}
\right) \otimes 
\kappa_{x},
\qquad
\text{and}
\qquad
\rho_{M} =  \inf\left(
\rho_{M^{0}}
\right) \otimes \kappa_{M},
\]
where $\inf \left(
\rho^{0}_{x}
\right)$ denotes the inflation of the representation $\rho^{0}_{x}$ to the group $K_{x}$ via the surjection
\[
K_{x} = K^{0}_{x} \cdot K_{x, 0+} \longrightarrow K^{0}_{x} \cdot K_{x, 0+} / K_{x, 0+} 
\simeq K^{0}_{x}/ \left(
K^{0}_{x} \cap K_{x, 0+}
\right)
= K^{0}_{x} / G^{0}(F)_{x, 0+},
\] and $\inf\left(
\rho_{M^{0}}
\right)$ denotes the inflation of $\rho_{M^{0}}$ to $K_{M}$ via the surjection
\[
K_{M} = K_{M^{0}} \cdot K_{M, 0+} 
\longrightarrow K_{M^{0}} \cdot K_{M, 0+} / K_{M, 0+} 
\simeq K_{M^{0}} / \left(
K_{M^{0}}  \cap K_{M, 0+}
\right)
\simeq 
K_{M^{0}} / M^{0}(F)_{x, 0+}.
\]

\begin{remark}
\label{remarkabouttwsitedanduntwistedKimYuconstruction}
In \cite{Kim-Yu}, Kim and Yu attached to a $\bC$-valued $G$-datum $\Sigma_{x}$ the non-twisted representation $\rho^{\nt}_{x}$. 
However, some of the results of \cite{Kim-Yu} rely on \cite[Proposition~14.1]{Yu} and \cite[Theorem~14.2]{Yu}, which were pointed out in \cite{MR4357723} to be false in general.
On the other hand, according to \cite[Corollary~4.1.11, Corollary~4.1.12]{FKS}, the twisted  representation $\rho_{x}$ satisfies the analogues of these propositions.
Thus, we can apply the results of \cite{Kim-Yu}, replacing their non-twisted construction with our twisted construction.
\\ %
Moreover, while Kim and Yu (\cite[\S2.3]{Kim-Yu}) assume that $p$ is not a torsion prime for the dual root datum of $G$, they only use this assumption for \cite[5.3~Lemma]{Kim-Yu}, which remains true without this assumption (see Lemma~\ref{lemma:genericityforGimpliesgenericityforM}).
\end{remark}

\begin{remark}
\label{remarkabouttheunitaricityofthecharacters}
Let $x \in \cA_{x_{0}}$ such that the diagram $\{ \iota \}$ is $\overrightarrow{r}/2$-generic relative to $x$.
For $0 \le i \le d$, let $K^{i}_{x}$ denote the compact, open subgroup of $G^{i}(F)$ defined as in Section~\ref{subsec:repfromHeisenbergWeil} from the Heisenberg--Weil datum $\HW(\Sigma)_{x}$.
Since the character $\phi_{i}$ is trivial on the open subgroup $G^{i}(F)_{x, r_{i}+}$ of $K^{i}_{x}$, the image of the restriction of $\phi_{i}$ to the compact subgroup $K^{i}_{x}$ is contained in the group $\mu(\Coeff)$ of roots of unity in $\Coeff^{\times}$.
Thus, we can take a character 
$
\phi'_{i} \colon G^{i}(F) \longrightarrow \mu(\Coeff)
$
that agrees with $\phi_{i}$ on $K^{i}_{x}$.
In particular, if $\Coeff$ admits a nontrivial involution,
we obtain that the characters $\phi'_{i}$ are unitary
(with respect to this involution)
for all $0 \le i \le d$. 
The construction of $\rho_{x}$ implies that the representations $\rho_{x}$ of $K_x$ constructed from $\Sigma_x$ and from the $G$-datum 
\(
\bigl((\overrightarrow{G}, M^0), \overrightarrow{r}, (x, \{\iota\}), (K_{M^0}, \rho_{M^0}), \left(\phi'_0, \ldots , \phi'_d\right)
\bigr)
\)
agree.
Hence, by replacing $\phi_{i}$ with $\phi'_{i}$ if necessary, we may and do assume that the characters $\phi_i$ are unitary (but this only matters for questions of unitarity and preservation of the anti-involution introduced in
\cite[Section~\ref{HAI-Anti-involution of the Hecke algebra}]{HAI},
so a reader not interested in this property can ignore this step).
\end{remark}

For $0 \le i \le d$ and $x \in \cA_{x_{0}}$ such that the diagram $\{ \iota \}$ is $\overrightarrow{r}/2$-generic relative to $x$, we write $\widehat{\phi}_{i, x}$ for the character of $K^{0}_{x} \cdot G^{i}(F)_{x, 0} \cdot G(F)_{x, r_{i}/2+}$ defined as in \cite[Section~4]{Yu} that extends $\phi_i\restriction{K^{0}_{x} \cdot G^{i}(F)_{x, 0}}$.
We write 
\(
\theta_{x} = \prod_{i=0}^{d} \widehat{\phi}_{i, x} \restriction_{K_{x, +}},
\)
analogous to Notation \ref{notation:theta_x}. Then we have the following basic observation that we will use later and which is \cite[Proposition~4.4]{Yu} in the case of supercuspidal representations. 

\begin{lemma}[{cf.\ \cite[Proposition~4.4]{Yu}}]
\label{lemmakapparestrictiontoK+KimYuver}
The restriction of the representation $\rho_{x}$ to the group $K_{x, +}$ is $\theta_{x}$-isotypic. 
\end{lemma}
\begin{proof}
	This follows from Lemma~\ref{lemmakapparestrictiontoK+}, and the fact that the representation $\inf \left(
	\rho^{0}_{x}
	\right)$ is trivial on $K_{x, +}$
\end{proof}

If $N_{G^{0}}(M^{0})(F)_{[x_{0}]_{M^{0}}}$ normalizes the group $K_{M^0}$, e.g., if $K_{M^{0}} = M^{0}(F)_{x_{0}}$ or $M^{0}(F)_{x_{0}, 0}$, then the objects $G^0$, $G$, $M^0$, $M$, $x_0$, $K_{M^0}$, $K_M, \rho_{M^0}, \rho_M$ and the families $\left\{(K^{0}_{x}, K^{0}_{x, +}, \rho^{0}_{x})\right\}$ and $\left\{(K_{x}, K_{x, +}, \rho_{x})\right\}$ for appropriate $x \in \cA_{x_0}$ are an example for the objects with the same names in
\cite[Section \ref{HAI-subsec:comparisonsetup}]{HAI}
and we will show in Section \ref{subsec:HeckealgebraisomforKimYu} that these objects satisfy all the desired axioms of
\cite[Section \ref{HAI-sec:comparison}]{HAI}
for the choice of
\index{notation-ky}{NrhoM@$N(\rho_{M^{0}})_{[x_{0}]_{M^{0}}}$}%
\[
\Nzeroheart = N(\rho_{M^{0}})_{[x_{0}]_{M^{0}}} \coloneqq  N_{G^{0}(F)}(\rho_{M^{0}}) \cap N_{G^{0}}(M^{0})(F)_{[x_{0}]_{M^{0}}}
\]
(see \cite[Notation~\ref{HAI-notationofNrhoM}]{HAI}).
Moreover, the collection of objects \label{KimYufamilies} $G$, $M$, $x_0$, $K_M$, $\rho_M$, $K_{x}$, $K_{x, +}$, $\rho_{x}$ and the collection of objects $G^0$, $M^0$, $x_0$, $K_{M^0}$, $\rho_{M^0}$, $K^{0}_{x}$, $K^{0}_{x, +}$, $\rho^{0}_{x}$ are also two examples for objects to which
\cite[Section \ref{HAI-Structure of a Hecke algebra}]{HAI}
can be applied for the group $\Nheart = N(\rho_{M^{0}})_{[x_{0}]_{M^{0}}}$ in both cases. We will show in Section \ref{subsec:HeckealgebraisomforKimYu} 
 that both collections of objects satisfy all the axioms of
\cite[Section \ref{HAI-Structure of a Hecke algebra}]{HAI}.

\subsection{Affine hyperplanes}
\label{subsec:affinehyperplanesKimYu}
In order to apply
Sections \ref{HAI-Structure of a Hecke algebra}
and \ref{HAI-sec:comparison} of \cite{HAI}
to the objects introduced in the previous subsection, Section \ref{subsec:construction}, we introduce an appropriate set of affine hyperplanes as in
\cite[Section \ref{HAI-subsec:hyperplanes}]{HAI} as follows.
Let $0 \le i \le d$.
We fix a maximal split torus $S^{i}$ of $M^{i}$ such that $x_{0} \in \cA(G^{i}, S^{i}, F)$.
For $a \in \nobreak \Phi_{\aff}(G^{i}, S^{i}) \smallsetminus \Phi_{\aff}(M^{i}, S^{i})$, we define the affine hyperplane $H_{a, r_{i-1}/2}$
\index{notation-ky}{Har@$H_{a, r_{i-1}/2}$}%
in $\cA(G^{i}, S^{i}, F)$ by
\[
H_{a, \frac{r_{i-1}}{2}} = 
\left\{
x \in \cA(G^{i}, S^{i}, F) \mid a(x) = \frac{r_{i-1}}{2}
\right\}.
\]
Since $a \not \in \Phi_{\aff}(M^{i}, S^{i})$, the intersection
$
\cA_{x_{0}} \cap H_{a, r_{i-1}/2} 
$
is an affine hyperplane in $\cA_{x_{0}} 
= x_{0} + \left( X_{*}(A_{M}) \otimes_{\mathbb{Z}} \mathbb{R} \right)
= x_{0} + \left( X_{*}(A_{M^0}) \otimes_{\mathbb{Z}} \mathbb{R} \right)
$.
We define the locally finite set $\mathfrak{H}^{i}_{S^{i}}$ of affine hyperplanes in $\cA_{x_{0}}$ by
\[
\mathfrak{H}^{i}_{S^{i}} = \left\{
\cA_{x_{0}} \cap H_{a, \frac{r_{i-1}}{2}} \mid a \in \Phi_{\aff}(G^{i}, S^{i}) \smallsetminus \Phi_{\aff}(M^{i}, S^{i})
\right\}.
\]
Note that $\mathfrak{H}^{0}_{S^{0}}$ is the set of affine hyperplanes that we introduced in the depth-zero setting in \cite[Section \ref{HAI-subsec:affinehyperplanes-depthzero}]{HAI} for the group $G^0$ with Levi subgroup $M^0$.

\begin{lemma} \label{lemma:independenceofS0} The set of affine functionals
	\(
	\left\{ a\restriction_{\cA_{x_0}} \mid a \in \Phi_{\aff}(G^{i}, S^{i}) \smallsetminus \Phi_{\aff}(M^{i}, S^{i})
	\right\}
	\) on $\cA_{x_0}$ and the set $\mathfrak{H}^{i}_{S^{i}}$
	do not depend on the choice of a maximal split torus $S^{i}$ of $M^i$.
\end{lemma}
\begin{proof}
	The proof is the same as the proof of \cite[Lemma~\ref{HAI-lemma:independenceofS0}]{HAI} replacing $G$, $M$, $S$ and $H_a$ in that proof by $G^i$, $M^i$, $S^i$ and $H_{a, \frac{r_{i-1}}{2}}$, and replacing the condition ``$=0$'' by ``$=\frac{r_{i-1}}{2}$''.
\end{proof}

We define the locally finite set $\mathfrak{H}$
\index{notation-ky}{H_@$\mathfrak{H}$}%
of affine hyperplanes in $\cA_{x_{0}}$ by
\[
\mathfrak{H} = \bigcup_{0 \le i \le d} \mathfrak{H}^{i}_{S^{i}}.
\]
The definition of $\mathfrak{H}$ implies that for $x \in \cA_{x_{0}}$, the diagram $\{ \iota \}$ is $\overrightarrow{r}/2$-generic relative to $x$ if and only if $x$ is not contained in any affine hyperplane $H \in \mathfrak{H}$, that is,
\index{notation-ky}{A gen@$\cA_{\gen}$}%
$x \in \cA_{\gen}= \cA_{x_{0}} \smallsetminus \left(
\bigcup_{H \in \mathfrak{H}} H\right)$.
In particular, we have $x_{0} \in \cA_{\gen}$.

For $x, y \in \cA_{\gen}$, as in  \cite[\S\ref{HAI-subsec:hyperplanes}]{HAI},
we define the subset $\mathfrak{H}_{x, y}$ of $\mathfrak{H}$ by
\[
\index{notation-ky}{H_xy@$\mathfrak{H}_{x, y}$}
\mathfrak{H}_{x, y} = \left\{
H \in \mathfrak{H} \mid \text{$x$ and $y$ are on opposite sides of $H$}
\right\}
\]
and we write \index{notation-ky}{d(,)@$d(\phantom{x},\phantom{y})$}
$
d(x, y) = \# \mathfrak{H}_{x, y}
$,
which is finite because $\mathfrak{H}$ is locally finite.

\begin{lemma}
\label{lemmaNzeroheartpreservesH}
The action of $N_{G^{0}}(M^{0})(F)_{[x_{0}]_{M^{0}}}$ on $\mathcal{A}_{x_{0}}$ preserves the set $\mathfrak{H}$.
\end{lemma}
\begin{proof}
The proof is analogous to the proof of
\cite[Lemma~\ref{HAI-lemmaNzeroheartpreservesH}]{HAI}.
\end{proof}

\subsection{Hecke algebra isomorphisms for the types constructed by Kim--Yu} \label{subsec:HeckealgebraisomforKimYu}
From now on, we assume that the group $K_{M^0}$ is normalized by $N_{G^{0}}(M^{0})(F)_{[x_{0}]_{M^{0}}}$ in order to apply the results in \cite[Section~\ref{HAI-subsec:structureofheckefordepthzerotypes}]{HAI}, especially \cite[Proposition~\ref{HAI-propproofofaxiombijectionofdoublecoset}]{HAI}, which states that the support of the Hecke algebra attached to $(K_x^0, \rho_x^0)$ is given by $K^{0}_{x} \cdot N(\rho_{M^{0}})_{[x_{0}]_{M^{0}}} \cdot K^{0}_{x}$.
For instance, if we choose $K_{M^0} = M^0(F)_{x_{0}}$ or $M^0(F)_{x_{0}, 0}$, then this assumption is satisfied.
If $\Coeff=\bC$, then the case of $K_{M^{0}} = M^{0}(F)_{x_{0}}$ corresponds to types for single Bernstein blocks.

In this subsection, we will prove Theorem \ref{heckealgebraisomforKim-Yutype} that states that there exists a support-preserving algebra isomorphism
\(
\cH(G^{0}(F), \rho^{0}_{x_{0}}) \isoarrow \cH(G(F), \rho_{x_{0}})
\) 
by verifying all the required axioms from
Sections \ref{HAI-Structure of a Hecke algebra}
and \ref{HAI-sec:comparison} of \cite{HAI}
that allow us to apply
\cite[Theorem \ref{HAI-thm:isomorphismtodepthzero} and Corollary \ref{HAI-cor:starpreservation}]{HAI}. 
We recall that we have constructed in Section \ref{subsec:construction} the two families
\[\cK^{0} =
\left\{
(K^{0}_{x}, K^{0}_{x, +}, (\rho^{0}_{x}, V_{\rho^{0}_{x}}))
\right\}_{x \in \cA_{\gen}}
\quad
\text{ and }
\quad
\cK =
\left\{
(K_{x}, K_{x, +}, (\rho_{x}, V_{\rho_{x}}))
\right\}_{x \in \cA_{\gen}} 
\] 
of quasi-$G^{0}$-cover-candidates and quasi-$G$-cover-candidates, respectively, as defined at the beginning of
Section \ref{HAI-subsec:familyofcovers} of \cite{HAI}.
\begin{proposition}
\label{prop:K0satisfiesaxiomaboutHNheartandKandaxiombijectionofdoublecoset}
The family $\cK^{0}$ satisfies
Axioms~\ref{HAI-axiomaboutHNheartandK} and \ref{HAI-axiombijectionofdoublecoset} of \cite{HAI} for the group $\Nzeroheart = N(\rho_{M^{0}})_{[x_{0}]_{M^{0}}}$.
\end{proposition}
\begin{proof}
 Since the set of affine hyperplanes
$\mathfrak{H}^{0}_{S^{0}} \subset \cA_{x_0}$
used to define the set of generic points 
in the depth-zero setting
in \cite[Section \ref{HAI-subsec:affinehyperplanes-depthzero}]{HAI}
is a subset of $\mathfrak{H}$,
our set $\cA_{\gen}$ of generic points is contained
in the set of generic points used in
\cite[Section \ref{HAI-subsec:affinehyperplanes-depthzero}]{HAI}.
	 Hence, the properties of Axioms~\ref{HAI-axiomaboutHNheartandK} and \ref{HAI-axiombijectionofdoublecoset} of \cite{HAI} other than Axiom~\ref{HAI-axiomaboutHNheartandK}~\eqref{HAI-axiomaboutHNheartandKNinvarianceofH} follow from \cite[Lemma~\ref{HAI-proofofaxiomaboutHNheartandK}]{HAI} and \cite[Proposition~\ref{HAI-propproofofaxiombijectionofdoublecoset}]{HAI}.
	 Axiom~\ref{HAI-axiomaboutHNheartandK}~\eqref{HAI-axiomaboutHNheartandKNinvarianceofH} follows from Lemma~\ref{lemmaNzeroheartpreservesH}.
\end{proof}

We will now prove that the family $\cK$ also satisfies Axiom \ref{HAI-axiomaboutHNheartandK} of \cite{HAI}.
\begin{lemma}
\label{proofofaxiomaboutHNheartandK}
\mbox{}
\begin{enumerate}[(1)] 
\item
For every $x \in \cA_{\gen}$, we have
\begin{enumerate}
	\item \(
	K_{n x} = n K_{x} n^{-1}
	\)
	and
	\(
	K_{n x, +} = n K_{x, +} n^{-1}
	\) for $n \in N(\rho_{M^{0}})_{[x_{0}]_{M^{0}}}$,
\item 
the pair $(K_{x}, \rho_{x})$ is a quasi-$G$-cover of $(K_{M}, \rho_{M})$,
\item
$K_{x} = K_{M} \cdot K_{x, +},$
\item
$K_{x, +} = \left(
K_{x, +} \cap U(F)
\right) \cdot \left(K_{x, +} \cap M(F)\right) \cdot \left(
K_{x, +} \cap \overline{U}(F)
\right)$ for all $U \in \cU(M)$.
\end{enumerate}
Moreover, the group $K_{x, +} \cap M(F)$ is independent of the point $x \in \cA_{\gen}$.
\item
For $x, y, z \in \cA_{\gen}$ such that 
\(
d(x, y) + d(y, z) = d(x, z),
\)
there exists $U \in \cU(M)$ such that
\[
K_{x} \cap U(F) \subseteq K_{y} \cap U(F) \subseteq K_{z} \cap U(F)
\quad
\text{ and }
\quad
K_{z} \cap \overline{U}(F)  \subseteq K_{y} \cap \overline{U}(F) \subseteq K_{x} \cap \overline{U}(F).
\]
\end{enumerate}
Thus, the family $\cK$ satisfies Axiom~\ref{HAI-axiomaboutHNheartandK} of \cite{HAI}
for the group $\Nheart = N(\rho_{M^{0}})_{[x_{0}]_{M^{0}}}$.
\end{lemma}
\begin{proof}
The first claim follows from the definitions, Lemma~\ref{lemmaindependencyofkappaMntandkappaM}, and \cite[4.3~Proposition, Theorem~7.5]{Kim-Yu}.
A proof for the second claim can be obtained by making several replacements
in the proof of the second claim of
Lemma~\ref{HAI-proofofaxiomaboutHNheartandK} of \cite{HAI}.
Specifically,
replace the symbol $G$ by $G^i$,
``$> 0 $'' and ``$<0$''
by ``$> \frac{r_{i-1}}{2}$''
and ``$< \frac{r_{i-1}}{2}$'',
and
$H_a$ by $H_{a, \frac{r_{i-1}}{2}}$ for some $0 \le i \le d$ such that there exists  $a \in \Phi_{\aff}(G^{i}, S^{i})$ whose gradient $Da \in \Phi(G^{i}, S^{i})$ occurs in the adjoint representation of $S^{i}$ on the Lie algebra of $U$ and
	\(
	a(x) > \tfrac{r_{i-1}}{2}
	\)
	and 
	\(
	a(y) < \tfrac{r_{i-1}}{2}
	\). 
The first two properties of Axiom~\ref{HAI-axiomaboutHNheartandK} of \cite{HAI}
follow from Lemma~\ref{lemmakapparestrictiontoK+KimYuver} and Lemma~\ref{lemmaNzeroheartpreservesH}, the remaining properties from the first two claims of this lemma.
\end{proof}

Next, we will prove that the quadruple $\left(
K_{M^{0}}, \rho_{M^{0}}, K_{M}, \rho_{M}
\right)$ together with the group $N(\rho_{M^{0}})_{[x_{0}]_{M^{0}}}$ satisfy Axiom~\ref{HAI-axiomaboutKM0vsKM} of \cite{HAI}.
To do so, we will define an explicit extension $\widetilde\kappa_{M}$ of $\kappa_{M}$ to the group 
\[
N_{G^{0}}(M^{0})(F)_{[x_0]_{M^{0}}} \cdot K_{M} \supseteq N(\rho_{M^{0}})_{[x_{0}]_{M^{0}}} \cdot K_{M},
\]
which might be of independent interest. 

Let $\widetilde \kappa_{M}^{\nt}$ denote the extension of $\kappa_{M}^{\nt}$ to the group $N_{G^{0}}(M^{0})(F)_{[x_0]_{M^{0}}} \cdot K_{M}$ defined in Proposition~\ref{proposition:extensionofuntwistedWeilHeisenbergrepresentation}.
For the reader who prefers to skip the details of Section~\ref{sec:twistedweil} on a first reading, here is a description of $\widetilde \kappa_{M}^{\nt}$:
The conjugation action of $N_{G^{0}}(M^{0})(F)_{[x_0]_{M^{0}}}$ on $M(F)$ induces symplectic automorphisms on the symplectic spaces appearing in Yu's construction of $\kappa_{M}^{\nt}$. Hence, we obtain a group homomorphism from $N_{G^{0}}(M^{0})(F)_{[x_0]_{M^{0}}}$ to the product of the associated symplectic groups. Then we can define the action of $N_{G^{0}}(M^{0})(F)_{[x_0]_{M^{0}}}$ by composing this homomorphism with the tensor product of the Weil representations, and take the tensor product of the resulting representation with $\prod_{0 \leq i \leq d} \phi_{i} \restriction_{N_{G^{0}}(M^{0})(F)_{[x_0]_{M^{0}}}}$.

\begin{definition} \label{definitionofkappatilde}
	We define the representation
\index{notation-ky}{kappaM_tilde@$\widetilde\kappa_M$}%
$\widetilde \kappa_{M}$ of $N_{G^{0}}(M^{0})(F)_{[x_0]_{M^{0}}} \cdot K_{M}$ by 
\[\widetilde \kappa_{M} = \widetilde \kappa_{M}^{\nt} \otimes \widetilde{\epsilon}^{\overrightarrow{G}}_{x_{0}}, 
\] where 
	$\widetilde{\epsilon}^{\overrightarrow{G}}_{x_{0}}$ denotes the extension of $\epsilon^{\overrightarrow{G}}_{x_{0}} \restriction_{K_{M}}$ to $N_{G^{0}}(M^{0})(F)_{[x_0]_{M^{0}}} \cdot K_{M}$ defined by $\widetilde{\epsilon}^{\overrightarrow{G}}_{x_{0}}\restriction_{K_{M, 0+}} = 1$ and
	\[
\widetilde{\epsilon}^{\overrightarrow{G}}_{x_{0}} \restriction_{N_{G^{0}}(M^{0})(F)_{[x_0]_{M^{0}}}} = \prod_{i=1}^{d} \widetilde{\epsilon}^{G^{i}/G^{i-1}}_{x_{0}}
\restriction_{N_{G^{0}}(M^{0})(F)_{[x_0]_{M^{0}}}} ,
\]
where $\widetilde{\epsilon}^{G^{i}/G^{i-1}}_{x_{0}}$
denotes the character introduced in
Definition \ref{definitiontildeepsilon}.
\end{definition}

\begin{proposition}
	\label{propproofofaxiomaboutKM0vsKM}
	The restriction of the representation $\widetilde \kappa_{M}$ to $K_M$ agrees with $\kappa_M$.
	
	Thus, the quadruple $\left(
	K_{M^{0}}, \rho_{M^{0}}, K_{M}, \rho_{M}
	\right)$ together with the group $\Nzeroheart \coloneqq N(\rho_{M^{0}})_{[x_{0}]_{M^{0}}}$ satisfy Axiom~\ref{HAI-axiomaboutKM0vsKM} of \cite{HAI}.\\
	Moreover, if $\Coeff$ admits a nontrivial involution,
then
 $\widetilde \kappa_{M}$
  is unitary.
\end{proposition}
\begin{proof}
	The first claim follows from the definitions of $\kappa_{M}$ and $\widetilde \kappa_{M}$ and Theorem \ref{twistextension}.
	Then the claim about Axiom~\ref{HAI-axiomaboutKM0vsKM} of \cite{HAI}
being satisfied follows from the definitions, Lemma~\ref{lemma:NG0M0xM0normalizescompactopensungroups}, and $N_{G^{0}}(M^{0})(F)_{[x_{0}]_{M^{0}}}\supseteq N(\rho_{M^{0}})_{[x_{0}]_{M^{0}}} \supseteq \nobreak K_{M^{0}}$.
	
	Suppose that $\Coeff$ admits a nontrivial involution.
	According to Remark~\ref{remarkabouttheunitaricityofthecharacters}, we arranged for the characters $\phi_{i}$ to be unitary for all $0 \le i \le d$.
	Then the claim follows from Lemma~\ref{lemma:kappatildeisunitary}, the definition of $\widetilde \kappa_{M}$, and the fact that the characters $\widetilde{\epsilon}^{G^{i}/G^{i-1}}_{x_{0}}$ are unitary for $1 \le i \le d$.
\end{proof}

\begin{proposition}
\label{proofofaxiomaboutK0vsK}
For each $x \in \cA_{\gen}$,
we have
\[
I_{G(F)}(\rho_{x}) = K_{x} \cdot I_{G^{0}(F)}(\rho^{0}_{x}) \cdot K_{x}= K_{x} \cdot N(\rho_{M^{0}})_{[x_{0}]_{M^{0}}} \cdot K_{x} .
\] 
Thus, the families $\cK^{0}$ and $\cK$ satisfy
Axiom~\ref{HAI-axiomaboutK0vsK} of \cite{HAI},
and the family $\cK$ satisfies
Axiom~\ref{HAI-axiombijectionofdoublecoset} of \cite{HAI}
for $\Nheart=N(\rho_{M^{0}})_{[x_{0}]_{M^{0}}}$.
\end{proposition}
\begin{proof}
The first equality follows from \cite[Theorem~8.1]{Kim-Yu} and Remark~\ref{remarkabouttwsitedanduntwistedKimYuconstruction}, which points out that the results of Kim and Yu apply to the twisted construction that we use in this section, and not to the non-twisted one. The second equality follows from \cite[Proposition \ref{HAI-propproofofaxiombijectionofdoublecoset}]{HAI}.
Using the definition of $K_{x}$ and $\rho_{x}$ in this section, we see that Axiom~\ref{HAI-axiomaboutK0vsK} of \cite{HAI} is satisfied.
\end{proof}

Now we introduce some groups that will be used to show that the families $\cK^{0}$ and $\cK$ satisfy
Axiom~\ref{HAI-axiomextensionoftheinductionofkappa} of \cite{HAI}.

\begin{notation}
\label{notationK0xyKxy0+kappaxy}
Let $x, y \in \cA_{\gen}$ with $d(x, y) = 1$.
We denote by
$H_{x,y} \in \mathfrak{H}$ the unique hyperplane that satisfies $\mathfrak{H}_{x, y} = \{ H_{x,y} \}$
and define the compact, open subgroups
\index{notation-ky}{Kay0 x y@$K^{0}_{x, y}$}%
$K^{0}_{x, y}$ of $G^{0}(F)$ and
\index{notation-ky}{Kay x y 0plus@$K_{x, y; 0+}$}%
$K_{x, y; 0+}$ of $G(F)$ by $K^{0}_{x, y} = K^{0}_{h}$ and $K_{x, y; 0+} = K_{h, 0+}$,
where $h \in H_{x,y}$ is the unique point for which $h = x + t \cdot (y - x)$ for some $0 < t < 1$, and $K^{0}_{h}$ and $K_{h, 0+}$ were defined in \eqref{definitionofcompactopensubgroupsdepthzerocase} and \eqref{definitionofcompactopensubgroupsKimYucase}.
We also define the irreducible smooth representation
\index{notation-ky}{kappaxy@$\kappa_{x, y}$}%
$\kappa_{x, y}$ of
$K_{x, y} \coloneqq  K^{0}_{x, y} \cdot K_{x, y; 0+} = K_{h}$ by $\kappa_{x, y} = \kappa_{h}$,
where $\kappa_{h}$ denotes the representation obtained from the Heisenberg--Weil datum 
\[
\HW(\Sigma)_{h} =\!
\begin{cases} \!
\bigl(
\left(
G^0 \subsetneq G^1 \subsetneq \ldots \subsetneq G^d
\right),
(r_0, \ldots , r_{d-1}), (h, \{\iota\}), K^{0}_{h}, (\phi_0, \ldots , \phi_{d-1})
\bigr) & \! \! (r_{d-1} = r_{d}), \\
\bigl(
\left(
G^0 \subsetneq G^1 \subsetneq \ldots \subsetneq G^d \subseteq G^{d+1} \coloneqq  G^{d}
\right),
(r_0, \ldots , r_{d}), (h, \{\iota\}), K^{0}_{h}, (\phi_0, \ldots , \phi_{d})
\bigr) & \! \! (r_{d-1} < r_{d})
\end{cases}
\]
via the twisted Heisenberg--Weil construction, see Notation \ref{notationuntwistedconstruction}.
\end{notation} 

\begin{lemma}
\label{proofofaxiomextensionoftheinductionofkappa}
Let $x, y \in \cA_{\gen}$ such that $d(x, y) = 1$.
Then the triple $(K^{0}_{x, y}, K_{x, y; 0+}, \kappa_{x, y})$ in Notation~\ref{notationK0xyKxy0+kappaxy} satisfies the following properties:
\begin{enumerate}[(1)]
\item
$K^{0}_{x, y}$ contains $K^{0}_{x}$ and $K^{0}_{y}$.
\item
$K_{x, y; 0+}$ is normalized by the group $K^{0}_{x, y}$, and we have
\[
K_{x, 0+} \subset \left(
G^{0}(F) \cap K_{x, 0+}
\right)
\cdot K_{x, y; 0+}
\quad
\text{ and }
\quad
K_{y, 0+} \subset \left(
G^{0}(F) \cap K_{y, 0+}
\right)
\cdot K_{x, y; 0+}.
\]
\item
The group $G^{0}(F) \cap K_{x, y; 0+}$ is contained in the kernels of $\rho^{0}_{x}$ and $\rho^{0}_{y}$.
\item
\label{claim:extensionoftheinductionofkappakappasxirreducibleinKimYusetting}
The restriction of $\kappa_{x, y}$ to $K_{x, y; 0+}$ is irreducible.
\item
\label{claim:extensionoftheinductionofkappacompatibilitywiththecompactindinKimYusetting}
We have isomorphisms
\[
\kappa_{x, y} \restriction_{K^{0}_{x} \cdot K_{x, y; 0+}} \isoarrow \ind_{K_{x}}^{K^{0}_{x} \cdot K_{x, y; 0+}} (\kappa_{x})
\quad
\text{ and }
\quad
\kappa_{x, y} \restriction_{K^{0}_{y} \cdot K_{x, y; 0+}} \isoarrow \ind_{K_{y}}^{K^{0}_{y} \cdot K_{x, y; 0+}} (\kappa_{y}).
\]
\end{enumerate}
Thus, the families $\cK^{0}$ and $\cK$ satisfy
Axiom~\ref{HAI-axiomextensionoftheinductionofkappa} of \cite{HAI}.
\end{lemma}
\begin{proof}
\addtocounter{equation}{-1}
\begin{subequations}
Since $d(x, y) = 1$, the definition of $\mathfrak{H}$ implies that we have
\begin{equation}
\label{inclusionsbetweenGihandGiz}
G^{i}(F)_{h, \frac{r_{i-1}}{2} +} \subseteq G^{i}(F)_{z, \frac{r_{i-1}}{2} +} \subseteq G^{i}(F)_{z, \frac{r_{i-1}}{2}} \subseteq G^{i}(F)_{h, \frac{r_{i-1}}{2}}
\end{equation}
for all $0 \le i \le d$ and $z \in \{x, y\}$, where $h \in H_{x,y}$ is the unique point for which $h = x + t \cdot (y - x)$ for some $0 < t < 1$,
as in Notation \ref{notationK0xyKxy0+kappaxy}.
The first three claims of the lemma follow from \eqref{inclusionsbetweenGihandGiz}.
Claim \eqref{claim:extensionoftheinductionofkappakappasxirreducibleinKimYusetting} follows from Lemma~\ref{irreducibilityofkapparestrictiontoJ}.
Claim \eqref{claim:extensionoftheinductionofkappacompatibilitywiththecompactindinKimYusetting} follows from \eqref{inclusionsbetweenGihandGiz} and Corollary~\ref{corollaryofpropositioninductiontwistedsequencever}.
\end{subequations}
\end{proof}

Recall from Definition \ref{HAI-definitionKrelevant} of \cite{HAI} that a hyperplane $H \in \mathfrak{H}$ is called $\cK$-relevant, resp.,\ $\cK^0$-relevant if there exists $x, y \in \cA_{\gen}$ such that 
$
\mathfrak{H}_{x, y} = \left\{
H
\right\}$
and
\(
\Theta_{x \mid y} \circ \Theta_{y \mid x} \notin \Coeff \cdot \id_{\ind_{K_{x}}^{G(F)}(\rho_{x})} 
\), resp.,\ $\Theta^0_{x \mid y} \circ \Theta^0_{y \mid x} \notin \Coeff \cdot \id_{\ind_{K^0_{x}}^{G^0(F)}(\rho^0_{x})}$, where the intertwining operators $\Theta_{x \mid y}$, $\Theta_{y \mid x}$, $\Theta^0_{x \mid y}$, and $\Theta^0_{y \mid x}$ are defined in \cite[\S\S\ref{HAI-subsec:intertwiningop},\ref{HAI-subsec:relevance}]{HAI}.
Following Definition \ref{HAI-definitionKrelevant}, we denote by $\mathfrak{H}_{\Krel}$, resp.,\ $\mathfrak{H}_{\Kzrel}$
the set of hyperplanes that are $\cK$-relevant, resp.,\ $\cK^0$-relevant.
\begin{lemma}
\label{lemmareloodcontainedindepthzero}
We have $
\mathfrak{H}_{\Krel}
= \nobreak \mathfrak{H}_{\Kzrel}
\subset \nobreak \mathfrak{H}^{0}_{S^{0}}
$, where $S^0$ is any maximal split torus of $M^{0}$ such that $x_{0} \in \cA(G^{0}, S^{0}, F)$.
\end{lemma}
\begin{proof}
The equality follows from Corollary~\ref{HAI-corollaryKrel=K0rel} of \cite{HAI}, whose assumptions are satisfied by Lemma~\ref{proofofaxiomaboutHNheartandK}, Proposition~\ref{propproofofaxiomaboutKM0vsKM}, Proposition~\ref{proofofaxiomaboutK0vsK}, and Lemma~\ref{proofofaxiomextensionoftheinductionofkappa}.
We will prove the inclusion.
Suppose that $H \in \mathfrak{H} \smallsetminus \mathfrak{H}^{0}_{S^{0}}$.
Then, for all $x, y \in \cA_{\gen}$ with $\mathfrak{H}_{x,y} = \{H\}$, the definitions of the family $\cK^{0}$ and the intertwining operator
$\Theta^{0}_{y \mid x}$
imply that we have $(K^{0}_{x}, \rho^{0}_{x}) = (K^{0}_{y}, \rho^{0}_{y})$ and
$
\Theta^{0}_{y \mid x} = \id_{\ind_{K^{0}_{x}}^{G^{0}(F)} (\rho^{0}_{x})}
$.
Hence, the affine hyperplane $H$ is not $\cK^{0}$-relevant, that is, $H \not \in \mathfrak{H}_{\Kzrel}$.
\end{proof}

We set $\Nzeroheart=N(\rho_{M^{0}})_{[x_{0}]_{M^{0}}}$ and recall from
\cite[Definition \ref{HAI-definitionofWheart}]{HAI} that 
\[
\Wzeroheart
\coloneqq
\Nzeroheart / \bigl( \Nzeroheart  \cap K_{M^0} \bigr) 
=
N(\rho_{M^{0}})_{[x_{0}]_{M^{0}}} / K_{M^0}
\]
and from \cite[Section \ref{HAI-subsection:indexinggroup}]{HAI}
that
\(
W_{\Kzrel} \coloneqq \langle s_{H} \mid H \in \mathfrak{H}_{\Kzrel} \rangle
\)
with set of simple reflections $S_{\Kzrel}$,
\index{notation-ky}{S K0 rel@$S_{\Kzrel}$}%
see \cite[Notation \ref{HAI-notationsimplereflections}]{HAI}.
Here, for a hyperplane $H \in \mathfrak{H}$,
we let $s_H$ denote the corresponding reflection.
Similarly, given a reflection $s$ of $\cA_{x_0}$,
we let $H_s$ denote the hyperplane fixed by $s$.
\begin{proposition}
\label{prop:axiomexistenceofRgrpandaxiomaboutdimensionofend}
The group $\Wzeroheart$ satisfies Axiom~\ref{HAI-axiomexistenceofRgrp} of \cite{HAI} with a normal subgroup $\Waffz$\index{notation-ky}{W aff0@$\Waffz$} of $\Wzeroheart$, and the family $\cK^{0}$ satisfies Axiom~\ref{HAI-axiomaboutdimensionofend} of \cite{HAI} with the group $K'_{x,s}=K^{0}_{x, s x}$ for each $s \in S_{\Kzrel}$ and $x \in \cA_{\gen}$ such that $\mathfrak{H}_{x, s x} = \{ H_{s} \}$.
\end{proposition}
\begin{proof}
Since the set of affine hyperplanes
$\mathfrak{H}^{0}_{S^{0}} \subset \cA_{x_0}$
used to define the set of generic points 
in the depth-zero setting
in \cite[Section \ref{HAI-subsec:affinehyperplanes-depthzero}]{HAI}
is a subset of $\mathfrak{H}$,
our set $\cA_{\gen}$ of generic points is contained
in the set of generic points used in
\cite[Section \ref{HAI-subsec:affinehyperplanes-depthzero}]{HAI}.
Therefore,
the proposition follows from \cite[Proposition~\ref{HAI-proofofaxiomexistenceofRgrpandaxiomsHisinKxy}]{HAI}
and Lemma~\ref{lemmareloodcontainedindepthzero}.
\end{proof}

While we are now already in position to conclude the main result of this subsection, Theorem \ref{heckealgebraisomforKim-Yutype},
let us first note a corollary of Proposition~\ref{prop:axiomexistenceofRgrpandaxiomaboutdimensionofend} based on
\cite[Section \ref{HAI-subsec:Heckeisom}]{HAI}.

\begin{corollary}\label{corofproofofaxiomexistenceofRgrpandaxiomsHisinKxyaxiomaboutdimensionofend}
	The family $\cK$ satisfies
	Axiom~\ref{HAI-axiomaboutdimensionofend} of \cite{HAI} with $\Nheart=N(\rho_{M^{0}})_{[x_{0}]_{M^{0}}}$ and the group $K'_{x,s}=K_{x, s x}$, see Notation \ref{notationK0xyKxy0+kappaxy}, for each $s \in S_{\Kzrel}$ and $x \in \cA_{\gen}$ such that $\mathfrak{H}_{x, s x} = \{ H_{s} \}$.
\end{corollary}
\begin{proof}
	The corollary follows from
\cite[Lemma \ref{HAI-lemmaaxiomstrongerthanaboutdimensionofendimpliesaxiomaboutdimensionofend}]{HAI}.
Note that we can apply the lemma because the family $\cK^0$ satisfies Axioms~\ref{HAI-axiomaboutHNheartandK} and \ref{HAI-axiombijectionofdoublecoset} of \cite{HAI}
	by Proposition~\ref{prop:K0satisfiesaxiomaboutHNheartandKandaxiombijectionofdoublecoset},
	the family $\cK$ satisfies Axiom~\ref{HAI-axiomaboutHNheartandK} of \cite{HAI} by 	 
	Lemma~\ref{proofofaxiomaboutHNheartandK},
Axioms \ref{HAI-axiomaboutKM0vsKM}, \ref{HAI-axiomaboutK0vsK}, and \ref{HAI-axiomextensionoftheinductionofkappa} of \cite{HAI} are satisfied by 	 
	Proposition~\ref{propproofofaxiomaboutKM0vsKM}, 
	Proposition~\ref{proofofaxiomaboutK0vsK}, 
	and 
	Lemma~\ref{proofofaxiomextensionoftheinductionofkappa}, 
	and Axioms \ref{HAI-axiomexistenceofRgrp} and \ref{HAI-axiomaboutdimensionofend} of \cite{HAI} for $\cK^0$ hold by Proposition~\ref{prop:axiomexistenceofRgrpandaxiomaboutdimensionofend}.
	\end{proof}

\begin{theorem}
	\label{heckealgebraisomforKim-Yutype} 
	There exists a support-preserving algebra isomorphism
	\[
	\Isom \colon \cH(G^{0}(F), \rho^{0}_{x_{0}}) \isoarrow \cH(G(F), \rho_{x_{0}}).
	\]
	If $\Coeff$ admits a nontrivial involution,
then there exists such an isomorphism that preserves the corresponding anti-involutions
on both sides defined in
\cite[Section~\ref{HAI-Anti-involution of the Hecke algebra}]{HAI}.
\end{theorem}
\begin{proof}
	The statement follows from
	\cite[Theorem \ref{HAI-thm:isomorphismtodepthzero} and Corollary \ref{HAI-cor:starpreservation}]{HAI},
	whose assumptions are satisfied by 
	 Proposition~\ref{prop:K0satisfiesaxiomaboutHNheartandKandaxiombijectionofdoublecoset},  
	  Lemma~\ref{proofofaxiomaboutHNheartandK}, 
	 Proposition~\ref{prop:axiomexistenceofRgrpandaxiomaboutdimensionofend}, 
	 Proposition~\ref{propproofofaxiomaboutKM0vsKM}, 
	 Proposition~\ref{proofofaxiomaboutK0vsK}, 
	  and 
	 Lemma~\ref{proofofaxiomextensionoftheinductionofkappa}. 
\end{proof}

The isomorphism $\Isom$ in the above theorem is described more explicitly in
\cite[Theorem \ref{HAI-thm:explicitisom}]{HAI}.

\subsection{The structure of Hecke algebras attached to the types constructed by Kim--Yu}
\label{subsec:structureKimYuHeckealgebra}

Since we have shown that all the axioms of
Sections \ref{HAI-Structure of a Hecke algebra} and
\ref{HAI-sec:comparison} of \cite{HAI}
are satisfied in the setting of the present section, we also obtain that the Hecke algebras attached to $(K_x^0, \rho_x^0)$ and $(K_{x_0}, \rho_{x_0})$ are isomorphic to a semi-direct product of a twisted group algebra with an affine Weyl group.

\begin{theorem}
	\label{theoremstructureofheckeforKimYu}
	We have isomorphisms of $\Coeff$-algebras
	\[
 \cH(G(F), \rho_{x_{0}}) \simeq 	\cH(G^{0}(F), \rho^{0}_{x_{0}}) \simeq  \Coeff[\Wzeroz, \muTzero] \ltimes \cH_\Coeff(\Waffz, q),
	\] 
where
\begin{itemize}
	\item 
$\Wzeroz$ denotes the subgroup of length-zero elements of $\Wzeroheart$ 
defined in
\cite[Notation~\ref{HAI-notation:Wzero}]{HAI},
\item 
$\muTzero$ denotes the restriction to $\Wzeroz \times \Wzeroz$ of the $2$-cocycle introduced in
\cite[Notation~\ref{HAI-notationofthetwococycle}]{HAI}
for a choice of a family $\cT^0$ satisfying the properties of
\cite[Choice~\ref{HAI-choice:tw}]{HAI} for the pair $(K_{M^0}, \rho_{M^0})$
and the collections of operators $\{\Phi^0_w\}$ and $\{\Phi^0_t\}$ defined on page \pageref{HAI-phi0w-page} of \cite{HAI},
\item $q$ denotes the parameter function $s \mapsto q_{s}$ appearing in
\cite[Choice~\ref{HAI-choice:tw}\eqref{HAI-conditionofthechoicequadraticrelations}]{HAI},
\item $\Coeff[\Wzeroz, \muTzero]$ denotes the twisted group algebra
recalled in Notation \ref{HAI-notn:algebras}(\ref{HAI-item:twisted-group-algebra})
of \cite{HAI},
and
\item $\cH_\Coeff(\Waffz, q)$ denotes the affine Hecke algebra with $\Coeff$-coefficients associated to the affine Weyl group $\Waffz$ with set of generators $S_{\Kzrel}$ and the parameter function $q$
recalled in Notation \ref{HAI-notn:algebras}(\ref{HAI-item:affine-hecke-algebra})
 of \cite{HAI}.
\end{itemize}
If $\Coeff$ admits a nontrivial involution,
then we can choose $\cT^0$ as in \cite[Choice~\ref{HAI-choice:star}]{HAI},
and the above isomorphisms can be chosen to preserve the anti-involutions
on each algebra defined in
\cite[Section~\ref{HAI-Anti-involution of the Hecke algebra}]{HAI}.
\end{theorem}
\begin{proof}
The statement follows from
\cite[Theorem~\ref{HAI-theoremstructureofhecke} and Proposition~\ref{HAI-propstarpreservationabstractheckevsourhecke}]{HAI}, whose assumptions are satisfied by 
	 Proposition~\ref{prop:K0satisfiesaxiomaboutHNheartandKandaxiombijectionofdoublecoset}, 
	 Lemma~\ref{proofofaxiomaboutHNheartandK}, 
	 Proposition~\ref{proofofaxiomaboutK0vsK},  
	 Proposition~\ref{prop:axiomexistenceofRgrpandaxiomaboutdimensionofend}, 
	 and
	 Corollary~\ref{corofproofofaxiomexistenceofRgrpandaxiomsHisinKxyaxiomaboutdimensionofend}. 
\end{proof}

\subsection{Application: Reduction to depth zero}
\label{subsec:reduction}

In this subsection we assume that $\Coeff = \bC$.
(We only do so because the literature on types currently makes this assumption.)
By \cite[Theorem~7.5]{Kim-Yu} and \cite{MR4357723}
the pair $(K_{x_{0}}, \rho_{x_{0}})$ is an $\fS(\Sigma)$-type for 
a finite subset $\fS(\Sigma)$ of the inertial equivalence classes $\IEC(G)$ for $G$
and the pair $(K^{0}_{x_{0}}, \rho^{0}_{x_{0}})$ is an $\fSz(\Sigma)$-type for
a finite subset $\fSz(\Sigma)$ of the inertial equivalence classes $\IEC(G^0)$ for $G^0$.
We denote by $\Rep^{\fS(\Sigma)}(G(F))$ and $\Rep^{\fSz(\Sigma)}(G^0(F))$ the corresponding union of Bernstein blocks that were recalled in \cite[Section \ref{HAI-subsec:application}]{HAI}.

\begin{corollary}
\label{cor:equiv-blocksKimYucase} 
We have an equivalence of categories $\Rep^{\fS(\Sigma)}(G(F)) \isoarrow \Rep^{\fSz(\Sigma)}(G^0(F))$. 
\end{corollary}

\begin{proof}
Apply \cite[Theorem~\ref{HAI-thm:equiv-blocks}]{HAI} to the pairs
$(K^{0}_{x_0}, \rho^{0}_{x_0})$ and $(K_{x_{0}}, \rho_{x_{0}})$.
\end{proof}

Combining Corollary~\ref{cor:equiv-blocksKimYucase} with the exhaustion result in \cite{Fi-exhaustion}, we obtain that if $p$ is large enough, every Bernstein block is equivalent to a depth-zero block:

\begin{theorem}
\label{thm:reduction}
We assume that $p$ does not divide the order of the absolute Weyl group of $G$.
Then for every inertial equivalence class $\fs \in \IEC(G)$,
there exists a tamely ramified twisted Levi subgroup $G^{0}$ of $G$ and $\fsz \in \IEC(G^{0})$ such that the full subcategory $\Rep^\fsz(G^0(F))$ consists of depth-zero representations, and we have an equivalence of categories $\Rep^{\fs}(G(F)) \isoarrow \Rep^{\fsz}(G^0(F))$. 
\end{theorem}
\begin{proof}
According to \cite[Theorem~7.12]{Fi-exhaustion} and the assumptions, for any $\fs \in \IEC(G)$, there exists a $G$-datum $\Sigma$ such that $\{\fs\} = \fS(\Sigma)$.
Then the theorem follows from
Corollary~\ref{cor:equiv-blocksKimYucase}.
\end{proof}

\begin{remark}
According to \cite[Theorem \ref{HAI-thm:plancherel}]{HAI},
when restricted to irreducible objects,
the equivalences of categories in
Corollary~\ref{cor:equiv-blocksKimYucase}
and Theorem \ref{thm:reduction}
preserve temperedness, and
preserve the Plancherel measure on the tempered dual up to an explicit constant
factor.
\end{remark}


\appendix
\section{Appendix}
\subsection{Decomposition of a symplectic space over $\protect\bF_p$:
	generalization of an argument of Yu}
\label{appendixYufortwopoints}
In this appendix, we generalize the arguments in \cite[Section~12 and 13]{Yu} to prove Lemma~\ref{lemmatotallyisotropic}.
We use the same notation as in Section~\ref{subsection: relative case}, i.e.\ $G'$ denotes a twisted Levi subgroup of a connected reductive group $G$ defined over $F$ that splits over a tamely ramified field extension of $F$, $r$ is a positive real number, $x, y \in \cB(G', F)$, and we write
$
\Vyx
= \left(
J_{x} \cap J_{y}
\right) \cdot J_{y, +}/J_{y, +}
$
and
$
\Vyxp  =
\left(
J_{x, +} \cap J_{y}
\right) \cdot J_{y, +}/J_{y, +}
$.
We explain another description of the spaces $\Vyx$ and $\Vyxp$.
Let $T$ be a maximal torus of $G'$ such that the splitting field $E$ of $T$ is tamely ramified over $F$ and such that $x, y \in \cA(G', T, E)$.
According to the discussion in the beginning of \cite[Section~2]{Yu} and the fact that any two points of $\cB(G', F)$ are contained in an apartment of $\cB(G', F)$, such a torus exists.
We define
\[
\begin{cases}
\Phi_{1} &= \{
\alpha \in \Phi(G, T) \mid \alpha(y - x) < 0
\}, \\
\Phi_{2} &= \{
\alpha \in \Phi(G, T) \mid \alpha(y - x) = 0
\} \cup \{0\}, \\
\Phi_{3} &= \{
\alpha \in \Phi(G, T) \mid \alpha(y - x) > 0
\}.
\end{cases}
\]
For $i = 1, 2, 3$, we define
$
\Phi'_{i} = \Phi_{i} \cap \left(
\Phi(G', T) \cup \{0\}
\right)
$
and
$
\Phi''_{i} = \Phi_{i} \smallsetminus \Phi'_{i}
$.
We define the function
\[
f_{i} \colon \Phi(G, T) \cup \{0\} \longrightarrow \widetilde{\mathbb{R}}
\quad 
\text{ by }
\quad
f_{i}(\alpha) =
\begin{cases}
r & (\alpha \in \Phi'_{i}), \\
\frac{r}{2} & (\alpha \in \Phi''_{i}), \\
\infty & (\text{otherwise}).
\end{cases}
\]
As explained in \cite[Section~13]{Yu}, the functions $f_{1}$, $f_{2}$, and $f_{3}$ are concave.
For $i \in \{1, 2, 3\}$, we denote the subgroup
\(
 G(F)_{y,f_i}
\)
by $J_{y,i}$, and the image of $J_{y, i}$ in $J_{y}/J_{y, +}$ by $V_{i}$.
According to (the same arguments used to prove)
\cite[Lemma~13.6]{Yu} (applied to the more general case where $x$ and $y$ are not necessarily in the same $G'(F)$-orbit), the spaces $V_{1}$ and $V_{3}$ are totally isotropic subspaces of $J_{y}/J_{y, +}$ and orthogonal to $V_{2}$ with respect to $\langle \phantom{x},\phantom{x} \rangle_{y}$, and we have \spacingatend{}
\[
\begin{cases}
J_{y}/J_{y, +} &= V_{1} \oplus V_{2} \oplus V_{3}, \\
\Vyx &= V_{1} \oplus V_{2}, \\
\Vyxp  &= V_{1}.
\end{cases}
\]
	\begin{subequations}
		For a subspace $W$ of $J_{y}/J_{y, +}$, let $W^{\perp}$ denote the orthogonal complement of $W$ in $J_{y}/J_{y, +}$ with respect to $\langle \phantom{x},\phantom{x} \rangle_{y}$.
		Then we have \spacingatend{deal at the end with vertical spacing, also improve numbering of equations}
		\begin{align}
		\label{orthogonalxompinclusion}
		\left(
		\Vyxp  
		\right)^{\perp}
		= V_{1}^{\perp} 
		\subset V_{1} \oplus V_{2} = \Vyx.
		\end{align}
		On the other hand, we claim that 
		\begin{equation}
		\label{equation:dimV1=dimV3}
		\dim_{\mathbb{F}_{p}} (V_{1}) + \dim_{\mathbb{F}_{p}}(V_{1} \oplus V_{2}) = \dim_{\mathbb{F}_{p}} (J_{y}/J_{y, +}).
		\end{equation}
		This can be proven by the arguments in the proof of \cite[Lemma~12.8]{Yu}.
		More precisely, 
		since $J_{y}/J_{y, +} = V_{1} \oplus V_{2} \oplus V_{3}$, it suffices to show that $\dim_{\mathbb{F}_{p}} (V_{1}) = \dim_{\mathbb{F}_{p}} (V_{3})$. Since $V_1 \oplus V_2 \oplus V_3$ is non-degenerate and $V_{2}^{\perp} \supset V_{1} \oplus V_{3}$, we obtain that $V_2$ is non-degenerate and $V_{2}^{\perp} = V_{1} \oplus V_{3}$.
		Since $V_{1}$ and $V_{3}$ are totally isotropic subspaces of the non-degenerate space $V_{2}^{\perp} = V_{1} \oplus V_{3}$, we have $\dim_{\mathbb{F}_{p}} (V_{1}), \dim_{\mathbb{F}_{p}} (V_{3}) \le \frac{1}{2}\dim_{\mathbb{F}_{p}} (V_{1} \oplus V_{3})$.
		Thus, we obtain that $\dim_{\mathbb{F}_{p}} (V_{1}) = \dim_{\mathbb{F}_{p}} (V_{3}) = \frac{1}{2}\dim_{\mathbb{F}_{p}} (V_{1} \oplus V_{3})$.
		
		Equation \eqref{equation:dimV1=dimV3} implies that
		$
		\dim_{\mathbb{F}_{p}}(V_{1}^{\perp}) = \dim_{\mathbb{F}_{p}}(V_{1} \oplus V_{2})
		$.
		Combining this with \eqref{orthogonalxompinclusion}, we obtain $V_1^\perp=V_1 \oplus V_2$ and see that Lemma \ref{lemmatotallyisotropic} holds true.
	\end{subequations}

\subsection{The quadratic twist is necessary}
\label{subsec:quadratictwistisnecessary}

In this appendix, we give an example to show that our main theorem, Theorem~\ref{heckealgebraisomforKim-Yutype}, would not be true in general if we replaced $\rho_{x_0}$ by $\rho^{\nt}_{x_0}$, i.e., if we omitted the quadratic twist in the construction of $\rho_{x_0}$.
Assume $\ell=0$ and recall that $p \neq 2$.
Let $G = \Symp_{4}$ over $F$ corresponding to the symplectic pairing given by
\[
J = 
\begin{pmatrix}
	0 & 0 & 0 & 1 \\
	0 & 0 & -1 & 0 \\
	0 & 1 & 0 & 0 \\
	-1 & 0 & 0 & 0
\end{pmatrix}.
\]
Let $T$ be the maximal torus of $G$ defined as
\[
T = \left\{
\begin{pmatrix}
	t_{1} & 0 & 0 & 0 \\
	0 & t_{2} & 0 & 0 \\
	0 & 0 & t_{2}^{-1} & 0 \\
	0 & 0 & 0 & t_{1}^{-1}
\end{pmatrix}
\right\}.
\]
Let $o \in \cA(G, T, F)$ be the point such that
$
G(F)_{o, 0} = \Symp_{4}(\cO_{F})
$.
We identify $\cA(G, T, F)$ with $\mathbb{R}^{2}$ via the bijection
\(
\mathbb{R}^{2} \simeq \cA(G, T, F)
\)
defined by
$
(x_{1}, x_{2}) \mapsto o + x_{1} \alpha_{1}^{\vee} + x_{2} \alpha_{2}^{\vee}
$,
where $\alpha_{1}^{\vee}$ and $\alpha_{2}^{\vee}$ are cocharacters of $T$ defined as
\[
\alpha_{1}^{\vee}(t) = 
\begin{pmatrix}
	t & 0 & 0 & 0 \\
	0 & 1 & 0 & 0 \\
	0 & 0 & 1 & 0 \\
	0 & 0 & 0 & t^{-1}
\end{pmatrix}
\qquad
\text{and}
\qquad
\alpha_{2}^{\vee}(t) = 
\begin{pmatrix}
	1 & 0 & 0 & 0 \\
	0 & t & 0 & 0 \\
	0 & 0 & t^{-1} & 0 \\
	0 & 0 & 0 & 1
\end{pmatrix},
\]
respectively.
We define points $x_{0}, x'_{0} \in \cA(G, T, F) = \mathbb{R}^{2}$ as $x_{0} = \left(- \frac{1}{4}, \frac{1}{16}\right)$ and $x'_{0} = \left(- \frac{1}{4}, \frac{1}{4}\right)$.

\label{page:overlineF}
We fix a uniformizer $\pi_{F}$ of $F$, and let $\varpi_{F}$ be an element of $\overline{F}$ such that $\varpi_{F}^{2} = \pi_{F}$.
We also fix an element $\sqrt{2 \varpi_{F}}$ in $\overline{F}$ such that ${\sqrt{2 \varpi_{F}}}^{2} = 2 \varpi_{F}$.
Let $E = F(\sqrt{2 \varpi_{F}})$.
We define an element $X \in \Lie^{*}(G)(F)$ by
\[
\Lie(G)(F) \ni (A_{i, j}) \stackrel{X}{\longmapsto} \pi_{F}^{-1} \cdot A_{1, 4} + A_{4, 1}.
\]
We define a twisted Levi subgroup $G^{0}$ of $G$ as the centralizer of $X$ in $G$.
Then we have
\[
G^{0}
=
\left\{
\begin{pmatrix}
	a & 0_{1 \times 2} & b \\
	0_{2 \times 1} & \SL_{2} & 0_{2 \times 1} \\
	b \pi_{F}^{-1} & 0_{1 \times 2} & a
\end{pmatrix} \;\middle|\; a^{2} - b^{2} \pi_{F}^{-1} = 1
\right\} 
\simeq
U(1) \times \SL_{2}.
\]
We also define a Levi subgroup $M^{0}$ of $G^{0}$ as
\[
M^{0} = \left\{
\begin{pmatrix}
	a & 0 & 0 & b \\
	0 & t & 0 & 0 \\
	0 & 0 & t^{-1} & 0 \\
	b \pi_{F}^{-1} & 0 & 0 & a
\end{pmatrix} \;\middle|\; a^{2} - b^{2} \pi_{F}^{-1} = 1
\right\}.
\]
We note that $M^{0}$ is a maximal torus of $G^{0}$ and $G$.
Moreover, we have
$
M^{0} = g T g^{-1}
$,
where
\[
g = \begin{pmatrix}
	\frac{
		\varpi_{F}
	}{
		\sqrt{2 \varpi_{F}}
	} & 0 & 0 & \frac{
		- \varpi_{F}
	}{
		\sqrt{2 \varpi_{F}}
	} \\
	0 & 1 & 0 & 0 \\
	0 & 0 & 1 & 0 \\
	\frac{
		1
	}{
		\sqrt{2 \varpi_{F}}
	} & 0 & 0 & \frac{
		1
	}{
		\sqrt{2 \varpi_{F}}
	}
\end{pmatrix}.
\]
Let $M$ denote the centralizer of $A_{M^{0}}$ in $G$.
Then we have
\[
A_{M^{0}} = 
\left\{
\begin{pmatrix}
	1 & 0 & 0 & 0 \\
	0 & t & 0 & 0 \\
	0 & 0 & t^{-1} & 0 \\
	0 & 0 & 0 & 1
\end{pmatrix}
\right\}
\]
and
\[
M = \left\{
\begin{pmatrix}
	a & 0 & 0 & b \\
	0 & t & 0 & 0 \\
	0 & 0 & t^{-1} & 0 \\
	c & 0 & 0 & d
\end{pmatrix}  \;\middle|\;
\begin{pmatrix}
	a & b \\
	c & d
\end{pmatrix} \in \SL_{2}
\right\}.
\]
We define characters $\alpha_{1}$ and $\alpha_{2}$ of $T$ as
\[
\alpha_{1} \left(
\begin{pmatrix}
	t_{1} & 0 & 0 & 0 \\
	0 & t_{2} & 0 & 0 \\
	0 & 0 & t_{2}^{-1} & 0 \\
	0 & 0 & 0 & t_{1}^{-1}
\end{pmatrix}
\right) = t_{1}^{2}
\qquad
\text{and}
\qquad
\alpha_{2} \left(
\begin{pmatrix}
	t_{1} & 0 & 0 & 0 \\
	0 & t_{2} & 0 & 0 \\
	0 & 0 & t_{2}^{-1} & 0 \\
	0 & 0 & 0 & t_{1}^{-1}
\end{pmatrix}
\right) = t_{2}^{2}.
\]
Then we have
\[
\Phi(G, T) = \left\{
\pm \alpha_{1}, \pm \alpha_{2}, \pm \frac{\alpha_{1} + \alpha_{2}}{2}, \pm \frac{\alpha_{1} - \alpha_{2}}{2}
\right\}
\qquad
\text{and}
\qquad
\Phi(M, T) = \left\{
\pm \alpha_{1}
\right\}.
\]

We fix a commutative diagram $\{ \iota \}$
\[
\xymatrix{
	\cB(G^{0}, F) \ar[r] & \cB(G, F)
	\\
	\cB(M^{0}, F) \ar[u] \ar[r] & \cB(M, F) \ar[u]
}
\]
of admissible embeddings of buildings and identify a point in $\cB(M^{0}, F)$ with its images via the embeddings $\{ \iota \}$.
\begin{lemma}
	\label{lemmaXisgeneric}
	The element $X$ is $G$-generic of depth $- \frac{1}{2}$ in the sense of \cite[Definition~3.5.2]{Fintzen-IHES} for the pair $G^{0} \subset G$, and the points $x_{0}$ and $x'_{0}$ are contained in (the image of) $\cB(M^{0}, F)$.
\end{lemma}
\begin{proof}
	The first claim follows from the same argument as the proof of \cite[Lemma~4.1]{MR4357723}.
	The proof of \cite[Lemma~4.1]{MR4357723} also implies that
	\[
	x'_{0} = g \left(
	0, \tfrac{1}{4}
	\right) \in \cA(G, {g T g^{-1}}, E) = \cB(M^{0}, E).
	\]
	Thus, we conclude that
	$
	x'_{0} \in \cB(G, F) \cap \cB(M^{0}, E) = \cB(M^{0}, F)
	$.
	Then we also have
	\[
	x_{0} = x'_{0} - \tfrac{3}{16}\alpha_{2}^{\vee}
	\in x'_{0} + \left(
	X_{*}(A_{M^{0}}) \otimes_{\mathbb{Z}} \mathbb{R}
	\right) 
	\subset \cB(M^{0}, F).
	\qedhere
	\]
\end{proof}
\begin{remark}
	Since $x'_{0} = g \left(
	0, \frac{1}{4}
	\right)$ and the element $g \in M(E)$ acts trivially on $X_{*}(A_{M^{0}}) \otimes_{\mathbb{Z}} \mathbb{R}$, we also obtain that $x_{0} = g \left(
	0, \frac{1}{16}
	\right)$.
\end{remark}
We define a character $\phi$ of $\Lie(G^{0})(F)_{x_{0}, \frac{1}{2}}$ by
$
\phi(Y) = \Psi(X(Y))
$.
Since $X$ is $G$-generic of depth $- \frac{1}{2}$, the character $\phi$ is trivial on $\Lie(G^{0})(F)_{x_{0}, \frac{1}{2}+}$.
Hence, by using the Moy--Prasad isomorphism
\[
G^{0}(F)_{x_{0}, \frac{1}{2}} / G^{0}(F)_{x_{0}, \frac{1}{2}+} \simeq \Lie(G^{0})(F)_{x_{0}, \frac{1}{2}} / \Lie(G^{0})(F)_{x_{0}, \frac{1}{2}+},
\]
we can also regard $\phi$ as a character of $G^{0}(F)_{x_{0}, \frac{1}{2}}$ that is trivial on $G^{0}(F)_{x_{0}, \frac{1}{2}+}$.
Moreover, the definition of $X$ implies that the character $\phi$ is trivial on the subgroup
\[
\left\{
\begin{pmatrix}
	1 & 0_{1 \times 2} & 0 \\
	0_{2 \times 1} & \SL_{2}(F) & 0_{2 \times 1} \\
	0 & 0_{1 \times 2} & 1
\end{pmatrix}
\right\} \cap G^{0}(F)_{x_{0}, \frac{1}{2}} 
\]
of $G^{0}(F)_{x_{0}, \frac{1}{2}}$.
Since $U(1)$ is abelian, we can extend $\phi$ to a character of 
$
G^{0}(F) \simeq \left(
U(1) \times \SL_{2}
\right)(F)
$
that is trivial on the subgroup
\[
\left\{
\begin{pmatrix}
	1 & 0_{1 \times 2} & 0 \\
	0_{2 \times 1} & \SL_{2}(F) & 0_{2 \times 1} \\
	0 & 0_{1 \times 2} & 1
\end{pmatrix}
\right\}
\]
of $G^{0}(F)$.
We use the same notation $\phi$ for this extension.
Since $X$ is $G$-generic of depth $- \frac{1}{2}$, the character $\phi$ is $G$-generic relative to $x_{0}$ of depth $r \coloneqq  \frac{1}{2}$ in the sense of \cite[Definition~3.5.2]{Fintzen-IHES}.
\begin{lemma}
	\label{lemmagenericityoftheembedding}
	The diagram of embeddings $\{\iota\}$ is $\left(
	0, \frac{r}{2}
	\right) = \left(
	0, \frac{1}{4}
	\right)$-generic relative to $x_{0}$ in the sense of \cite[3.5~Definition]{Kim-Yu}.
\end{lemma}
\begin{proof}
	We will prove that for all $\alpha \in \Phi(G, T) \smallsetminus \Phi(M, T)$ and $t \in \{0, 1/4\}$, we have $U_{\alpha}(E)_{x_{0},t} = U_{\alpha}(E)_{x_{0}, t+}$.
	Noting that 
	\[
	U_{\alpha}(E)_{x_{0},t} = U_{\alpha}(E)_{o, t - \left \langle
		\alpha, - \tfrac{1}{4} \alpha_{1}^{\vee} + \tfrac{1}{16} \alpha_{2}^{\vee}
		\right \rangle},
	\]
	it suffices to show that
	$
	\left \langle
	\alpha, - \tfrac{1}{4} \alpha_{1}^{\vee} + \tfrac{1}{16} \alpha_{2}^{\vee}
	\right \rangle \not \in \tfrac{1}{4} \mathbb{Z}
	$ for all $\alpha \in \Phi(G, T) \smallsetminus \Phi(M, T)$.
	This follows from the calculations 
	\[
	\begin{cases}
		\left \langle
		\pm \alpha_{2}, - \frac{1}{4} \alpha_{1}^{\vee} + \frac{1}{16} \alpha_{2}^{\vee}
		\right \rangle &= \pm \frac{1}{8}, \\
		\left \langle
		\pm \frac{\alpha_{1} + \alpha_{2}}{2}, - \frac{1}{4} \alpha_{1}^{\vee} + \frac{1}{16} \alpha_{2}^{\vee}
		\right \rangle &= \mp \frac{3}{16}, \\
		\left \langle
		\pm \frac{\alpha_{1} - \alpha_{2}}{2}, - \frac{1}{4} \alpha_{1}^{\vee} + \frac{1}{16} \alpha_{2}^{\vee}
		\right \rangle &= \mp \frac{5}{16}.
	\end{cases}
	\qedhere
	\]
\end{proof}
We define a compact, open subgroup $K^{0}_{x_{0}}$ of $G^{0}(F)$ as $K^{0}_{x_{0}} = G^{0}(F)_{x_{0}, 0}$, and let $\rho^{0}_{x_{0}}$ denote the trivial representation of $K^{0}_{x_{0}}$.
We also write $\rho_{M^{0}}$
for the trivial representation of $M^{0}(F)_{x_{0}, 0}$.
Let $(K_{x_{0}}, \rho^{\nt}_{x_{0}})$ denote the pair constructed from the $G$-datum
\[
\bigl(
( G^{0} \subset G, M^0),
( \tfrac{1}{2}, \tfrac{1}{2}),
(x_{0}, \{\iota\}),
( M^0(F)_{x_{0}, 0}, \rho_{M^0}),
(\phi, 1)
\bigr)
\]
in Section \ref{subsec:construction}.
\begin{proposition}
\label{prop:twist-necessary}
	There is no support-preserving algebra isomorphism 
	\[
	\cH\left(
	G^{0}(F), \rho^{0}_{x_{0}}
	\right) \isoarrow \cH\left(
	G(F), \rho^{\nt}_{x_{0}}
	\right).
	\]
\end{proposition}
\begin{proof}
	Since $(K_{x_{0}}, \rho^{\nt}_{x_{0}})$ agrees with the pair $(K_{x_{0}}, \rho_{x_{0}})$ obtained by replacing the above $G$-datum by the following twisted one
	\[
	\bigl(
	( G^{0} \subset G, M^0),
	( \tfrac{1}{2}, \tfrac{1}{2} ),
	(x_{0}, \{\iota\}),
	\bigl(
	M^0(F)_{x_{0}, 0}, \epsilon^{G/G^{0}}_{x_{0}}\restriction_{M^0(F)_{x_{0}, 0}}
	\bigr),
	(\phi, 1)
	\bigr)
	\]
	in the construction in Section \ref{subsec:construction}, 
	according to Theorem~\ref{heckealgebraisomforKim-Yutype}, we have a support-preserving algebra isomorphism
	\[
	\cH( G^{0}(F), \epsilon^{G/G^{0}}_{x_{0}})
	\isoarrow
	\cH( G(F), \rho^{\nt}_{x_{0}}).
	\]
	Hence, it suffices to show that there is no support-preserving algebra isomorphism 
	\[
	\cH ( G^{0}(F), \rho^{0}_{x_{0}})
	\isoarrow
	\cH( G^{0}(F), \epsilon^{G/G^{0}}_{x_{0}}).
	\]
	
	Let 
	\[
	s = 
	\begin{pmatrix}
		1 & 0 & 0 & 0 \\
		0 & 0 & 1 & 0 \\
		0 & -1 & 0 & 0 \\
		0 & 0 & 0 & 1
	\end{pmatrix}.
	\]
	Then we have
	\[
	K^{0}_{x_{0}}= \left\{
	\begin{pmatrix}
		a & 0 & 0 & b \\
		0 & \cO_{F}^{\times} & \cO_{F} & 0 \\
		0 & \fp_{F} & \cO_{F}^{\times} & 0 \\
		b \pi_{F}^{-1} & 0 & 0 & a
	\end{pmatrix}
	\right\}
	\cap G^{0}(F)
	\quad
	\text{and}
	\quad
	s K^{0}_{x_{0}} s^{-1} = \left\{
	\begin{pmatrix}
		a & 0 & 0 & b \\
		0 & \cO_{F}^{\times} & \fp_{F}& 0 \\
		0 & \cO_{F} & \cO_{F}^{\times} & 0 \\
		b \pi_{F}^{-1} & 0 & 0 & a
	\end{pmatrix}
	\right\}
	\cap G^{0}(F).
	\]
	Hence, we can take a set of representatives for 
	\(
	K^{0}_{x_{0}}/ \left(
	K^{0}_{x_{0}}\cap {s K^{0}_{x_{0}} s^{-1}}
	\right)
	\)
	as
	\(
	\left\{
	u(x) \mid x \in \ff 
	\right\}
	\),
	where
	\[
	u(x) = \begin{pmatrix}
		1 & 0 & 0 & 0 \\
		0 & 1 & x & 0 \\
		0 & 0 & 1 & 0 \\
		0 & 0 & 0 & 1 
	\end{pmatrix}.
	\]
	
	Suppose that there is a support-preserving algebra isomorphism 
	$
	\cH( G^{0}(F), \rho^{0}_{x_{0}})
	\isoarrow
	\cH( G^{0}(F), \epsilon^{G/G^{0}}_{x_{0}})
	$.
	Then we can take
	$\varphi_{1} \in \cH ( G^{0}(F), \rho^{0}_{x_{0}})$
	and
	$\varphi_{2} \in \cH( G^{0}(F), \epsilon^{G/G^{0}}_{x_{0}} )$,
	both supported in $K^0_{x_0} s K^0_{x_0}$,
	such that $\varphi_{1}$ and $\varphi_{2}$ satisfy the same quadratic relation.
	Let $\varphi = \varphi_{1}$ or $\varphi_{2}$.
	We can calculate the convolution product $( \varphi * \varphi )(s)$ as
	\begin{eqnarray*}
	\left(
	\varphi * \varphi
	\right)(s) 
	= \sum_{h \in K^{0}_{x_{0}} s K^{0}_{x_{0}} / K^{0}_{x_{0}}} \varphi(h) \cdot \varphi(h^{-1}s) 
	&=& \sum_{k \in K^{0}_{x_{0}} / \left(
		K^{0}_{x_{0}} \cap {s K^{0}_{x_{0}} s^{-1}}
		\right)
	} \varphi(ks) \cdot \varphi(s^{-1} k^{-1} s) \\
	&=& \sum_{x \in \ff}  \varphi(u(x) s) \cdot \varphi(s^{-1} u(-x) s).
	\end{eqnarray*} 
	For $x \in \cO_{F}$, we have
	\begin{align*}
		s^{-1} u(-x) s 
		= \begin{pmatrix}
			1 & 0 & 0 & 0 \\
			0 & 1 & 0 & 0 \\
			0 & x & 1 & 0 \\
			0 & 0 & 0 & 1 
		\end{pmatrix}.
	\end{align*}
	A standard calculation implies that $s^{-1} u(x)^{-1} s \in K^{0}_{x_{0}} s K^{0}_{x_{0}}$ if and only if $x \in \cO_{F}^{\times}$, and in this case, we have
	\[
	s^{-1} u(x)^{-1} s = \begin{pmatrix}
		1 & 0 & 0 & 0 \\
		0 & -x^{-1} & -1 & 0 \\
		0 & 0 & -x & 0 \\
		0 & 0 & 0 & 1 
	\end{pmatrix}
	\cdot
	s
	\cdot
	\begin{pmatrix}
		1 & 0 & 0 & 0 \\
		0 & 1 & x^{-1} & 0 \\
		0 & 0 & 1 & 0 \\
		0 & 0 & 0 & 1 
	\end{pmatrix} 
	= \alpha_{2}^{\vee}(- x^{-1}) \cdot u(x) \cdot s \cdot u(x^{-1}).
	\]
	Since $\epsilon^{G/G^{0}}_{x_{0}}$ is a sign character, it is trivial on the pro-$p$-group
	\(
	\left\{
	u(x) \mid x \in \cO_{F}
	\right\} \simeq \cO_{F}
	\).
	Hence, according to Lemma~\ref{epsilonatalpha2vee} below, we have
	\begin{align*}
		\varphi_{2}\left(
		\alpha_{2}^{\vee}(- x^{-1}) \cdot u(x) \cdot s \cdot u(x^{-1})
		\right) &= 
		\epsilon^{G/G^{0}}_{x_{0}} \left(
		\alpha_{2}^{\vee}(- x^{-1}) \cdot u(x)
		\right) \cdot \varphi_{2}(s) \cdot \epsilon^{G/G^{0}}_{x_{0}} \left(
		u(x^{-1})
		\right) \\
		&= \sgn_{\ff}(- x^{-1} \bmod \fp_{F}) \cdot \varphi_{2}(s)
	\end{align*}
	for all $x \in \cO_{F}^{\times}$.
	Thus, we obtain that
	\begin{align*}
		\left(
		\varphi * \varphi
		\right)(s) 
		&= \sum_{x \in \ff^{\times}} \varphi(u(x) s) \cdot \varphi(s^{-1} u(-x) s) \\
		&= \sum_{x \in \ff^{\times}} \varphi(u(x) s) \cdot \varphi \left(
		\alpha_{2}^{\vee}(- x^{-1}) \cdot u(x) \cdot s \cdot u(x^{-1})
		\right) \\
		&=
		\begin{cases}
			\sum_{x \in \ff^{\times}} \varphi_{1}(s) \cdot \varphi_{1}(s) & (\varphi = \varphi_{1}), \\
			\sum_{x \in \ff^{\times}} \varphi_{2}(s) \cdot \sgn_{\ff}(- x^{-1}) \cdot \varphi_{2}(s) & (\varphi = \varphi_{2})
		\end{cases} \\
		&= 
		\begin{cases}
			\varphi_{1}(s)^{2} \sum_{x \in \ff^{\times}} 1 & (\varphi = \varphi_{1}), \\
			\varphi_{2}(s)^{2} \sum_{x \in \ff  ^{\times}} \sgn_{\ff  }(x) & (\varphi = \varphi_{2})
		\end{cases} \\
		&=
		\begin{cases}
			(q_{F} - 1) \cdot \varphi_{1}(s)^{2} & (\varphi = \varphi_{1}), \\
			0 & (\varphi = \varphi_{2}).
		\end{cases}
	\end{align*}
	In particular, we obtain that 
	\[
	\left(
	\varphi_{1} * \varphi_{1}
	\right)(s) \neq 0
	\qquad
	\text{and}
	\qquad
	\left(
	\varphi_{2} * \varphi_{2}
	\right)(s) = 0,
	\]
	which contradicts the fact that $\varphi_{1}$ and $\varphi_{2}$ satisfy the same quadratic relation.
\end{proof}

Now we take care of the last piece of unfinished business from the proof of
Proposition \ref{prop:twist-necessary}.

\begin{lemma}
	\label{epsilonatalpha2vee}
	Let $t \in \cO_{F}^{\times}$.
	Then we have
	\[
	\epsilon^{G/G^{0}}_{x_{0}}(\alpha_{2}^{\vee}(t)) = \sgn_{\ff  }(t \bmod \fp_{F}).
	\]
\end{lemma}
\begin{proof}
	We define characters $\beta_{1}$ and $\beta_{2}$ of $M^{0}$ as
	\[
	\beta_{1} \left(
	\begin{pmatrix}
		a & 0 & 0 & b \\
		0 & t & 0 & 0 \\
		0 & 0 & t^{-1} & 0 \\
		b \pi_{F}^{-1} & 0 & 0 & a
	\end{pmatrix}
	\right) = (a + b \varpi_{F}^{-1})^{2}
	\qquad
	\text{and}
	\qquad
	\beta_{2} \left(
	\begin{pmatrix}
		a & 0 & 0 & b \\
		0 & t & 0 & 0 \\
		0 & 0 & t^{-1} & 0 \\
		b \pi_{F}^{-1} & 0 & 0 & a
	\end{pmatrix}
	\right) = t^{2}.
	\]
	Then we have
	\[
	\Phi(G, M^{0}, E) = \left\{
	\pm \beta_{1}, \pm \beta_{2}, \pm \frac{\beta_{1} + \beta_{2}}{2}, \pm \frac{\beta_{1} - \beta_{2}}{2}
	\right\}
	\qquad
	\text{and}
	\qquad
	\Phi(G^{0}, M^{0}, E) = \left\{
	\pm \beta_{2}
	\right\}.
	\]
	We define
	$
	\Phi(G/G^{0}, M^{0})  = \Phi(G, M^{0}, E) \smallsetminus \Phi(G^{0}, M^{0}, E)
	$,
	and we denote by $\Phi(G/G^{0}, M^{0})_{\asym}$, $\Phi(G/G^{0}, M^{0})_{\sym, \unr}$, and $\Phi(G/G^{0}, M^{0})_{\sym, \ram}$ the set of roots in $\Phi(G/G^{0}, M^{0})$ that are asymmetric, unramified symmetric, and ramified symmetric in the sense of \cite[Section~2]{FKS}, respectively.
	Then we have
	\[
	\Phi(G/G^{0}, M^{0})_{\asym} = \left\{
	\pm \frac{\beta_{1} + \beta_{2}}{2}, \pm \frac{\beta_{1} - \beta_{2}}{2}
	\right\},
	\]
	\[
	\Phi(G/G^{0}, M^{0})_{\sym, \unr} = \emptyset,
	\]
	and
	\[
	\Phi(G/G^{0}, M^{0})_{\sym, \ram} = \left\{
	\pm \beta_{1}
	\right\}.
	\]
	Since $x_{0} = g \left(
	0, \frac{1}{16}
	\right)$ and 
	$
	\beta_{i}(g t g^{-1}) = \alpha_{i} (t)
	$
	for $i= 1,2$ and for all $t \in T$, we obtain that
	\begin{align*}
		\left \langle
		\pm \tfrac{\beta_{1} + \beta_{2}}{2}, - \tfrac{1}{4} \alpha_{1}^{\vee} + \tfrac{1}{16} \alpha_{2}^{\vee}
		\right \rangle 
		&=  \left \langle
		\pm \tfrac{\beta_{1} + \beta_{2}}{2}, g \cdot \tfrac{1}{16} \alpha_{2}^{\vee}
		\right \rangle \\
		&= \left \langle
		\pm \tfrac{\alpha_{1} + \alpha_{2}}{2}, \tfrac{1}{16} \alpha_{2}^{\vee}
		\right \rangle \\
		&= \pm \tfrac{1}{16} \\
		&\not \in \tfrac{1}{4}\mathbb{Z}.
	\end{align*}
	Similarly, we have
	$
	\left \langle
	\pm \tfrac{\beta_{1} - \beta_{2}}{2}, - \tfrac{1}{4} \alpha_{1}^{\vee} + \tfrac{1}{8} \alpha_{2}^{\vee}
	\right \rangle 
	= \mp \tfrac{1}{16} \not \in \tfrac{1}{4}\mathbb{Z}
	$.
	Thus, we conclude that
	\(
	\frac{r}{2} =
	\frac{1}{4} \not \in \ord_{x_{0}}(\beta)
	\)
	for all $\beta \in \Phi(G/G^{0}, M^{0})_{\asym}$, where $\ord_{x_{0}}(\beta)$ denotes the set defined in \cite[Section~3]{FKS}.
	We also note that we have $\beta_{1}(\alpha_{2}^{\vee}(t)) = 1$ and the restriction of any element in $\Phi(G/G^{0}, M^{0})_{\asym}$ to the center $Z(G^{0}) / \{\pm 1\}$ of $G^{0} / \{\pm1\}$ is ramified symmetric.
	Then according to \cite[Definition~3.1, Theorem~3.4]{FKS}, we have
	\[
	\epsilon^{G/G^{0}}_{x_{0}}(\alpha_{2}^{\vee}(t)) = \sgn_{\ff  }(\beta(\alpha_{2}^{\vee}(t)) \bmod \fp_{F})
	= \sgn_{\ff  }(t \bmod \fp_{F}),
	\]
	where $\beta$ denotes any element in $\Phi(G/G^{0}, M^{0})_{\asym}$.
\end{proof}
\begin{remark}
	\label{remarkexampleofepsilonGneqepsilonM}
	Since the image of $\alpha_{2}^{\vee}$ is contained in the center of $M$, we have
	$
	\epsilon^{M/M^{0}}_{x_{0}}(\alpha_{2}^{\vee}(t)) = 1
	$
	for all $t \in \cO_{F}^{\times}$.
	In particular, we obtain that
	$
	\epsilon^{G/G^{0}}_{x_{0}} \restriction_{M^{0}(F)_{x_{0}, 0}} \neq \epsilon^{M/M^{0}}_{x_{0}}
	$.
\end{remark}

\printindex{notation-ky}{Selected notation}


\addcontentsline{toc}{section}{References}
\bibliographystyle{halpha-abbrv}
\bibliography{ourbib}

\end{document}